\documentclass[11pt]{article}

\usepackage{geometry}
\usepackage{latexsym}
\usepackage{mathrsfs}
\usepackage{amsfonts}
\usepackage{amssymb}
\usepackage{subfigure}    %for parallel figures
\usepackage{graphicx}
\usepackage{epstopdf}
\usepackage{float}
\usepackage{graphicx}
\usepackage{multirow}
\usepackage{bm}
\usepackage{amsmath}
\usepackage{amsthm}
\usepackage{color}
\usepackage[subnum]{cases}
\usepackage{rotating}
\usepackage{enumitem}
\usepackage{accents}
\usepackage{algorithm}
\usepackage{algorithmic}
\usepackage[title]{appendix}
\usepackage{mathtools}
\usepackage{caption}
\usepackage{rotating}
\usepackage{url}

\makeatletter
\@addtoreset{equation}{section}
\makeatother

\makeatletter
\def \wideubar{\underaccent{{\cc@style\underline{\mskip15mu}}}}
\def \widebar{\accentset{{\cc@style\underline{\mskip10mu}}}}
\makeatother

\graphicspath{{images/}}

\definecolor{blue}{rgb}{0,0,0.9}
\definecolor{red}{rgb}{0.9,0,0}
\definecolor{green}{rgb}{0,0.9,0}
\definecolor{brown}{rgb}{0.6,0.1,0.1}
\definecolor{lightgreen}{rgb}{0.1,0.5,0.1}

\newcommand{\red}[1]{\begin{color}{red}#1\end{color}}

\newcommand{\brown}[1]{\begin{color}{brown}#1\end{color}}

\usepackage[colorlinks=true,
breaklinks=true,
bookmarks=true,
urlcolor=blue,
citecolor=lightgreen,
linkcolor=lightgreen,
bookmarksopen=false,
draft=false]{hyperref}

\begin{document}

\newtheorem{property}{Property}[section]
\newtheorem{proposition}{Proposition}[section]
\newtheorem{append}{Appendix}[section]
\newtheorem{definition}{Definition}[section]
\newtheorem{lemma}{Lemma}[section]
\newtheorem{corollary}{Corollary}[section]
\newtheorem{theorem}{Theorem}[section]
\newtheorem{remark}{Remark}[section]
\newtheorem{problem}{Problem}[section]
\newtheorem{example}{Example}[section]
\newtheorem{assumption}{Assumption}
\renewcommand*{\theassumption}{\Alph{assumption}}

\title{An Inexact Bregman Proximal Difference-of-Convex Algorithm with Two Types of Relative Stopping Criteria}

\author{
Lei Yang\thanks{School of Computer Science and Engineering, and Guangdong Province Key Laboratory of Computational Science, Sun Yat-sen University, Guangzhou, China ({\tt yanglei39@mail.sysu.edu.cn}). The research of this author is supported in part by
the National Key Research and Development Program of China under grant 2023YFB3001704, the National Natural Science Foundation of China under grant 12301411, and the Natural Science Foundation of Guangdong under grant 2023A1515012026.},
\quad
Jingjing Hu\thanks{(Corresponding author) School of Computer Science and Engineering, Sun Yat-sen University, Guangzhou, China ({\tt hujj53@mail2.sysu.edu.cn}). },
\quad
Kim-Chuan Toh\thanks{Department of Mathematics, and Institute of Operations Research and Analytics, National University of Singapore, Singapore ({\tt mattohkc@nus.edu.sg}). The research of this author is supported in part by the Ministry of Education, Singapore, under Academic Research Fund Tier 1 Grant A00084930000.
}
}

\date{\today} % started from November 5, 2019

\maketitle

\begin{abstract}
In this paper, we consider a class of difference-of-convex (DC) optimization problems, which require only a weaker restricted $L$-smooth adaptable property on the smooth part of the objective function, instead of the standard global Lipschitz gradient continuity assumption. Such problems are prevalent in many contemporary applications such as compressed sensing, statistical regression, and machine learning, and can be solved by a general Bregman proximal DC algorithm (BPDCA). However, the existing BPDCA is developed based on the stringent requirement that the involved subproblems must be solved \textit{exactly}, which is often impractical and limits the applicability of the BPDCA. To facilitate the practical implementations and wider applications of the BPDCA, we develop an inexact Bregman proximal difference-of-convex algorithm (iBPDCA) by incorporating two types of relative-type stopping criteria for solving the subproblems. The proposed inexact framework has considerable flexibility to encompass many existing exact and inexact methods, and can accommodate different types of errors that may occur when solving the subproblem. This enables the potential application of our inexact framework across different DC decompositions to facilitate the design of a more efficient DCA scheme in practice. The global subsequential convergence and the global sequential convergence of our iBPDCA are established under suitable conditions including the Kurdyka-{\L}ojasiewicz property. Some numerical experiments on the $\ell_{1-2}$ regularized least squares problem and the constrained $\ell_{1-2}$ sparse optimization problem are conducted to show the superior performance of our iBPDCA in comparison to existing algorithms. These results also empirically validate the necessity and significance of developing different types of stopping criteria to facilitate the efficient computation of the subproblem in each iteration of our iBPDCA.

\vspace{5mm}

%\noindent {\bf AMS subject classifications.} 90C05, 90C06, 90C25

\end{abstract}

%%%%%%%%%%%%%%%%%%%%%%%%%%%%%%%%%%%%%%%%%%%%%%%%%%%
\section{Introduction}

In this paper, we consider a class of difference-of-convex (DC) optimization problems as follows:
\begin{equation}\label{DCpro}
\min\limits_{\bm{x}\in\mathcal{Q}}~~f(\bm{x}) + P_1(\bm{x}) - P_2(\bm{x}),
\end{equation}
where $\mathcal{Q}\subseteq\mathbb{E}$ is a closed convex set with nonempty interior denoted by $\mathrm{int}\,\mathcal{Q}$, and $\mathbb{E}$ is a real finite-dimensional Euclidean space equipped with an inner product $\langle\cdot,\cdot\rangle$ and its induced norm $\|\cdot\|$. The function $P_1:\mathbb{E}\to(-\infty, \infty]$ is a proper closed convex function, $P_2:\mathbb{E}\to\mathbb{R}$ is a continuous convex function, and $f:\mathbb{E}\to\mathbb{R}$ is a continuously differentiable (possibly nonconvex) function on $\mathrm{int}\,\mathcal{Q}$, but may not have a globally Lipschitz continuous gradient. More specific assumptions on problem \eqref{DCpro} can be found later in Assumption \ref{assumA}. Problem \eqref{DCpro} arises in a variety of applications such as compressed sensing \cite{lyhx2015computing,ylhx2015minimization}, statistical regression \cite{fl2001variable,zhang2010nearly}, and machine learning \cite{zhang2010analysis,zts2020learning}, and has been extensively studied in the literature. We refer the readers to \cite{apx2017difference,ht1999dc,ltpd2018DC,lp2024open,pl1997convex} and references therein for more details on DC optimization problems and DC programming.

Due to the significance and popularity of the problem \eqref{DCpro}, tremendous efforts have been made to solve it efficiently, particularly when the problem involves a large number of variables. A classical algorithm for solving problem \eqref{DCpro} is the so-called DC algorithm (DCA) \cite{pl1997convex}, which iteratively replaces the concave part of the objective with a linear majorant approximation and solves the resulting subproblem as follows:
%\begin{equation*}%\label{DCAsubp}
$
\bm{x}^{k+1}\in\min\limits_{\bm{x}\in{\mathcal{Q}}}
\left\{ P_1(\bm{x}) + f(\bm{x})
- \langle\bm{\xi}^k,\,\bm{x}-\bm{x}^k\rangle\right\},
$
%\end{equation*}
where $\bm{\xi}^k\in\partial P_2(\bm{x}^k)$ is a (classical) subgradient of $P_2$ at $\bm{x}^k\in \mathcal{Q}$ (see the next section for the definition). But solving this subproblem is generally computationally demanding. By further exploring the structures of the objective, a so-called proximal DC algorithm (pDCA) (see, e.g, \cite{gtt2018DC,pl1998dc}) can be adopted to solve problem \eqref{DCpro}, which, in each iteration, not only replaces the concave part of the objective with a linear majorant approximation, but also replaces the smooth part with a quadratic majorant approximation. The basic iterative step of the pDCA reads as follows:
\begin{equation*}\label{pDCAsub}
\bm{x}^{k+1} = \min\limits_{\bm{x}\in{\mathcal{Q}}}
\left\{{P_1(\bm{x}) + \langle\nabla f(\bm{x}^k)-\bm{\xi^k},\,\bm{x}-\bm{x}^k\rangle}
+ \frac{\gamma}{2}\|\bm{x}-\bm{x^k}\|^2\right\},
\end{equation*}
where $\bm{\xi}^k\in\partial P_2(\bm{x}^k)$ and $\gamma>0$ is a proximal parameter depending on the Lipschitz constant $L$ of $\nabla f$. When $\mathcal{Q}$ is the entire space $\mathbb{E}$, the subproblem of the pDCA amounts to computing the proximal mapping of $P_1$, which is indeed an easy task for various choices of $P_1$ arising from contemporary application problems. Recently, inspired by Nesterov's acceleration techniques (see, e.g, \cite{n1983a,n2004introductory,n2013gradient}), Wen et al. \cite{wcp2018proximal} further incorporated an extrapolation step into the pDCA to achieve possible acceleration. The resulting algorithm is called pDCAe, which exhibits superior numerical performance.

The success of the pDCA and the pDCAe also motivates many subsequent works focusing on them and their variants for solving problem \eqref{DCpro} under appropriate settings; see, for example, \cite{lz2019nonmonotone,lzs2019enhanced,pra2017computing,plt2023difference}. Among them, Takahashi et al. \cite{tft2022new} considered the Bregman distance $\mathcal{D}_{\phi}$ (defined in the next section) associated with a kernel function $\phi$ as a proximity measure, and generalized the pDCA to a Bregman proximal DC algorithm (BPDCA), whose basic iterative step reads as follows:
\begin{equation}\label{BPDCAsub}
\bm{x}^{k+1} = \min\limits_{\bm{x}\in{\mathcal{Q}}}
\left\{P_1(\bm{x})
+ \langle\nabla f(\bm{x}^k)-\bm{\xi^k},\,\bm{x}-\bm{x}^k\rangle
+ \gamma\mathcal{D}_{\phi}(\bm{x},\,\bm{x}^k)\right\},
\end{equation}
where $\bm{\xi}^k\in\partial P_2\left(\bm{x}^k\right)$. The development of the BPDCA is indeed inspired by recent significant advancements in the Bregman proximal gradient method (see, e.g, \cite{bbt2017descent,bstv2017first,lu2018relatively,t2018simplified}). In comparison to the pDCA, such an extended Bregman framework presents two notable advantages. First, the BPDCA can be developed based on the notion of $L$-smooth adaptable property \cite[Definition 2.2]{bstv2017first}, which is less restrictive than the global Lipschitz gradient continuity of $f$ that is required by the pDCA(e) and its variants. Therefore, the BPDCA is applicable for a wider range of problems. Second, by selecting an appropriate kernel function $\phi$, the BPDCA can better leverage the inherent geometry and structure of the problem, potentially leading to a more tractable subproblem \eqref{BPDCAsub}; see examples in \cite{tft2022new,tti2023blind}.

Despite the aforementioned advantages, the BPDCA still encounters limitations in practical implementations. Indeed, although the kernel function can possibly capture the inherent geometry and structure of the problem, the associated subproblem in form of \eqref{BPDCAsub} generally has no closed-form solution and its computation may still be numerically demanding; see examples in \cite{cv2018entropic,rd2020bregman,yt2023inexact} and our numerical section \ref{sec-num}. Therefore, for the BPDCA to be implementable and practical, it should allow \textit{approximate} solutions to the subproblem with progressive accuracy and the corresponding
stopping criterion should be practically \textit{verifiable}. However, to our knowledge, research on the inexact version of the BPDCA remains unexplored. On the other hand, allowing the approximate minimization of the subproblem in proximal DC-type algorithms using the classical or weighted quadratic proximal term (corresponding to $\phi(\cdot):=\frac{1}{2}\|\cdot\|^2$ or $\phi(\cdot):=\frac{1}{2}\|\cdot\|^2_{H}$ with a symmetric positive definite matrix $H$) has been widely investigated under different (typically simpler) contexts in previous works.
%see, for example, \cite{lt2022inexact,nny2024inexact,ot2019inertial,sos2016global,ssc2003proximal,twst2020sparse,zts2020learning,zy2024inexact}.
However, these inexact versions either require the exact computation of an element in $\partial P_1$ (e.g., \cite{lt2022inexact,nny2024inexact,sos2016global,twst2020sparse,zts2020learning}), which could be difficult to satisfy when $P_1$ is not a simple function (see an example in the numerical section \ref{sec-num-l12con}), or they lack the global sequential convergence result (e.g., \cite{ot2019inertial,ssc2003proximal,zy2024inexact}).
Moreover, all these inexact versions %in \cite{lt2022inexact,nny2024inexact,ot2019inertial,sos2016global,zts2020learning}
require the global Lipschitz continuity of $\nabla f$ when they explicitly approximate the smooth part $f$ by a quadratic majorant. The aforementioned inadequacies could potentially restrict the applicability of these existing inexact frameworks, even under the Euclidean setting.
%studied in \cite{lt2022inexact,nny2024inexact,ot2019inertial,sos2016global,ssc2003proximal,twst2020sparse,zts2020learning,zy2024inexact}.

In this paper, to facilitate the practical implementations of the BPDCA and complement the aforementioned existing inexact proximal DC-type algorithms, %(see, e.g., \cite{lt2022inexact,nny2024inexact,ot2019inertial,sos2016global,ssc2003proximal,twst2020sparse,zts2020learning,zy2024inexact}),
we attempt to develop a general inexact Bregman proximal DC algorithm (iBPDCA) for solving the general problem \eqref{DCpro}. As we shall see later in Section \ref{sec-iBPDCA}, our inexact framework is developed based on two types of relative stopping criteria (SC1) and (SC2), which vary in their strategy to control the error incurred in the inexact minimization of the subproblem.
%\blue{Compared to the exact BPDCA in \cite{tft2022new}, as we shall see later in Section \ref{sec-iBPDCA}, in the $k$-iteration our inexact framework, an approximate solution is characterized by the error term $\Delta^k$ appearing on the left-hand-side of the optimality condition and $\partial_{\delta_k}P_1$ serving as an approximation of $\partial P_1$, and the error pair $(\Delta^k, \delta^k)$ is controlled by the relative stopping criteria (SC1) and (SC2), which vary in their strategy to control the error incurred in the inexact minimization of the subproblem.}
Although the difference may look subtle at the first glance, it indeed influences the behavior of our iBPDCA in both theoretical analysis and computational performance; see more discussions following Algorithm \ref{alg-iBPDCA} as well as Remarks \ref{rek-diff} and \ref{rek-verif}. Moreover, the resulting inexact Bregman proximal DCA framework is rather broad to cover many existing exact and inexact methods (studied in, e.g., \cite{sos2016global,ssc2003proximal,tft2022new,twst2020sparse,zts2020learning}), while also introducing new inexact variants.
%potentially introduce a new inexact version for them.
Additionally, it can accommodate different types of errors that may occur when solving the subproblem. The inherent versatility and flexibility of our iBPDCA
readily enable its potential application across different DC decompositions for a given problem, thereby facilitating the design of a more efficient DCA scheme, as deliberated in Remark \ref{rek-diff}. Furthermore, in Section \ref{sec-conv}, we meticulously establish both global subsequential and global sequential convergence results for our iBPDCA under suitable conditions including the renown Kurdyka-{\L}ojasiewicz property. Finally, we conduct some numerical experiments in Section \ref{sec-num} to illustrate the performance of our iBPDCA for solving the $\ell_{1-2}$ regularized least squares problem and the constrained $\ell_{1-2}$ sparse optimization problem. The computational results demonstrate the superior performance of our inexact framework in comparison to existing algorithms, and they also empirically validate the necessity and significance of developing different types of stopping criteria for solving the subproblem.

The rest of this paper is organized as follows. We present notation and preliminaries in Section \ref{sec-not}. We then describe our iBPDCA for solving problem \eqref{DCpro} and establish preliminary properties in Section \ref{sec-iBPDCA}, followed by a comprehensive study on convergence analysis in Section \ref{sec-conv}. Some numerical experiments are conducted in Section \ref{sec-num}, with conclusions given in Section \ref{sec-conc}.

%%%%%%%%%%%%%%%%%%%%%%%%%%%%%%%%%%%%%%%%%%%%%%%%%%%%%%%%%%%%%%%%%%%%%
\section{Notation and preliminaries}\label{sec-not}

In this paper, we present scalars, vectors, and matrices in lower case letters, bold lower case letters, and upper case letters, respectively. We also use $\mathbb{R}$, $\mathbb{R}^n$, and $\mathbb{R}^{m\times n}$ to denote the set of real numbers, $n$-dimensional real vectors, and $m\times n$ real matrices, respectively. For a vector $\bm{x}\in\mathbb{R}^n$, $x_i$ denotes its $i$-th entry, $\|\bm{x}\|$ denotes its Euclidean norm, $\|\bm{x}\|_1$ denotes its $\ell_1$ norm defined by $\|\bm{x}\|_1:=\sum^n_{i=1}|x_i|$, and $\|\bm{x}\|_H:=\sqrt{\langle\bm{x},\, H \bm{x}\rangle}$ denotes its weighted norm associated with a symmetric positive definite matrix $H$.
% For a matrix $A\in\mathbb{R}^{m\times n}$, $a_{ij}$ denotes its $(i,j)$th entry, $A_{:j}$ denotes its $j$th column, $\|A\|$ denotes its spectral norm, $\|A\|_F$ denotes its Frobenius norm, and $\lambda_{\max}(A)$ and $\lambda_{\min}(A)$ denote its largest and smallest eigenvalues, respectively.
For a closed set $\mathcal{X}\subseteq\mathbb{E}$, its indicator function $\iota_{\mathcal{X}}$ is defined by $\iota_{\mathcal{X}}(\bm{x})=0$ if $\bm{x}\in\mathcal{X}$ and $\iota_{\mathcal{X}}(\bm{x})=+\infty$ otherwise. The distance from a point $\bm{x}$ to $\mathcal{X}$ is defined by $\mathrm{dist}(\bm{x},\,\mathcal{X}):=\inf_{\bm{y}\in\mathcal{X}}\|\bm{y}-\bm{x}\|$.

For an extended-real-valued function $h: \mathbb{E} \rightarrow [-\infty,\infty]$, we say that it is \textit{proper} if $h(\bm{x}) > -\infty$ for all $\bm{x} \in \mathbb{E}$ and its domain ${\rm dom}\,h:=\{\bm{x} \in \mathbb{E} : h(\bm{x}) < \infty\}$ is nonempty. A proper function $h$ is said to be closed if it is lower semicontinuous. We also use the notation $\bm{y} \xrightarrow{h} \bm{x}$ to denote $\bm{y} \rightarrow \bm{x}$ and $h(\bm{y}) \rightarrow h(\bm{x})$. The \textit{(limiting) subdifferential} (see \cite[Definition 8.3]{rw1998variational}) of $h$ at $\bm{x} \in \mathrm{dom}\,h$ is given by
\begin{equation*}
\partial h(\bm{x}) := \left\{ \bm{d} \in \mathbb{E}: \exists \,\bm{x}^k \xrightarrow{h} \bm{x}, ~\bm{d}^k \rightarrow \bm{d} ~~\mathrm{with}~\liminf\limits_{\bm{y} \rightarrow \bm{x}^k, \,\bm{y} \neq \bm{x}^k}\, \frac{h(\bm{y})-h(\bm{x}^k)-\langle \bm{d}^k, \bm{y}-\bm{x}^k\rangle}{\|\bm{y}-\bm{x}^k\|} \geq 0\ ~\forall k\right\}.
\end{equation*}
%It can be observed from the above definition that
%\begin{equation}\label{robust}
%\left\{ \bm{d}\in\mathbb{E}: \exists \,\bm{x}^k \xrightarrow{h} \bm{x}, ~\bm{d}^k \rightarrow \bm{d} ~\mathrm{with}~\bm{d}^k \in \partial h(\bm{x}^k)~\mathrm{for}~\mathrm{each}~k \right\} \subseteq \partial h(\bm{x}).
%\end{equation}
When $h$ is continuously differentiable or convex, the above subdifferential coincides with the classical concept of derivative or convex subdifferential of $h$; see, for example, \cite[Exercise~8.8]{rw1998variational} and \cite[Proposition~8.12]{rw1998variational}. Moreover, for a proper closed convex function $h: \mathbb{E} \rightarrow (-\infty, \infty]$ and a given $\varepsilon \geq 0$, the $\varepsilon$-subdifferential of $h$ at $\bm{x}\in{\rm dom}\,h$ is defined by $\partial_{\varepsilon} h(\bm{x}):=\{\bm{d}\in\mathbb{E}: h(\bm{y}) \geq h(\bm{x}) + \langle \bm{d}, \,\bm{y}-\bm{x} \rangle - \varepsilon, ~\forall\,\bm{y}\in\mathbb{E}\}$, and when $\varepsilon=0$, $\partial_{\varepsilon}h$ simply reduces to the classical convex subdifferential $\partial h$.

%The conjugate function of $h$ is the function $h^*: \mathbb{E} \rightarrow (-\infty,\infty]$ defined by $h^*(\bm{y}):=\sup\left\{\langle \bm{y},\,\bm{x}\rangle-h(\bm{x}): \bm{x}\in\mathbb{E}\right\}$.
The function $h$ is \textit{essentially smooth} if (i) $\mathrm{int}\,\mathrm{dom}\,h$ is non-empty; (ii) $h$ is differentiable on $\mathrm{int}\,\mathrm{dom}\,h$; (iii) $\|\nabla h(\bm{x}^k)\|\to\infty$ for every sequence $\{\bm{x}^k\}$ in $\mathrm{int}\,\mathrm{dom}\,h$ converging to a boundary point of $\mathrm{int}\,\mathrm{dom}\,h$; see
\cite[page 251]{r1970convex}. For a proper closed function $h: \mathbb{E} \rightarrow (-\infty,\infty]$ and $\nu>0$, the proximal mapping of $\nu h$ at $\bm{y}\in\mathbb{E}$ is defined by
\begin{equation*}
\texttt{prox}_{\nu h}(\bm{y}) := \mathop{\mathrm{Argmin}}\limits_{\bm{x}\in\mathbb{E}} \left\{h(\bm{x}) + \frac{1}{2\nu}\|\bm{x} - \bm{y}\|^2\right\}.
\end{equation*}

Given a proper closed strictly convex function $\phi: \mathbb{E} \rightarrow (-\infty,\infty]$, finite at $\bm{x}$, $\bm{y}$ and differentiable at $\bm{y}$ but not necessarily at $\bm{x}$, the \textit{Bregman distance} \cite{b1967relaxation} between $\bm{x}$ and $\bm{y}$ associated with the kernel function $\phi$ is defined as
\begin{equation*}
\mathcal{D}_{\phi}(\bm{x}, \,\bm{y}) := \phi(\bm{x}) - \phi(\bm{y}) - \langle \nabla\phi(\bm{y}), \,\bm{x} - \bm{y} \rangle. %\quad \forall\,\bm{x}\in\mathrm{dom}\,\phi, ~\bm{y}\in\mathrm{int}\,\mathrm{dom}\,\phi.
\end{equation*}
It is easy to see that $D_{\phi}(\bm{x}, \,\bm{y})\geq0$ and equality holds if and only if $\bm{x}=\bm{y}$ due to the strict convexity of $\phi$. When $\mathbb{E}:=\mathbb{R}^n$ and $\phi(\cdot):=\|\cdot\|^2$, $\mathcal{D}_{\phi}(\cdot,\cdot)$ readily recovers the classical squared Euclidean distance. Based on the Bregman distance, we then define a \textit{restricted $L$-smooth adaptable property} as follows.

\begin{definition}[{\bf Restricted $L$-smooth adaptable on $\mathcal{X}$}]\label{defresLsmad}
Let $f, \,\phi: \mathbb{E}\to(-\infty, \infty]$ be proper closed convex functions with $\mathrm{dom}\,f\supseteq\mathrm{dom}\,\phi$, and  $f,\,\phi$ are differentiable on $\mathrm{int}\,\mathrm{dom}\,\phi$. Given a closed set $\mathcal{X}\subseteq\mathbb{E}$ with $\mathcal{X}\cap\mathrm{int}\,\mathrm{dom}\,\phi\neq\emptyset$, we say that $(f,\phi)$ is $L$-smooth adaptable restricted on $\mathcal{X}$ if there exists $L\geq0$ such that
\begin{equation*}\label{resLsmad}
\big|f(\bm{y}) - f(\bm{x}) - \langle \nabla f(\bm{x}), \,\bm{y}-\bm{x}\rangle\big|
\leq L\mathcal{D}_{\phi}(\bm{y}, \,\bm{x}), \quad \forall\,
\bm{x}\in\mathcal{X}\cap \mathrm{int}\,\mathrm{dom}\,\phi,
~\bm{y}\in\mathcal{X}\cap \mathrm{dom}\,\phi.
\end{equation*}
\end{definition}

Some remarks are in order concerning this definition. The above restricted $L$-smooth adaptable property is derived by modifying the original $L$-smooth adaptable property introduced in \cite[Section 2.2]{bstv2017first} by imposing a restricted set $\mathcal{X}$, and it readily reduces to the original notion when $\mathcal{X}$ is the entire space $\mathbb{E}$. Implementing such a restriction would facilitate the expansion of the original notion to encompass a broader range of choices of $(f, \,\phi)$ with a proper $\mathcal{X}$. For example, when $\nabla f$ is $L_f$-Lipschitz continuous on $\mathbb{R}^n$ and $\phi$ is $\mu_{\phi}$-strongly convex on $\mathcal{X}\subseteq\mathbb{R}^n$, then it can be verified that $(f,\phi)$ is $\frac{L_f}{\mu_{\phi}}$-smooth adaptable restricted on $\mathcal{X}$, but $(f,\phi)$ may not be $L$-smooth adaptable on $\mathrm{int}\,\mathrm{dom}\,\phi$ according to the original definition in \cite[Section 2.2]{bstv2017first} because $\phi$ may not be strongly convex on its whole interior. For example, the entropy function $\phi(x)=\sum^n_{i=1}x_i(\log x_i - 1)$ is $\frac{1}{\alpha}$-strongly convex on $[0,\,\alpha]^n$ with any $\alpha>0$, but it is not strongly convex on $\mathrm{int}\,\mathrm{dom}\,\phi=\mathbb{R}^n_{++}$. Therefore, adopting the notion of restricted $L$-smooth adaptable property in Definition \ref{defresLsmad} could broaden the potential applications of the algorithm developed in this work.

In order to establish the rigorous convergence analysis for the Bregman-type algorithm, we make the following blanket technical assumptions.

\begin{assumption}\label{assumA}
Problem \eqref{DCpro} and the kernel function $\phi$ satisfy the following assumptions.
\begin{itemize}%[leftmargin=1cm]

\item[{\bf A1.}] $\phi:\mathbb{E}\to(-\infty, \infty]$ is essentially smooth and strictly convex on  $\mathrm{dom}\,\phi$. Moreover, $\overline{\mathrm{dom}\,\phi} = \mathcal{Q}$, where $\overline{\mathrm{dom}\,\phi}$ denotes the closure of ${\mathrm{dom}\,\phi}$.

\item[{\bf A2.}] $P_1:\mathbb{E}\to(-\infty, \infty]$ is a proper closed convex function with $\mathrm{dom}\,P_1\cap\mathrm{int}\,\mathcal{Q}\neq\emptyset$, and $P_2:\mathbb{E}\to\mathbb{R}$ is a continuous convex function.

\item[{\bf A3.}] $f:\mathbb{E}\to(-\infty, \infty]$ is a proper closed function with $\mathrm{dom}\,f\supseteq\mathrm{dom}\,\phi$ and $f$ is continuously differentiable on $\mathrm{int}\,\mathrm{dom}\,f$. Moreover, there exists a closed set $\mathcal{X} \supseteq \mathrm{dom}\,P_1\cap\mathcal{Q}$ such that $(f,\phi)$ is ${L}$-smooth adaptable restricted on $\mathcal{X}$.

\item[{\bf A4.}] $F(\bm{x}) := \iota_{\mathcal{Q}}(\bm{x}) + P_1(\bm{x}) - P_2(\bm{x}) + f(\bm{x})$ is level-bounded, i.e., the level set $\big\{\bm{x}\in\mathbb{E}: F(\bm{x}) \leq \alpha\big\}$ is bounded (possibly empty) for every $\alpha\in\mathbb{R}$, where $\iota_{\mathcal{Q}}$ is the indicator function on the set $\mathcal{Q}$.
    %and $F^*:=\inf \{F(\bm{x}): \bm{x} \in \mathcal{Q}\}>-\infty$, i.e., problem (1.1) is bounded from below.

\item[{\bf A5.}] Each subproblem of the iBPDCA in Algorithm \ref{alg-iBPDCA} is well-defined in the sense that the set of optimal solutions of the subproblem is nonempty and located in $\mathrm{int}\,\mathrm{dom}\,\phi$.
\end{itemize}
\end{assumption}

The above assumptions are commonly made for studying the convergence of the Bregman-type algorithms. Assumption \ref{assumA}5 ensures the well-posedness of iterates. It can be satisfied when, for example, $\phi$ is strongly convex with $\mathrm{dom}\,\phi=\mathbb{E}$. Other sufficient conditions can be found in \cite[Section 3]{bstv2017first}.

We next recall the Kurdyka-{\L}ojasiewicz (KL) property (see \cite{ab2009on,abrs2010proximal,abs2013convergence,bdl2007the,bst2014proximal} for more details), which is now a common technical condition for establishing the convergence of the whole sequence. A large number of functions such as proper closed semialgebraic functions satisfy the KL property \cite{abrs2010proximal,abs2013convergence}. For notational simplicity, let $\Xi_{\nu}$ ($\nu>0$) denote a class of concave functions $\varphi:[0,\nu) \rightarrow \mathbb{R}_{+}$ satisfying: (i) $\varphi(0)=0$; (ii) $\varphi$ is continuously differentiable on $(0,\nu)$ and continuous at $0$; (iii) $\varphi'(t)>0$ for all $t\in(0,\nu)$. Then, the KL property can be described as follows.

\begin{definition}[\textbf{KL property and KL function}]
Let $h: \mathbb{R}^n \rightarrow \mathbb{R} \cup \{+\infty\}$ be a proper closed function.
\begin{itemize}
\item[(i)] For $\tilde{\bm{x}}\in{\rm dom}\,\partial h:=\{\bm{x} \in \mathbb{R}^{n}: \partial h(\bm{x}) \neq \emptyset\}$, if there exist a $\nu\in(0, +\infty]$, a neighborhood $\mathcal{V}$ of $\tilde{\bm{x}}$ and a function $\varphi \in \Xi_{\nu}$ such that for all $\bm{x} \in \mathcal{V} \cap \{\bm{x}\in \mathbb{R}^{n} : h(\tilde{\bm{x}})<h(\bm{x})<h(\tilde{\bm{x}})+\nu\}$, it holds that
    \begin{eqnarray*}
    \varphi'(h(\bm{x})-h(\tilde{\bm{x}}))\,\mathrm{dist}(0, \,\partial h(\bm{x})) \geq 1,
    \end{eqnarray*}
    then $h$ is said to have the \textbf{Kurdyka-{\L}ojasiewicz (KL)} property at $\tilde{\bm{x}}$.

\item[(ii)] If $h$ satisfies the KL property at each point of ${\rm dom}\,\partial h$, then $h$ is called a KL function.
\end{itemize}
\end{definition}

%Based on the above definition, we then introduce the KL exponent \cite{abrs2010proximal,lp2017calculus}.
%
%\begin{definition}[\textbf{KL exponent}]
%Suppose that $h: \mathbb{R}^n \rightarrow \mathbb{R} \cup \{+\infty\}$ is a proper closed function satisfying the KL property at $\tilde{\bm{x}}\in{\rm dom}\,\partial h$ with $\varphi(t)=a' t^{1-\theta}$ for some $a' > 0$ and $\theta\in[0, 1)$, i.e., there exist $a, \varepsilon, \nu > 0$ such that
%\begin{eqnarray*}
%\mathrm{dist}(0, \,\partial h(\bm{x})) \geq a \left(h(\bm{x}) - h(\tilde{\bm{x}})\right)^{\theta}
%\end{eqnarray*}
%whenever $\bm{x} \in {\rm dom}\,\partial h$, $\|\bm{x}-\tilde{\bm{x}}\| \leq \varepsilon$ and $h(\tilde{\bm{x}})<h(\bm{x})<h(\tilde{\bm{x}})+\nu$. Then, $h$ is said to have the KL property at $\tilde{\bm{x}}$ with an exponent $\theta$. If $h$ is a KL function and has the same exponent $\theta$ at any $\tilde{\bm{x}}\in{\rm dom}\,\partial h$, then $h$ is said to be a KL function with an exponent $\theta$.
%\end{definition}

We also recall the uniformized KL property, which was established in \cite[Lemma 6]{bst2014proximal}.

\begin{proposition}[\textbf{Uniformized KL property}]\label{uniKL}
Suppose that $h: \mathbb{R}^n \rightarrow \mathbb{R} \cup \{+\infty\}$ is a proper closed function and $\Gamma$ is a compact set. If $h \equiv \zeta$ on $\Gamma$ for some constant $\zeta$ and satisfies the KL property at each point of $\Gamma$, then there exist $\varepsilon>0$, $\nu>0$ and $\varphi \in \Xi_{\nu}$ such that
\begin{eqnarray*}
\varphi'(h(\bm{x}) - \zeta)\,\mathrm{dist}(0, \,\partial h(\bm{x})) \geq 1
\end{eqnarray*}
for all $\bm{x} \in \{\bm{x}\in\mathbb{R}^{n}: \mathrm{dist}(\bm{x},\,\Gamma)<\varepsilon\} \cap \{\bm{x}\in \mathbb{R}^{n} : \zeta < h(\bm{x}) < \zeta + \nu\}$.
\end{proposition}

%We close this section by giving a supporting lemma which is routine to verify and will be used in the subsequent analysis. % The identity in first lemma is routine to verify and the proofs of last two lemmas are relegated to Appendix \ref{apd-lemmas}.
%
%\begin{lemma}[{\bf Four points identity}]%\label{lemfourId}
%Suppose that a proper closed strictly convex function $\phi: \mathbb{E} \rightarrow (-\infty,\infty]$ is finite at $\bm{a},\,\bm{b},\,\bm{c},\,\bm{d}$ and differentiable at $\bm{a},\,\bm{b}$. Then,
%\begin{equation*}\label{fourId}
%\langle \nabla\phi(\bm{a})-\nabla\phi(\bm{b}),\,\bm{c}-\bm{d} \rangle = \mathcal{D}_{\phi}(\bm{c},\,\bm{b}) + \mathcal{D}_{\phi}(\bm{d},\,\bm{a}) - \mathcal{D}_{\phi}(\bm{c},\,\bm{a}) - \mathcal{D}_{\phi}(\bm{d},\,\bm{b}).
%\end{equation*}
%\end{lemma}

%%%%%%%%%%%%%%%%%%%%%%%%%%%%%%%%%%%%%%%%%%%%%%%%%%%%%%%%%%%%%%%%%%%%
\section{An inexact Bregman proximal difference-of-convex algorithm}\label{sec-iBPDCA}

In this section, we shall develop an inexact Bregman proximal difference-of-convex algorithm (iBPDCA) based on two types of relative stopping criteria for solving problem \eqref{DCpro}, and study the preliminary convergence properties. The complete framework is presented in Algorithm \ref{alg-iBPDCA}.

\begin{algorithm}[ht]
\caption{An inexact Bregman proximal difference-of-convex algorithm (iBPDCA) for solving problem \eqref{DCpro}}\label{alg-iBPDCA}
\textbf{Input:} Follow Assumption \ref{assumA} to choose a kernel function $\phi$.
% with $L\geq0$ and $\mathcal{X}\supseteq\mathrm{dom}\,P_1\,\cap\,\mathcal{Q}$.
Arbitrarily choose $\bm{x}^0\in\mathrm{dom}\,P_1\,\cap\,\mathrm{int}\,\mathrm{dom}\,\phi$ and a sequence $\{\gamma_k\}_{k=0}^{\infty}$ satisfying $L<\gamma_{\min}\leq\gamma_k\leq\gamma_{\max}<\infty$ for all $k\geq0$. For criterion (SC1), choose $0\leq\sigma<(\gamma_{\min}-L)/\gamma_{\min}$. For criterion (SC2), choose {$0\leq\sigma<(\gamma_{\min}-L)/\gamma_{\max}$} and $\bm{x}^{-1}\in\mathrm{int}\,\mathrm{dom}\,\phi$.  \\[3pt]
\textbf{while} a termination criterion is not met, \textbf{do} \vspace{-2mm}
\begin{itemize}[leftmargin=2cm]
\item[\textbf{Step 1}.] Take any $\bm{\xi}^k\in\partial P_2(\bm{x}^k)$, and find a point $\bm{x}^{k+1}$ associated with an error pair $(\Delta^k,\,\delta_k)$ by approximately solving
    \begin{equation}\label{iBCDA-subpro}
    \min\limits_{\bm{x}}~ P_1(\bm{x}) + \langle \nabla f(\bm{x}^k)-\bm{\xi}^k, \,\bm{x}-\bm{x}^k\rangle + \gamma_k\mathcal{D}_{\phi}(\bm{x},\,\bm{x}^k),
    \end{equation}
    such that $\bm{x}^{k+1}\in\mathrm{dom}\,P_1\,\cap\,\mathrm{int}\,\mathrm{dom}\,\phi$, $\delta_k\geq0$, and
	\begin{equation}\label{iBPDCA-inexcond}
    \Delta^k\in \partial_{\delta_k}{P_1}(\bm{x}^{k+1}) + \nabla f(\bm{x}^k)-\bm{\xi}^k + \gamma_k\big(\nabla \phi(\bm{x}^{k+1})-\nabla \phi(\bm{x}^{k})\big),
	\end{equation}
	satisfying one of the following relative stopping criteria:
    \begin{equation*}
    \begin{aligned}
    &\textbf{(SC1)} \qquad \|\Delta^k\|^2 + |\langle\Delta^k, \,\bm{x}^{k+1}-\bm{x}^{k}\rangle| + \delta_k \leq \sigma\gamma_k\mathcal{D}_{\phi}(\bm{x}^{k+1}, \,\bm{x}^{k}),  \\
    &\textbf{(SC2)} \qquad \|\Delta^k\|^2 + |\langle\Delta^k, \,\bm{x}^{k+1}-\bm{x}^{k}\rangle| + \delta_k \leq \sigma\gamma_k\mathcal{D}_{\phi}(\bm{x}^{k}, \,\bm{x}^{k-1}).
    %&\textbf{(SC3)} \qquad  D_{\phi}(\widetilde{\bm{x}}^{k+1}, \,\bm{x}^{k+1}) + \lambda^{-1}\delta_k \leq \sigma_k\min\big\{D_{\phi}(\widetilde{\bm{x}}^{k+1}, \,\bm{x}^{k}), \,D_{\phi}(\widetilde{\bm{x}}^{k}, \,\bm{x}^{k+1})\big\}.
    \end{aligned}
    \end{equation*}

\item [\textbf{Step 2}.] Set $k = k+1$ and go to \textbf{Step 1}. \vspace{-1.5mm}
\end{itemize}
\textbf{end while}  \\
\textbf{Output}: $\bm{x}^{k}$ \vspace{0.5mm}
\end{algorithm}

Note from the iBPDCA in Algorithm \ref{alg-iBPDCA} that, at each iteration, our inexact framework allows one to approximately solve the subproblem \eqref{iBCDA-subpro} under condition \eqref{iBPDCA-inexcond} satisfying one of the relative stopping criteria (SC1) and (SC2). Since the subproblem \eqref{iBCDA-subpro} admits a solution $\bm{x}^{k,*} \in \mathrm{dom}\,P_1\,\cap\,\mathrm{int}\,\mathrm{dom}\,\phi$ (by Assumption \ref{assumA}5), then condition \eqref{iBPDCA-inexcond} always holds at $\bm{x}^{k+1}=\bm{x}^{k,*}$ with $\|\Delta^k\|=\delta_k=0$ and hence it is achievable. Moreover, when setting $\sigma=0$ in (SC1) or (SC2), it means that $\bm{x}^{k+1}$ is an exact solution of the subproblem, and thus, our iBPDCA readily reduces to an exact BPDCA studied in \cite{tft2022new}.

%\blue{Compared to the exact BPDCA in \cite{tft2022new},} 
One can also observe that, in our inexact framework, an approximate solution is characterized by the error term $\Delta^k$ appearing on the left-hand-side of the optimality condition and $\partial_{\delta_k}P_1$ serving as an approximation of $\partial P_1$. This makes our inexact framework rather flexible to accommodate different scenarios when solving the subproblem inexactly. Moreover, the key difference between (SC1) and (SC2) lies in the strategy to control the error incurred in the inexact minimization of the subproblem. Specifically, the error ${\|\Delta^k\|^2 + |\langle\Delta^k, \,\bm{x}^{k+1}-\bm{x}^{k}\rangle| + \delta_k}$ is regulated by the most recent successive difference $\sigma\gamma_k\mathcal{D}_{\phi}(\bm{x}^{k+1}, \,\bm{x}^{k})$ in (SC1), whereas it is governed by the preceding successive difference $\sigma\gamma_k\mathcal{D}_{\phi}(\bm{x}^{k}, \,\bm{x}^{k-1})$ in (SC2). This difference may look subtle, but it indeed results in distinct behaviors of our iBPDCA in subsequent theoretical analysis and computational performance. Specifically, in comparison to the iBPDCA with (SC1), the iBPDCA with (SC2) would require a possibly narrower selection range for the tolerance parameter $\sigma$ to guarantee the convergence, and the related analysis is more intricate. But, from a computational perspective, the iBPDCA with (SC2) could be more advantageous in saving the verification cost, especially when the computation of the associated Bregman distance is demanding for large-scale problems; see Remark \ref{rek-verif} for more details. This computational benefit indeed serves as the motivation for developing the stopping criterion (SC2) in this work.

The iBPDCA in Algorithm \ref{alg-iBPDCA} also offers a comprehensive inexact algorithmic framework that can encompass numerous existing exact and inexact methods (studied in, e.g., \cite{sos2016global,ssc2003proximal,tft2022new,twst2020sparse,zts2020learning}) and potentially introduces new inexact variants. For example, when the concave part $P_2\equiv0$ and the smooth part $f$ is convex, our iBPDCA gives an inexact Bregman proximal gradient method for solving a convex composite problem. This inexact version employs a more flexible relative stopping criterion,  (SC1) or (SC2), making it different from the existing inexact version developed in \cite{yt2023inexact}. Moreover, when the smooth part $f$ is convex with a global Lipschitz continuous gradient and the kernel function is chosen as $\phi(\bm{x}):=\frac{1}{2}\|\bm{x}\|^2$, our iBPDCA also gives an inexact version for the classic proximal DC algorithm. In addition, when the smooth part $f\equiv0$, the kernel function is chosen as $\phi(\bm{x})=\frac{1}{2}\|\bm{x}\|^2+\frac{1}{2}\|A\bm{x}\|^2$ with $A\in\mathbb{R}^{m\times n}$ being a coefficient matrix specified by the problem,
and no error occurs in the computation of $\partial P_1$ (namely, $\delta_k\equiv0$), our iBPDCA with (SC1) subsumes the inexact proximal DC algorithms developed by Tang et al. \cite{twst2020sparse} and Zhang et al. \cite{zts2020learning} for solving the nonconvex square-root-loss regression problem and the MCP penalized graph Laplacian learning problem, respectively. In comparison to the inexact versions in \cite{twst2020sparse,zts2020learning}, our iBPDCA addresses a more general problem \eqref{DCpro} in a broader Bregman framework. It also allows an appropriate approximation of $\partial P_1$, which could be essential for practical implementations when solving a problem with complex constraints; see an example in the numerical section \ref{sec-num-l12con}. This versatility thus enhances the applicability of our approach to a wider range of problems.

We next establish some useful properties of our iBPDCA, which will pave the way to the convergence analysis in the subsequent section. %\blue{Attention should be paid to the fact that since the error terms $\Delta^k$ and $\delta^k$ are incorporated into the algorithm framework, the corresponding analysis of our iBPDCA in the subsequent is more complicated, and cannot be directly adapted from \cite{tft2022new}.}

\begin{lemma}[\textbf{Approximate sufficient descent property}]\label{lem-suffdes}
Suppose that Assumption \ref{assumA} holds. Let $\{\bm{x}^k\}$ be the sequence generated by the iBPDCA in Algorithm \ref{alg-iBPDCA}, then
\begin{equation}\label{suffdes-iBPDCA}
F(\bm{x}^{k+1}) - F(\bm{x}^{k})
\leq -(\gamma_k-L)\mathcal{D}_{\phi}(\bm{x}^{k+1},\,\bm{x}^{k})
-\gamma_k\mathcal{D}_{\phi}(\bm{x}^k,\,\bm{x}^{k+1})
+ \big|\langle\Delta^k, \,\bm{x}^{k+1}-\bm{x}^{k}\rangle\big| + \delta_k.
\end{equation}
\end{lemma}
\begin{proof}
First, from condition \eqref{iBPDCA-inexcond}, there exists a $\bm{d}^{k+1} \in \partial_{\delta_k} {P_1}(\bm{x}^{k+1})$ such that
\begin{equation*}%\label{optcondcvx2}
\Delta^k = \bm{d}^{k+1} + \nabla f(\bm{x}^k)-\bm{\xi}^k + \gamma_k\big(\nabla \phi(\bm{x}^{k+1})-\nabla \phi(\bm{x}^{k})\big).
\end{equation*}
It then follows from $\bm{x}^k\in\mathrm{dom}\,P_1\,\cap\,\mathrm{int}\,\mathrm{dom}\,\phi$ for all $k\geq0$ and the definition of $\partial_{\delta_k}P_1$ that
\begin{equation*}
\begin{aligned}
~ P_1(\bm{x}^k)&\geq P_1(\bm{x}^{k+1}) + \langle \bm{d}^{k+1}, \,\bm{x}^k - \bm{x}^{k+1} \rangle - \delta_k\\
&= P_1(\bm{x}^{k+1}) + \langle \Delta^k-\nabla f(\bm{x}^k)+\bm{\xi}^k - \gamma_k\big(\nabla \phi(\bm{x}^{k+1})-\nabla \phi(\bm{x}^{k})\big), \,\bm{x}^{k} - \bm{x}^{k+1}\rangle - \delta_k,
\end{aligned}
\end{equation*}
which implies that
\begin{equation*}%\label{ineq1tmp1-iBPDCA2}
\begin{aligned}
P_1(\bm{x}^{k+1})
&\leq  P_1(\bm{x}^k) - \langle \nabla f(\bm{x}^k)-\bm{\xi}^k, \,\bm{x}^{k+1} - \bm{x}^k \rangle  \\
&\qquad + \gamma_k\langle \nabla \phi(\bm{x}^{k+1})-\nabla \phi(\bm{x}^{k}), \,\bm{x}^k-\bm{x}^{k+1}\rangle + \big|\langle\Delta^k, \,\bm{x}^{k+1}-\bm{x}^{k}\rangle\big| + \delta_k.
\end{aligned}
\end{equation*}
Moreover, note from the definition of $\mathcal{D}_{\phi}$ that
\begin{equation*}%\label{ineq1tmp2-iBPDCA2}
\langle \,\nabla\phi(\bm{x}^{k+1}) - \nabla\phi(\bm{x}^{k}), \,\bm{x}^k-\bm{x}^{k+1} \,\rangle
=-\mathcal{D}_{\phi}(\bm{x}^k,\,\bm{x}^{k+1}) -\mathcal{D}_{\phi}(\bm{x}^{k+1},\,\bm{x}^k).
\end{equation*}
Combining above two relations, we have that
\begin{equation}\label{ineq1-iBPDCA2}
\begin{aligned}
P_1(\bm{x}^{k+1})
&\leq P_1(\bm{x}^k) - \langle \nabla f(\bm{x}^k)-\bm{\xi}^k, \,\bm{x}^{k+1} - \bm{x}^k \rangle \\
&\qquad - \gamma_k\big(\mathcal{D}_{\phi}(\bm{x}^k,\,\bm{x}^{k+1})
+ \mathcal{D}_{\phi}(\bm{x}^{k+1},\,\bm{x}^k)\big)
+ \big|\langle\Delta^k, \,\bm{x}^{k+1}-\bm{x}^{k}\rangle\big| + \delta_k.
\end{aligned}
\end{equation}
On the other hand, from the convexity of $P_2$ and $\bm{\xi}^k\in\partial P_2(\bm{x}^{k})$, we obtain that
\begin{equation}\label{ineq2-iBPDCA2}
P_2(\bm{x}^{k+1}) \geq P_2(\bm{x}^k) + \langle\bm{\xi}^k, \,\bm{x}^{k+1}-\bm{x}^{k}\rangle ~~ \Longleftrightarrow ~~
-P_2(\bm{x}^{k+1}) \leq -P_2(\bm{x}^k) - \langle\bm{\xi}^k, \,\bm{x}^{k+1}-\bm{x}^{k}\rangle.
\end{equation}
Moreover, since $(f,\phi)$ is ${L}$-smooth adaptable restricted on $\mathcal{X} \supseteq \mathrm{dom}\,P_1\cap\mathcal{Q}$ (by Assumption \ref{assumA}3) and $\bm{x}^{k},\,\bm{x}^{k+1}\in\mathrm{dom}\,P_1\,\cap\,\mathrm{int}\,\mathrm{dom}\,\phi\,\subseteq\,\mathrm{dom}\,P_1\cap\mathcal{Q}$ (by Assumption \ref{assumA}1), we obtain that
\begin{equation}\label{ineq3-iBPDCA2}
f(\bm{x}^{k+1}) \leq f(\bm{x}^k) + \langle \nabla f(\bm{x}^k), \,\bm{x}^{k+1}-\bm{x}^{k}\rangle + {L}\mathcal{D}_{\phi}(\bm{x}^{k+1},\,\bm{x}^{k}).
\end{equation}
Thus, summing \eqref{ineq1-iBPDCA2}, \eqref{ineq2-iBPDCA2} and \eqref{ineq3-iBPDCA2}, together with $\bm{x}^k\in\mathrm{int}\,\mathrm{dom}\,\phi\subseteq\mathcal{Q}$ for all $k\geq0$, we can obtain the desired result. This completes the proof.
\end{proof}

Based on the above approximate sufficient descent property, we further have the following results for the iBPDCA with (SC1) and (SC2), respectively.

\begin{proposition}[\textbf{Properties of the iBPDCA with (SC1)}]\label{pro-bdSC1}
Suppose that Assumption \ref{assumA} holds, $L<\gamma_{\min}\leq\gamma_k\leq\gamma_{\max}<\infty$ and $0\leq\sigma<(\gamma_{\min}-L)/\gamma_{\min}$ for all $k\geq0$. Let $\{\bm{x}^k\}$ be the sequence generated by the iBPDCA with (SC1) in Algorithm \ref{alg-iBPDCA}. The following statements hold.
\begin{itemize}[leftmargin=8mm]
\item[{\rm (i)}] The sequence $\{F(\bm{x}^{k})\}_{k=0}^{\infty}$ is non-increasing and $\zeta:=\lim\limits_{k\to\infty}F(\bm{x}^{k})$ exists.
\item[{\rm (ii)}] The sequence $\{\bm{x}^k\}$ is bounded.
\item[{\rm (iii)}] $\mathcal{D}_{\phi}(\bm{x}^{k+1},\,\bm{x}^{k})\to0$ and  $\mathcal{D}_{\phi}(\bm{x}^k,\,\bm{x}^{k+1})\to0$.
\item[{\rm (iv)}] The sequence $\left\{ \|\Delta^k\|^2+\big|\langle\Delta^k, \,\bm{x}^{k+1}-\bm{x}^{k}\rangle\big| + \delta_k\right\}$ is summable. Hence, $\|\Delta^k\|\to0$, $\big|\langle\Delta^k, \,\bm{x}^{k+1}-\bm{x}^{k}\rangle\big|\to0$, and $\delta_k\to0$.
\item[{\rm (v)}] It holds that
    \begin{equation*}
    \begin{aligned}
    \min\limits_{0\leq i \leq k-1}
    \big{\{}\mathcal{D}_{\phi}(\bm{x}^{i+1},\,\bm{x}^{i})\big{\}} &\leq \frac{F(\bm{x}^{0})-\zeta}{(1-\sigma)\gamma_{\min}-L}\cdot\frac{1}{k}, \\[3pt]
    \min\limits_{0\leq i \leq k-1}
    \big{\{}\mathcal{D}_{\phi}(\bm{x}^{i},\,\bm{x}^{i+1})\big{\}} &\leq \frac{F(\bm{x}^{0})-\zeta}{\gamma_{\min}}\cdot\frac{1}{k}.
    \end{aligned}
    \end{equation*}
\end{itemize}
\end{proposition}
\begin{proof}
\textit{Statement (i)}. It follows from \eqref{suffdes-iBPDCA} and (SC1) that, for any $k\geq1$,
\begin{equation}\label{suffdes-C1}
F(\bm{x}^{k+1}) - F(\bm{x}^{k})
\leq -\big((1-\sigma)\gamma_k-L\big)\mathcal{D}_{\phi}(\bm{x}^{k+1},\,\bm{x}^{k})
- \gamma_k\mathcal{D}_{\phi}(\bm{x}^{k},\,\bm{x}^{k+1})
\leq 0,
\end{equation}
where the last inequality follows from $L<\gamma_{\min}\leq\gamma_k\leq\gamma_{\max}<\infty$ and $0\leq\sigma<1-L/\gamma_{\min}\leq1-L/\gamma_k$ for all $k\geq0$. Thus, $\{F(\bm{x}^{k})\}_{k=0}^{\infty}$ is non-increasing. Moreover, since $F$ is level-bounded (by Assumption \ref{assumA}4), we have that $F^* := \mathrm{\inf}\big\{F(\bm{x})\,|\,\bm{x}\in\mathcal{Q}\big\}>-\infty$ and hence the sequence $\big\{F(\bm{x}^{k})\big\}_{k=0}^{\infty}$ is bounded from below. Then, we can conclude that $\zeta:=\lim\limits_{k\to\infty}F(\bm{x}^{k})$ exists.

\textit{Statement (ii)}. It is easy to see from \eqref{suffdes-C1} that $F(\bm{x}^k) \leq F(\bm{x}^0)$ for all $k\geq1$. Since $F$ is level-bounded (by Assumption \ref{assumA}4), the sequence $\{\bm{x}^k\}$ is then bounded.

\textit{Statement (iii)}. From \eqref{suffdes-C1}, $0\leq\sigma<1-L/\gamma_{\min}$, and $L<\gamma_{\min} \leq\gamma_k\leq\gamma_{\max}<\infty$ for all $k \geq 0$, we see that
\begin{equation*}%\label{suffbyfval}
\begin{aligned}
&~~\big((1-\sigma)\gamma_{\min}-L\big)\mathcal{D}_{\phi}(\bm{x}^{k+1},\,\bm{x}^{k})
+ \gamma_{\min}\mathcal{D}_{\phi}(\bm{x}^{k},\,\bm{x}^{k+1}) \\
&\leq \big((1-\sigma)\gamma_k-L\big)\mathcal{D}_{\phi}(\bm{x}^{k+1},\,\bm{x}^{k})
+ \gamma_k\mathcal{D}_{\phi}(\bm{x}^{k},\,\bm{x}^{k+1})
\leq F(\bm{x}^{k}) - F(\bm{x}^{k+1}).
\end{aligned}
\end{equation*}
Summing the above inequality from $k=0$ to $k=K$, we have
\begin{equation}\label{sumbdsubC1}
\big((1-\sigma)\gamma_{\min}-L\big){\textstyle\sum^{K}_{k=0}}\,
\mathcal{D}_{\phi}(\bm{x}^{k+1},\,\bm{x}^{k})
+ \gamma_{\min}{\textstyle\sum^{K}_{k=0}}\,
\mathcal{D}_{\phi}(\bm{x}^k,\,\bm{x}^{k+1})
\leq F(\bm{x}^{0}) - F(\bm{x}^{K+1}).
\end{equation}

Then, passing to the limit in the above relation induces that
\begin{equation}\label{summaC1-p}
\begin{aligned}
&~~\big((1-\sigma)\gamma_{\min}-L\big){\textstyle\sum^{\infty}_{k=0}}\,
\mathcal{D}_{\phi}(\bm{x}^{k+1},\,\bm{x}^{k})
+ \gamma_{\min}{\textstyle\sum^{\infty}_{k=0}}\,
\mathcal{D}_{\phi}(\bm{x}^k,\,\bm{x}^{k+1}) \\
&\leq F(\bm{x}^{0}) - \lim_{K\to\infty}F(\bm{x}^{K+1})
= F(\bm{x}^{0}) - \zeta
< \infty,
\end{aligned}
\end{equation}
where the equality follows from statement (i). This together with $0\leq\sigma<1-L/\gamma_{\min}$ and $\gamma_{\min}>0$ implies that ${\textstyle\sum^{\infty}_{k=0}}\,
\mathcal{D}_{\phi}(\bm{x}^{k+1},\,\bm{x}^{k})<\infty$ and ${\textstyle\sum^{\infty}_{k=0}}\,
\mathcal{D}_{\phi}(\bm{x}^k,\,\bm{x}^{k+1})<\infty$. Thus, we have that $\mathcal{D}_{\phi}(\bm{x}^{k+1},\,\bm{x}^{k}) \to 0$ and $\mathcal{D}_{\phi}(\bm{x}^k,\,\bm{x}^{k+1})\to0$.

\textit{Statement (iv)}. We see from (\ref{summaC1-p}), (SC1), and $L<\gamma_{\min}\leq\gamma_k\leq\gamma_{\max}<\infty$ that
\begin{equation*}
\sum_{k=0}^{\infty}\left[\|\Delta^k\|^2 + \big|\langle\Delta^k, \,\bm{x}^{k+1}-\bm{x}^{k}\rangle\big| + \delta_k\right]
\leq\sigma\gamma_{\max}\sum_{k=0}^{\infty}\mathcal{D}_{\phi}(\bm{x}^{k+1},\,\bm{x}^{k})
<\infty,
\end{equation*}
which readily implies that $\|\Delta^k\|\to0$, $\big|\langle\Delta^k, \,\bm{x}^{k+1}-\bm{x}^{k}\rangle\big|\to0$ and $\delta_k\to0$.

\textit{Statement (v)}. From \eqref{sumbdsubC1} and \eqref{summaC1-p}, we further have that
\begin{equation*}
\begin{aligned}
&~~k\big((1-\sigma)\gamma_{\min}-L\big)\min\limits_{0\leq i \leq k-1}\big{\{}\mathcal{D}_{\phi}(\bm{x}^{i+1},\,\bm{x}^{i})\big{\}}\\
&\leq \big((1-\sigma)\gamma_{\min}-L\big){\textstyle\sum^{k-1}_{i=0}}\,
\mathcal{D}_{\phi}(\bm{x}^{i+1},\,\bm{x}^{i})
\leq F(\bm{x}^{0}) - \zeta,
\end{aligned}
\end{equation*}
and
\begin{equation*}
k\gamma_{\min}\min\limits_{0 \leq i \leq k-1}
\big{\{}\mathcal{D}_{\phi}(\bm{x}^{i},\,\bm{x}^{i+1})\big{\}}
\leq \gamma_{\min}{\textstyle\sum^{k-1}_{i=0}}\,
\mathcal{D}_{\phi}(\bm{x}^{i},\,\bm{x}^{i+1})
\leq F(\bm{x}^{0})-\zeta.
\end{equation*}
Dividing the above relations by $k\big((1-\sigma)\gamma_{\min}-L\big)$ and $k\gamma_{\min}$ respectively, we obtain the desired results and complete the proof.
\end{proof}

\begin{proposition}[\textbf{Properties of the iBPDCA with (SC2)}]\label{pro-bdSC2}
Suppose that Assumption \ref{assumA} holds, $L<\gamma_{\min}\leq\gamma_k\leq\gamma_{\max}<\infty$ and $0\leq\sigma<(\gamma_{\min}-L)/\gamma_{\max}$ for all $k\geq0$. Let $\{\bm{x}^k\}$ be the sequence generated by the iBPDCA with (SC2) in Algorithm \ref{alg-iBPDCA}. The following statements hold.
\begin{itemize}[leftmargin=8mm]
\item[{\rm (i)}] The sequence $\big\{F(\bm{x}^{k})+\sigma\gamma_{\max}\mathcal{D}_{\phi}(\bm{x}^{k},\,\bm{x}^{k-1})\big\}_{k=0}^{\infty}$ is non-increasing, and $\widetilde{\zeta}:=\lim\limits_{k\to\infty}F(\bm{x}^{k})+\sigma\gamma_{\max}\mathcal{D}_{\phi}(\bm{x}^{k},\,\bm{x}^{k-1})$ exists.
\item[{\rm (ii)}] The sequence $\{\bm{x}^k\}$ is bounded.
\item[{\rm (iii)}] $\mathcal{D}_{\phi}(\bm{x}^{k+1},\,\bm{x}^{k})\to0$ and  $\mathcal{D}_{\phi}(\bm{x}^k,\,\bm{x}^{k+1})\to0$.
\item[{\rm (iv)}] The sequence $\left\{\|\Delta^k\|^2+\big|\langle\Delta^k, \,\bm{x}^{k+1}-\bm{x}^{k}\rangle\big| + \delta_k\right\}$ is summable. Hence, $\|\Delta^k\|\to0$, $\big|\langle\Delta^k, \,\bm{x}^{k+1}-\bm{x}^{k}\rangle\big|\to0$, and $\delta_k\to0$.
\item[{\rm (v)}] It holds that
                   \begin{equation*}
                   \begin{aligned}
                   \min\limits_{0\leq i \leq k-1}
                   \big{\{}\mathcal{D}_{\phi}(\bm{x}^{i+1},\,\bm{x}^{i})\big{\}}
                   &\leq \frac{F(\bm{x}^{0})+\sigma\gamma_{\max}\mathcal{D}_{\phi}(\bm{x}^{0},\,\bm{x}^{-1})-\widetilde{\zeta}}{\gamma_{\min}-L-\sigma\gamma_{\max}}\cdot\frac{1}{k}, \\[3pt]
                   \min\limits_{0\leq i \leq k-1}
                   \big{\{}\mathcal{D}_{\phi}(\bm{x}^{i},\,\bm{x}^{i+1})\big{\}}
                   &\leq \frac{F(\bm{x}^{0})+\sigma\gamma_{\max}\mathcal{D}_{\phi}(\bm{x}^{0},\,\bm{x}^{-1})-\widetilde{\zeta}}{\gamma_{\min}}\cdot\frac{1}{k}.
                   \end{aligned}
                   \end{equation*}
\end{itemize}
\end{proposition}
\begin{proof}
\textit{Statement (i)}. It follows from \eqref{suffdes-iBPDCA} and (SC2) that, for any $k\geq1$,
\begin{equation}\label{suffdes-C2p}
F(\bm{x}^{k+1}) - F(\bm{x}^{k})
\leq -\big(\gamma_k-L\big)\mathcal{D}_{\phi}(\bm{x}^{k+1},\,\bm{x}^{k})
- \gamma_k\mathcal{D}_{\phi}(\bm{x}^{k},\,\bm{x}^{k+1}) + \sigma\gamma_k\mathcal{D}_{\phi}(\bm{x}^{k},\,\bm{x}^{k-1}),
\end{equation}
which implies that
\begin{equation}\label{suffdes-C2}
\begin{aligned}
&~~F(\bm{x}^{k+1})+ \sigma\gamma_{\max}\mathcal{D}_{\phi}(\bm{x}^{k+1},\,\bm{x}^{k})\\
& \leq F(\bm{x}^{k})+\sigma\gamma_k\mathcal{D}_{\phi}(\bm{x}^{k},\,\bm{x}^{k-1})-(\gamma_k-L-\sigma\gamma_{\max})\mathcal{D}_{\phi}(\bm{x}^{k+1},\,\bm{x}^{k})
- \gamma_k\mathcal{D}_{\phi}(\bm{x}^{k},\,\bm{x}^{k+1})\\
& \leq F(\bm{x}^{k})+\sigma\gamma_{\max}\mathcal{D}_{\phi}(\bm{x}^{k},\,\bm{x}^{k-1}),
\end{aligned}
\end{equation}
where the last inequality follows from $L<\gamma_{\min}\leq\gamma_k\leq\gamma_{\max}<\infty$ and $0\leq\sigma<(\gamma_{\min}-L)/\gamma_{\max}\leq(\gamma_{k}-L)/\gamma_{\max}$ for all $k\geq0$. Thus, the sequence $\big\{F(\bm{x}^{k})+\sigma\gamma_{\max}\mathcal{D}_{\phi}(\bm{x}^{k},\,\bm{x}^{k-1})\big\}_{k=0}^{\infty}$ is non-increasing. Moreover, since $F$ is level-bounded (by Assumption \ref{assumA}4), we have that $F^* := \mathrm{\inf}\big\{F(\bm{x})\,|\,\bm{x}\in\mathcal{Q}\big\}>-\infty$ and hence the sequence $\big\{F(\bm{x}^{k})\big\}_{k=0}^{\infty}$ is bounded from below. This together with the nonnegativity of $\mathcal{D}_{\phi}(\bm{x}^{k},\,\bm{x}^{k-1})$ implies that  $\big\{F(\bm{x}^{k})+\sigma\gamma_{\max}\mathcal{D}_{\phi}(\bm{x}^{k},\,\bm{x}^{k-1})\big\}_{k=0}^{\infty}$ is also bounded from below. Then, we can conclude that $\widetilde{\zeta}:=\lim\limits_{k\to\infty}F(\bm{x}^{k})+\sigma\gamma_{\max}\mathcal{D}_{\phi}(\bm{x}^{k},\,\bm{x}^{k-1})$ exists.

\textit{Statement (ii)}. We obtain from \eqref{suffdes-C2} that, for any $k\geq0$,
\begin{equation*}
F(\bm{x}^{k})\leq F(\bm{x}^{k})+ \sigma\gamma_{\max}\mathcal{D}_{\phi}(\bm{x}^{k},\,\bm{x}^{k-1})\leq F(\bm{x}^{0})+\sigma\gamma_{\max}\mathcal{D}_{\phi}(\bm{x}^{0},\,\bm{x}^{-1}).
\end{equation*}
Since $F$ is level-bounded (by Assumption \ref{assumA}4), the sequence $\{\bm{x}^k\}$ is then bounded.

\textit{Statement (iii)}. Then, from \eqref{suffdes-C2}, $L<\gamma_{\min}\leq\gamma_k\leq\gamma_{\max}<\infty$ and $0\leq\sigma<(\gamma_{\min}-L)/\gamma_{\max}\leq(\gamma_{k}-L)/\gamma_{\max}$ for all $k \geq 0$, we see that
\begin{equation*}%\label{suffbyfval}
\begin{aligned}
&~~\big(\gamma_{\min}-L-\sigma\gamma_{\max}\big)\mathcal{D}_{\phi}(\bm{x}^{k+1},\,\bm{x}^{k})
+ \gamma_{\min}\mathcal{D}_{\phi}(\bm{x}^{k},\,\bm{x}^{k+1}) \\
&\leq \big({\gamma}_k-L-\sigma\gamma_{\max}\big)\mathcal{D}_{\phi}(\bm{x}^{k+1},\,\bm{x}^{k})
+ \gamma_k\mathcal{D}_{\phi}(\bm{x}^{k},\,\bm{x}^{k+1}) \\
&\leq \left(F(\bm{x}^{k})+\sigma\gamma_{\max}\mathcal{D}_{\phi}(\bm{x}^{k},\,\bm{x}^{k-1})\right)
- \left(F(\bm{x}^{k+1})+\sigma\gamma_{\max}\mathcal{D}_{\phi}(\bm{x}^{k+1},\,\bm{x}^{k})\right).
\end{aligned}
\end{equation*}
Summing the above inequality from $k=0$ to $k=K$, we have
\begin{equation}\label{sumbdsubC2}
\begin{aligned}
~~&\big(\gamma_{\min}-L-\sigma\gamma_{\max}\big){\textstyle\sum^{K}_{k=0}}\mathcal{D}_{\phi}(\bm{x}^{k+1},\,\bm{x}^{k})+\gamma_{\min}{\textstyle\sum^{K}_{k=0}}\,\mathcal{D}_{\phi}(\bm{x}^k,\,\bm{x}^{k+1})\\
& \leq \left(F(\bm{x}^{0})+\sigma\gamma_{\max}\mathcal{D}_{\phi}(\bm{x}^{0},\,\bm{x}^{-1})\right)
-
\left(F(\bm{x}^{K+1})+\sigma\gamma_{\max}\mathcal{D}_{\phi}(\bm{x}^{K+1},\,\bm{x}^{K})\right).
\end{aligned}
\end{equation}
Then, passing to the limit in the above relation induces that
\begin{equation}\label{summaC2-p}
\begin{aligned}
~~&\big(\gamma_{\min}-L-\sigma\gamma_{\max}\big){\textstyle\sum^{\infty}_{k=0}}\mathcal{D}_{\phi}(\bm{x}^{k+1},\,\bm{x}^{k})+\gamma_{\min}{\textstyle\sum^{\infty}_{k=0}}\,\mathcal{D}_{\phi}(\bm{x}^k,\,\bm{x}^{k+1})\\
&\leq F(\bm{x}^{0})+\sigma\gamma_{\max}\mathcal{D}_{\phi}(\bm{x}^{0},\,\bm{x}^{-1})
- \lim\limits_{K\to\infty}\left\{F(\bm{x}^{K+1})
+ \sigma\gamma_{\max}\mathcal{D}_{\phi}(\bm{x}^{K+1},\,\bm{x}^{K})\right\} \\
& = F(\bm{x}^{0})+\sigma\gamma_{\max}\mathcal{D}_{\phi}(\bm{x}^{0},\,\bm{x}^{-1})-
\widetilde{\zeta}
<\infty,
\end{aligned}
\end{equation}
where the equality follows from statement (i). This together with $0\leq\sigma<(\gamma_{\min}-L)/\gamma_{\max}$ and $\gamma_{\min}>0$ implies that ${\textstyle\sum^{\infty}_{k=0}}\mathcal{D}_{\phi}(\bm{x}^{k+1},\,\bm{x}^{k})<\infty$ and ${\textstyle\sum^{\infty}_{k=0}}\,\mathcal{D}_{\phi}(\bm{x}^k,\,\bm{x}^{k+1})<\infty$. Thus, we have that $\mathcal{D}_{\phi}(\bm{x}^{k+1},\,\bm{x}^{k}) \to 0$ and $\mathcal{D}_{\phi}(\bm{x}^k,\,\bm{x}^{k+1})\to0$.

\textit{Statement (iv)}. From \eqref{summaC2-p}, (SC2), and $L<\gamma_{\min}\leq\gamma_k\leq\gamma_{\max}<\infty$ for all $k \geq 0$, we have that
\begin{equation*}
\sum_{k=0}^{\infty}\left[\|\Delta^k\|^2 + \big|\langle\Delta^k, \,\bm{x}^{k+1}-\bm{x}^{k}\rangle\big| + \delta_k\right]
\leq\sigma\gamma_{\max}\sum_{k=0}^{\infty}\mathcal{D}_{\phi}(\bm{x}^{k},\,\bm{x}^{k-1})
<\infty,
\end{equation*}
which readily implies that $\|\Delta^k\|\to0$, $\big|\langle\Delta^k, \,\bm{x}^{k+1}-\bm{x}^{k}\rangle\big|\to0$ and $\delta_k\to0$.

\textit{Statement (v)}. From \eqref{sumbdsubC2} and \eqref{summaC2-p}, we further have that
\begin{equation*}
\begin{aligned}
&~~k(\gamma_{\min}-L-\sigma\gamma_{\max})\min\limits_{0\leq i\leq k-1}\left\{\mathcal{D}_{\phi}(\bm{x}^{i+1},\,\bm{x}^{i})\right\}\\
&\leq (\gamma_{\min}-L-\sigma\gamma_{\max}){\textstyle\sum^{k-1}_{i=0}}\,
\mathcal{D}_{\phi}(\bm{x}^{i+1},\,\bm{x}^{i})
\leq F(\bm{x}^{0})+\sigma\gamma_{\max}\mathcal{D}_{\phi}(\bm{x}^{0},\,\bm{x}^{-1}) - \widetilde{\zeta},
\end{aligned}
\end{equation*}
and
\begin{equation*}
k\gamma_{\min}\min\limits_{0 \leq i \leq k-1}
\left\{\mathcal{D}_{\phi}(\bm{x}^{i},\,\bm{x}^{i+1})\right\}
\leq \gamma_{\min}{\textstyle\sum^{k-1}_{i=0}}\,
\mathcal{D}_{\phi}(\bm{x}^{i},\,\bm{x}^{i+1})
\leq F(\bm{x}^{0})+\sigma\gamma_{\max}\mathcal{D}_{\phi}(\bm{x}^{0},\,\bm{x}^{-1})-\widetilde{\zeta}.
\end{equation*}
Dividing the above relations by $k(\gamma_{\min}-L-\sigma\gamma_{\max})$ and $k\gamma_{\min}$ respectively, we obtain the desired results and complete the proof.
\end{proof}

%%%%%%%%%%%%%%%%%%%%%%%%%%%%%%%%%%%%%%%%%%%%%%%%%%%
\section{Convergence analysis}\label{sec-conv}

In this section, we attempt to establish the convergence of our iBPDCA in Algorithm \ref{alg-iBPDCA}. To this end, we need to make additional assumptions as follows.

\begin{assumption}\label{assumB}
$f$ is continuously differentiable on an open set $\mathcal{N}_f$ that contains $\mathrm{dom}\,P_1\cap\mathcal{Q}$.
\end{assumption}

Assumption \ref{assumB} says that $f$ is continuously differentiable at any point in $\mathrm{dom}\,P_1\cap\mathcal{Q}$. Under Assumptions \ref{assumA} and \ref{assumB}, it is routine to show that any local minimizer $\widehat{\bm{x}}$ of problem \eqref{DCpro} satisfies the following relations:
\begin{equation}\label{stationary-point}
\begin{aligned}
0 \in \partial F(\widehat{\bm{x}})
= \partial\left(\iota_{\mathcal{Q}}+P_1-P_2\right)(\widehat{\bm{x}})
+ \nabla f(\widehat{\bm{x}})
\subseteq \partial\left(\iota_{\mathcal{Q}}+P_1\right)(\widehat{\bm{x}})
- \partial P_2(\widehat{\bm{x}}) + \nabla f(\widehat{\bm{x}}),
\end{aligned}
\end{equation}
where the first inclusion follows from the generalized Fermat's rule \cite[Theorem~10.1]{rw1998variational}, the equality follows from \cite[Exercise~8.8(c)]{rw1998variational}, and the last inclusion follows from \cite[Corollary~3.4]{mny2006frechet}. Thus, in this paper, we follow \cite[Definition 4.1]{wcp2018proximal} to say that $\bm{x}^*$ is a \textit{stationary point} of problem \eqref{DCpro} if $\bm{x}^*$ satisfies
\begin{equation}\label{defsta}
0 \in \partial\left(\iota_{\mathcal{Q}}+P_1\right)(\bm{x}^*)
- \partial P_2(\bm{x}^*) + \nabla f(\bm{x}^*).
\end{equation}
Moreover, if $P_2$ is continuously differentiable near $\bm{x}^*$, the last inclusion in \eqref{stationary-point} can be shown to hold as an equality at $\bm{x}^*$. In this case, we follow \cite[Section 3]{cps2018composite} to say that $\bm{x}^*$ is \red{an} \textit{$\ell$(imiting) stationary point} of problem \eqref{DCpro} if $\bm{x}^*$ satisfies
\begin{equation}\label{deflimsta}
0 \in \partial\left(\iota_{\mathcal{Q}}+P_1\right)(\bm{x}^*)
- \nabla P_2(\bm{x}^*) + \nabla f(\bm{x}^*).
\end{equation}
Note that, an $\ell$-stationary point, as defined in \eqref{deflimsta}, can also be shown to be a $d$(irectional)-stationary point; see, for example, \cite[Proposition~1]{lz2019nonmonotone} and \cite[Proposition~1]{lzs2019enhanced}. In the following, we denote the set of all stationary points of problem \eqref{DCpro} by $\mathcal{S}$.

\begin{assumption}\label{assumC}
Assume that there exists an open set $\mathcal{N}_{\phi}$ such that $\mathrm{dom}\,P_1\cap\mathcal{Q}\subseteq\mathcal{N}_{\phi}\subseteq\mathrm{int}\,\mathrm{dom}\,\phi$ and the kernel function $\phi$ has the following properties on $\mathcal{N}_{\phi}$.
\begin{itemize}[leftmargin=1cm]
%\item[{\bf C1.}] $\mathrm{dom}\phi = \overline{\mathrm{dom}}\phi = \mathcal{Q}$, i.e., the domain of $\phi$ is closed.
\item[{\bf C1.}] $\phi$ is $\mu$-strongly convex on $\mathcal{N}_{\phi}$.

\item[{\bf C2.}] $\phi$ is differentiable on $\mathcal{N}_{\phi}$ and $\nabla \phi$ is Lipschtiz continuous on any bounded subset of $\mathcal{N}_{\phi}$.
\end{itemize}
\end{assumption}

This type of assumption is commonly adopted to establish the convergence of the Bregman-type method in the nonconvex setting (see, for example,
\cite[Assumption D]{bstv2017first} and \cite[Assumption 4]{tft2022new}). It can be satisfied by, for example, $\phi(\bm{x})=\frac{1}{2}\|\bm{x}\|^2$ and $\phi(\bm{x})=\frac{1}{2}\|\bm{x}\|^2+\frac{1}{2}\|A\bm{x}\|^2$, both of which will be used in our numerical section. With these additional assumptions, we are now able to establish the global subsequential convergence of our iBPDCA as follows.
%the iterates to a limiting critical point of the problem \eqref{DCpro}.

\begin{theorem}\label{thm-subseq}
Suppose that Assumptions \ref{assumA}, \ref{assumB}, \ref{assumC} hold and $L<\gamma_{\min}\leq\gamma_k\leq\gamma_{\max}<\infty$ for all $k\geq0$. Moreover, suppose that $0\leq\sigma<(\gamma_{\min}-L)/\gamma_{\min}$ for (SC1), and that $0\leq\sigma<(\gamma_{\min}-L)/\gamma_{\max}$ for (SC2). Let $\{\bm{x}^k\}$ be the sequence generated by the iBPDCA in Algorithm \ref{alg-iBPDCA} with (SC1) or (SC2). Then, any cluster point $\bm{x}^*$ of $\{\bm{x}^k\}$ is a stationary point of problem \eqref{DCpro}.
\end{theorem}
\begin{proof}
First, it follows from Proposition \ref{pro-bdSC1}(ii) and Proposition \ref{pro-bdSC2}(ii) that the sequence $\{\bm{x}^k\}$ is bounded and thus has at least one cluster point. Suppose that $\bm{x}^*$ is a cluster point of $\{\bm{x}^k\}$ and $\{\bm{x}^{k_i}\}$ is a convergent subsequence such that $\lim\limits_{i\to\infty}\bm{x}^{k_i}=\bm{x}^*$. Since $\bm{x}^{k_i}\in \mathrm{dom}\,P_1\cap\mathrm{int}\,\mathrm{dom}\,\phi$ for all $k_i\geq0$, $P_1$ is a proper closed convex function and $\overline{\mathrm{dom}\,\phi}=Q$ (by Assumption \ref{assumA}1), then $\bm{x}^*\in\mathrm{dom}\,P_1\cap\mathcal{Q}$. Moreover, it is clear that  $\bm{x}^{k_i}\in\mathcal{N}_{\phi}$ for all $k_i\geq0$, where the open set $\mathcal{N}_{\phi}$ is given in Assumption \ref{assumC}. It then follows from the strong convexity of $\phi$ on $\mathcal{N}_{\phi}$ (by Assumption \ref{assumC}1) that, for all $k_i \geq 0$,
\begin{equation*}
\mathcal{D}_{\phi}({\bm{x}}^{k_i+1},\,\bm{x}^{k_i})
\geq \frac{\mu}{2}\|{\bm{x}}^{k_i+1}-\bm{x}^{k_i}\|^2.
\end{equation*}
Recall from Proposition \ref{pro-bdSC1}(iii) and Proposition \ref{pro-bdSC2}(iii) that $\mathcal{D}_{\phi}({\bm{x}}^{k+1}, \,\bm{x}^{k}) \to 0$. This yields
\begin{equation}\label{succlim2SC1}
\bm{x}^{k_i+1} - \bm{x}^{k_i} \to 0.
\end{equation}
Moreover, since $L<\gamma_{\min}\leq\gamma_k\leq\gamma_{\max}<\infty$ for all $k\geq0$ and $\{\bm{x}^{k_i}\}$ is bounded and contained in $\mathcal{N}_{\phi}$, it then follows from Assumption \ref{assumC}2 that
\begin{equation*}
\begin{aligned}
\gamma_{k_i}\|\nabla\phi(\bm{x}^{k_i+1})-\nabla \phi(\bm{x}^{k_i})\|
\leq \gamma_{\max}\,\beta\|\bm{x}^{k_i+1}-\bm{x}^{k_i}\|, \quad \forall \,k_i \geq 0,
\end{aligned}
\end{equation*}
where $\beta$ is a Lipschitz constant of $\nabla\phi$ on a certain bounded subset that contains $\{\bm{x}^{k_i}\}$. This relation and \eqref{succlim2SC1} imply that
\begin{equation}\label{succlim3SC1}
\gamma_{k_i}\big(\nabla\phi(\bm{x}^{k_i+1}) - \nabla \phi(\bm{x}^{k_i})\big) \to 0.
\end{equation}
In addition, $\overline{\mathrm{dom}}\,\phi=\mathcal{Q}$ in Assumption \ref{assumA}1 implies $\mathrm{int}\,\mathrm{dom}\,\phi=\mathrm{int}\,\mathcal{Q}$ due to the convexity of $\mathrm{dom}\,\phi$ and \cite[Proposition 3.36(iii)]{bc2011convex}. We then have that $\bm{x}^k\in\mathrm{int}\,\mathrm{dom}\,\phi=\mathrm{int}\,\mathcal{Q}$ and hence $\partial\iota_{\mathcal{Q}}({\bm{x}}^{k})=\{0\}$ for all $k\geq0$. Using this fact and condition \eqref{iBPDCA-inexcond}, we further obtain that
\begin{equation}\label{suboptki-iBPDCA2SC1}
\begin{aligned}
-\gamma_{k_i}(\nabla \phi(\bm{x}^{k_i+1})-\nabla \phi(\bm{x}^{k_i}))+ \Delta^{k_i} &\in \partial_{\delta_{k_i}}P_1({\bm{x}}^{k_i+1}) - \bm{\xi}^{k_i} + \nabla f(\bm{x}^{k_i}) \\
&= \partial\iota_{\mathcal{Q}}({\bm{x}}^{k_i+1})
+ \partial_{\delta_{k_i}}P_1({\bm{x}}^{k_i+1}) - \bm{\xi}^{k_i}
+ \nabla f(\bm{x}^{k_i})  \\
&\subseteq \partial_{\delta_{k_i}}\left(\iota_{\mathcal{Q}}+P_1\right)({\bm{x}}^{k_i+1})
- \bm{\xi}^{k_i} + \nabla f(\bm{x}^{k_i}),
\end{aligned}
\end{equation}
where the last inclusion can be verified by the definition of the $\varepsilon$-subdifferential of $\iota_{\mathcal{Q}}+P_1$ at ${\bm{x}}^{k_i+1}$. Note that the sequence $\{\bm{\xi}^{k_i}\}$ is bounded due to the continuity and convexity of $P_2$ and the boundedness of $\{{\bm{x}}^{k_i}\}$. Thus, by passing to a further subsequence if necessary, we may assume without loss of generality that $\bm{\xi}^*:=\lim\limits_{i\to\infty}\bm{\xi}^{k_i}$ exists and $\bm{\xi}^*\in\partial P_2(\bm{x}^*)$ due to the closedness of $\partial P_2$. In addition, we have $\delta_k\to0$ and $\|\Delta^k\|\to0$ from Proposition \ref{pro-bdSC1}(iv) for (SC1) and Proposition \ref{pro-bdSC2}(iv) for (SC2). Then, passing to the limit in \eqref{suboptki-iBPDCA2SC1}, and invoking \eqref{succlim2SC1}, \eqref{succlim3SC1}, $\|\Delta^k\| \to 0$, Assumption \ref{assumB} with $\{\bm{x}^{k_i}\}$ and $\bm{x}^*$ contained in $\mathcal{N}_{f}$, and the outer semicontinuity of $\partial_{\delta_k}\left(\iota_{\mathcal{Q}}+P_1\right)$ (see, for example, \cite[Proposition 4.1.1]{hl1993convex}) with $\delta_k\to0$, we obtain that
\begin{equation*}
0 \in \partial\left(\iota_{\mathcal{Q}}+P_1\right)(\bm{x}^*) - \partial P_2(\bm{x}^*) + \nabla f(\bm{x}^*),
\end{equation*}
which implies that $\bm{x}^*$ is a stationary point of problem \eqref{DCpro} and completes the proof.
\end{proof}

%%%%%%%%%%%%%%%%%%%%%%%%%%%%%%%%%%%%%%%%%%%%%%%%%%%%%%%
\subsection{Global convergence of the whole sequence}

We next develop the global convergence of the whole sequence generated by the iBPDCA in Algorithm \ref{alg-iBPDCA}. To this end, we first characterize the sequence of objective values along $\{\bm{x}^k\}$ in the following proposition.

\begin{proposition}\label{pro-fval}
Suppose that Assumptions \ref{assumA}, \ref{assumB}, \ref{assumC} hold and $L<\gamma_{\min}\leq\gamma_k\leq\gamma_{\max}<\infty$ for all $k\geq0$. Moreover, suppose that $0\leq\sigma<(\gamma_{\min}-L)/\gamma_{\min}$ for (SC1), and that $0\leq\sigma<(\gamma_{\min}-L)/\gamma_{\max}$ for (SC2). Let $\{\bm{x}^k\}$ be the sequence generated by the iBPDCA in Algorithm \ref{alg-iBPDCA}. Then, the following statements hold.
\begin{itemize}
\item[{\rm(i)}] $\zeta:=\lim\limits_{k\to\infty}F({\bm{x}}^k)$ exists.
\item[{\rm(ii)}] $F \equiv \zeta$ on $\Omega$, where $\Omega$ is the set of all cluster points of $\{\bm{x}^{k}\}$.
\end{itemize}
\end{proposition}
\begin{proof}
\textit{Statement (i)}. For (SC1), it has been proved in Proposition \ref{pro-bdSC1}(i).
For (SC2), it can be easily derived from Proposition \ref{pro-bdSC2}(i) and Proposition \ref{pro-bdSC2}(iii) that
\begin{equation*}
\begin{aligned}
\widetilde{\zeta}
&= \lim_{k\to\infty}\big( F({\bm{x}}^k) + \sigma\gamma_{\max}\mathcal{D}_{\phi}({\bm{x}}^k,\,\bm{x}^{k-1})\big) \\
&= \lim_{k\to\infty} F({\bm{x}}^k) + \sigma\gamma_{\max}\lim_{k\to\infty}\mathcal{D}_{\phi}({\bm{x}}^k,\,\bm{x}^{k-1})=\lim_{k\to\infty} F({\bm{x}}^k),
\end{aligned}
\end{equation*}
which implies that $\lim\limits_{k\to\infty} F({\bm{x}}^k)$ exists. For simplicity, we also use $\zeta$ to denote this limit in the rest of this paper.

\textit{Statement (ii)}. We first note from the boundedness of $\{{\bm{x}}^k\}$ that $\emptyset \neq \Omega \subseteq\mathcal{S}$, where $\mathcal{S}$ is the set of all stationary points of problem \eqref{DCpro}. Take any $\bm{x}^{\infty}\in\Omega$ and let $\{{\bm{x}}^{k_i}\}$ be a convergent subsequence such that $\lim\limits_{i\to\infty} {\bm{x}}^{k_i} = \bm{x}^{\infty}$. Since $\bm{x}^{k_i} \in \mathrm{dom}\,P_1\cap\mathrm{int}\,\mathrm{dom}\,\phi$ for all $k_i\geq0$, $P_1$ is a proper closed convex function and $\overline{\mathrm{dom}\,\phi}=Q$ (by Assumption \ref{assumA}1), then $\bm{x}^{\infty}\in\mathrm{dom}\,P_1\cap\mathcal{Q}$, and $\nabla f(\bm{x}^{\infty})$ and $\nabla \phi(\bm{x}^{\infty})$ are well-defined from Assumptions \ref{assumB} and \ref{assumC}2, respectively. Moreover, it is easy to see that $\bm{x}^{k_i},\,\bm{x}^{\infty}\in\mathcal{X}\cap\mathcal{N}_{\phi}\subseteq\mathrm{int}\,\mathrm{dom}\,\phi$, where the sets $\mathcal{X}$ and $\mathcal{N}_{\phi}$ are given in Assumptions \ref{assumA} and \ref{assumC}, respectively.
%Moreover, for any $\varepsilon$, there must exist a $K_0>0$ such that $\mathrm{dist}\left(\bm{x}^{k_i}, \,\mathrm{dom}\,P_1\cap\mathcal{Q}\right)<\varepsilon$ and $\bm{x}^{k_i}\in\mathcal{X}\cap\mathcal{N}_{f}\cap\mathcal{N}_{\phi}$ for all $k_i\geq K_0$, where the open sets $\mathcal{X}$, $\mathcal{N}_{f}$ and $\mathcal{N}_{\phi}$ are given in Assumptions \ref{assumA}, \ref{assumB} and \ref{assumC}, respectively.

From condition \eqref{iBPDCA-inexcond}, for all $k_i\geq 1$, there exists a $\bm{d}^{k_i} \in \partial_{\delta_{k_i-1}}P_1({\bm{x}}^{k_i})$ such that
\begin{equation*}%\label{optcondsubSC1}
\Delta^{k_i-1} = \bm{d}^{k_i}
+ \nabla f(\bm{x}^{k_i-1}) - \bm{\xi}^{k_i-1}
+ \gamma_{k_i-1}\big(\nabla \phi(\bm{x}^{k_i}) - \nabla \phi(\bm{x}^{k_i-1})\big).
\end{equation*}
Using this equality and the definition of $\partial_{\delta_{k_i-1}}P_1$, we see that, for all $k_i\geq 1$,
\begin{equation*}
\begin{aligned}
& P_1(\bm{x}^{\infty})
\geq  P_1({\bm{x}}^{k_i}) + \langle \bm{d}^{k_i}, \,\bm{x}^{\infty} - {\bm{x}}^{k_i} \rangle
- \delta_{k_i-1}  \\
&= P_1({\bm{x}}^{k_i}) + \langle -\nabla f(\bm{x}^{k_i-1}) + \bm{\xi}^{k_i-1} - \gamma_{k_i-1}\big(\nabla \phi(\bm{x}^{k_i})-\nabla \phi(\bm{x}^{k_i-1})\big) + \Delta^{k_i-1}, \,\bm{x}^{\infty} - {\bm{x}}^{k_i}\rangle
- \delta_{k_i-1},  \\
\end{aligned}
\end{equation*}
which implies that
\begin{equation}\label{ineq1subSC1}
\begin{aligned}
P_1({\bm{x}}^{k_i})
&\leq P_1(\bm{x}^{\infty})
+ \langle\bm{\xi}^{k_i-1}, \,\bm{x}^{k_i} - \bm{x}^{k_i-1}\rangle
+ \langle\bm{\xi}^{k_i-1}, \,\bm{x}^{k_i-1} - \bm{x}^{\infty}\rangle \\
&\quad - \langle \nabla f(\bm{x}^{k_i-1}) + \gamma_{k_i-1}\big(\nabla \phi(\bm{x}^{k_i})-\nabla \phi(\bm{x}^{k_i-1})\big) - \Delta^{k_i-1}, \,{\bm{x}}^{k_i} - \bm{x}^{\infty}\rangle
+ \delta_{k_i-1}.
\end{aligned}
\end{equation}
%Moreover, note that
%\begin{equation}\label{ineq1tmp2sub}
%\begin{aligned}
%&\quad\langle \nabla \phi(\bm{x}^{k_i})-\nabla \phi(\bm{x}^{\infty}), \,\bm{x}^{k_i}-\bm{x}^{\infty}\rangle \\
%&=\big{(}\phi(\bm{x}^{\infty})-\phi(\bm{x}^{k_i})-\langle \nabla \phi(\bm{x}^{k_i}), \,\bm{x}^{\infty} - \bm{x}^{k_i} \rangle\big{)}
%+ \big{(}\phi(\bm{x}^{k_i})-\phi(\bm{x}^{\infty})-\langle \nabla \phi(\bm{x}^{\infty}), \,\bm{x}^{k_i} - \bm{x}^{\infty} \rangle\big{)} \\
%&=\mathcal{D}_{\phi}(\bm{x}^{\infty},\,\bm{x}^{k_i}) + \mathcal{D}_{\phi}(\bm{x}^{k_i},\,\bm{x}^{\infty})
%\geq (1+\alpha(\phi))\mathcal{D}_{\phi}(\bm{x}^{k_i},\,\bm{x}^{\infty}),
%\end{aligned}
%\end{equation}
%where $\alpha(\phi)$ is a symmetric measure belonging to $[0,1]$ (see \cite[Definition 2]{bbt2017descent}). Then, combining \eqref{ineq1subSC1} and \eqref{ineq1tmp2sub}, we have that
%\begin{equation}\label{ineq1sub}
%\begin{aligned}
%\quad P_1(\bm{x}^{k_i}) &\leq P_1(\bm{x}^{\infty}) + \langle \Delta^{k_i-1}, \,\bm{x}^{k_i}-\bm{x}^{\infty}\rangle + \langle -\nabla f(\bm{x}^{k_i-1}), \,\bm{x}^{k_i}-\bm{x}^{\infty}\rangle + \langle \bm{\xi}^{k_i-1}, \,\bm{x}^{k_i}-\bm{x}^{\infty}\rangle \\
%& \qquad + \gamma_{k_i-1}\langle \nabla \phi(\bm{x}^{k_i-1})-\nabla \phi(\bm{x}^{\infty}), \,\bm{x}^{k_i}-\bm{x}^{\infty}\rangle - (1+\alpha(\phi))\gamma_{k_i-1}\,\mathcal{D}_{\phi}(\bm{x}^{k_i},\,\bm{x}^{\infty}).
%\end{aligned}
%\end{equation}
On the other hand, it follows from the convexity of $P_2$ and $\bm{\xi}^{k}\in\partial P_2(\bm{x}^k)$ that, for all $k_i\geq 1$,
\begin{equation}\label{ineq2subSC1}
%P_2({\bm{x}}^{k_i}) \geq P_2({\bm{x}}^{k_i-1}) + \langle\bm{\xi}^{k_i-1}, \,{\bm{x}}^{k_i}-{\bm{x}}^{k_i-1}\rangle \Longleftrightarrow
-P_2({\bm{x}}^{k_i}) \leq -P_2({\bm{x}}^{k_i-1}) - \langle\bm{\xi}^{k_i-1}, \,{\bm{x}}^{k_i}-{\bm{x}}^{k_i-1}\rangle.
\end{equation}
Moreover, since $(f,\phi)$ is ${L}$-smooth adaptable restricted on $\mathcal{X} \supseteq \mathrm{dom}\,P_1\cap\mathcal{Q}$ (by Assumption \ref{assumA}3) and $\bm{x}^{k_i},\,\bm{x}^{\infty}\in\mathcal{X}\cap\mathcal{N}_{\phi}\subseteq\mathrm{int}\,\mathrm{dom}\,\phi$, it holds that
\begin{equation}\label{ineq3subSC1}
f(\bm{x}^{k_i}) \leq f(\bm{x}^{\infty}) + \langle \nabla f(\bm{x}^{\infty}), \,\bm{x}^{k_i}-\bm{x}^{\infty}\rangle 
+ {L}\mathcal{D}_{\phi}(\bm{x}^{k_i},\,\bm{x}^{\infty}).
\end{equation}
Then, summing \eqref{ineq1subSC1}, \eqref{ineq2subSC1} and \eqref{ineq3subSC1}, together with $\bm{x}^{k_i},\,\bm{x}^{\infty}\in\mathcal{Q}$, we obtain that, for all $k_i\geq 1$,
\begin{equation*}
\begin{aligned}
&\quad F({\bm{x}}^{k_i})
:= \iota_{\mathcal{Q}}({\bm{x}}^{k_i}) + P_1({\bm{x}}^{k_i}) - P_2({\bm{x}}^{k_i}) + f({\bm{x}}^{k_i}) \\
&\leq F(\bm{x}^{\infty}) + P_2(\bm{x}^{\infty}) -P_2({\bm{x}}^{k_i-1})
+ \langle \bm{\xi}^{k_i-1}, \,{\bm{x}}^{k_i-1}-{\bm{x}}^{\infty}\rangle \\
&\quad
- \langle \nabla f(\bm{x}^{k_i-1})-\nabla f(\bm{x}^{\infty}) + \gamma_{k_i-1}\big(\nabla \phi(\bm{x}^{k_i})-\nabla \phi(\bm{x}^{k_i-1})\big) 
- \Delta^{k_i-1}, \,{\bm{x}}^{k_i} - \bm{x}^{\infty}\rangle
+ \delta_{k_i-1}.
\end{aligned}
\end{equation*}
Since ${\bm{x}}^{k_i}-{\bm{x}}^{k_i-1}\to0$ from \eqref{succlim2SC1}, then $\lim\limits_{i\to\infty} {\bm{x}}^{k_i-1} = \bm{x}^{\infty}$. This together with Assumptions \ref{assumB} and \ref{assumC}2,  $L<\gamma_{\min}\leq\gamma_k\leq\gamma_{\max}<\infty$, the boundedness of $\{{\bm{x}}^{k_i}\}$, and $\|\Delta^k\|\to0$ (by Propositions \ref{pro-bdSC1}(iv) and \ref{pro-bdSC2}(iv)) implies that
\begin{equation*}
\langle \nabla f(\bm{x}^{k_i-1}) - \nabla f(\bm{x}^{\infty})  + \gamma_{k_i-1}\big(\nabla \phi(\bm{x}^{k_i})-\nabla \phi(\bm{x}^{k_i-1})\big) - \Delta^{k_i-1}, \,{\bm{x}}^{k_i} - \bm{x}^{\infty}\rangle
\to 0.
\end{equation*}
In addition, it follows from the continuity and convexity of $P_2$ and $\lim\limits_{i\to\infty}{\bm{x}}^{k_i-1}=\bm{x}^{\infty}$ (by \eqref{succlim2SC1}) that the sequence $\{\bm{\xi}^{k_i-1}\}$ is bounded, $\langle \bm{\xi}^{k_i-1}, \,{\bm{x}}^{k_i-1}-{\bm{x}}^{\infty}\rangle\to0$, and $P_2(\bm{x}^{\infty}) -P_2({\bm{x}}^{k_i-1})\to0$. Using these facts and invoking $\delta_k\to0$, we have that
\begin{equation*}
\zeta = \lim\limits_{i\to\infty} F({\bm{x}}^{k_i}) \leq \limsup\limits_{i\to\infty} F({\bm{x}}^{k_i}) \leq F(\bm{x}^{\infty}).
\end{equation*}
Finally, since $F$ is lower semicontinuous, we also have that
\begin{equation*}
F(\bm{x}^{\infty}) \leq \liminf\limits_{i\to\infty} F({\bm{x}}^{k_i}) \leq \lim\limits_{i\to\infty} F({\bm{x}}^{k_i}) = \zeta.
\end{equation*}
Consequently, $F(\bm{x}^{\infty})=\lim\limits_{i\to\infty} F({\bm{x}}^{k_i}) = \zeta$. Since $\bm{x}^{\infty}\in\Omega$ is arbitrary, we conclude that $F$ is constant on $\Omega$. This completes the proof.
\end{proof}

To establish the global sequential convergence, we further make the following additional assumptions.

\begin{assumption}\label{assumD}
Assume that $f$ and $P_2$ satisfy the following properties.
\begin{itemize}[leftmargin=1cm]
\item[{\bf D1.}] $\nabla f$ is Lipschtiz continuous on any bounded subset of the open set $\mathcal{N}_f$ given in Assumption \ref{assumB}.

\item[{\bf D2.}] $P_2$ is continuously differentiable on an open set $\mathcal{N}_{\mathcal{S}}$ that contains $\mathcal{S}$, where $\mathcal{S}$ is the set of all stationary points of problem \eqref{DCpro}. Moreover, $\nabla P_2$ is locally Lipschitz continuous on $\mathcal{N}_{\mathcal{S}}$.
\end{itemize}
\end{assumption}

Assumptions \ref{assumD}1 and \ref{assumD}2 are common technical assumptions employed in the global convergence analysis of existing Bregman-type algorithms or proximal DC-type algorithms; see, for example, \cite{bstv2017first,tft2022new,wcp2018proximal}.

\begin{assumption}\label{assumE}
Assume that there exists a $K_0>0$ such that $\delta_{k}\equiv0$ for all $k \geq K_0$ in \eqref{iBPDCA-inexcond}.
\end{assumption}

Note that, under Assumption \ref{assumE}, no error $\delta_k$ is allowed in the computation of $\partial P_1$ after finitely many iterations, while a non-zero error term $\Delta^k$ remains permissible in \eqref{iBPDCA-inexcond}. Consequently, even with this assumption, Algorithm \ref{alg-iBPDCA} retains its inexact nature compared to the exact BPDCA studied in \cite{tft2022new}, and thus the convergence of the whole sequence must still be established accordingly. While Assumption \ref{assumE} looks restrictive, it is indeed achievable in many application problems; see examples in \cite{twst2020sparse,zts2020learning} and the $\ell_{1-2}$ regularized least squares problem in our numerical section \ref{sec-num-l12reg}. Nevertheless, Assumption \ref{assumE} could limit the applicability of our iBPDCA when addressing a problem with complex constraints, for example, the constrained $\ell_{1-2}$ sparse optimization problem in our numerical section \ref{sec-num-l12con}. Investigating the possibility of eliminating this requirement when establishing the convergence of the whole sequence remains an interesting direction for our future research.
%It would be interesting to explore the possibility of eliminating such a requirement when establishing the global sequential convergence. We will leave it as a future research topic.

With the above preparations, we are now ready to establish the global convergence of the whole sequence generated by our iBPDCA with (SC1). Our analysis follows a similar line of arguments as in \cite{tft2022new,wcp2018proximal} for exact proximal DC-type methods, but is more involved due to the presence of errors in our algorithm.

\begin{theorem}[\textbf{Global sequential convergence of the iBPDCA with (SC1)}]\label{thm-wholeseqSC1}
Suppose that Assumptions \ref{assumA}, \ref{assumB}, \ref{assumC}, \ref{assumD}, \ref{assumE} hold, $L<\gamma_{\min}\leq\gamma_k\leq\gamma_{\max}<\infty$, $0\leq\sigma<(\gamma_{\min}-L)/\gamma_{\min}$ for all $k\geq0$, and $\sum_{k=0}^{\infty}\|\Delta^k\|<\infty$. Let $\{\bm{x}^k\}$ be the sequence generated by the iBPDCA with (SC1) in Algorithm \ref{alg-iBPDCA}. If $F$ is a KL function, then the sequence $\{\bm{x}^k\}$ converges to an $\ell$-stationary point of problem \eqref{DCpro}.
\end{theorem}
\begin{proof}
In view of Theorem \ref{thm-subseq}, we have that $\lim\limits_{k\to\infty}\mathrm{dist}({\bm{x}}^k, \,\Omega)=0$ with $\Omega\subseteq\mathcal{S}$, where $\Omega$ is the set of all cluster points of $\{\bm{x}^{k}\}$ and $\mathcal{S}$ is the set of all stationary points of problem \eqref{DCpro}. Note that, under Assumption \ref{assumD}2, a stationary point as defined in \eqref{defsta} is inherently an $\ell$-stationary point as defined in \eqref{deflimsta}. Thus, we only need to show that the sequence is convergent.

To this end, we first note from $\Omega\subseteq\mathrm{dom}\,P_1\cap\,\mathcal{Q}$, $\Omega\subseteq\mathcal{S}$, and Assumptions \ref{assumB}, \ref{assumC}, \ref{assumD} that, for any $\tilde{\varepsilon}>0$, there exists $K_1>0$ such that $\mathrm{dist}({\bm{x}}^k, \,\Omega)<\tilde{\varepsilon}$ and $\bm{x}^k\in\mathcal{N}_{\phi}\cap\mathcal{N}_f\cap\mathcal{N}_{\mathcal{S}}$ for all $k \geq K_1$. Moreover, since $\Omega$ is compact (due to the boundedness of $\{{\bm{x}}^k\}$), by shrinking $\tilde{\varepsilon}$ if necessary, we may assume without loss of generality that $\nabla f$, $\nabla P_2$ and $\nabla\phi$ are globally Lipschitz continuous on the bounded set $\mathcal{N}_{\tilde{\varepsilon}}:=\big\{\bm{x}\in\mathrm{dom}\,P_1\cap\mathcal{Q}\cap\mathcal{N}_{\phi}\cap\mathcal{N}_f\cap\mathcal{N}_{\mathcal{S}}:\mathrm{dist}(\bm{x}, \,\Omega)<\tilde{\varepsilon}\big\}$.

We next consider the subdifferential of $F$ at ${\bm{x}}^k$ for any $k\geq \max\{K_0,K_1\}+1$, where $K_0$ is given in Assumption \ref{assumE}. Note that, for all $k \geq K_1$, ${\bm{x}}^k\in\mathcal{N}_{\mathcal{S}}$ and $P_2$ is continuously differentiable on $\mathcal{N}_{\mathcal{S}}$ (by Assumption \ref{assumD}2). Then, for all $k \geq \max\{K_0,K_1\}+1$,
\begin{equation*}
\partial F({\bm{x}}^k) = \partial\left(\iota_{\mathcal{Q}}+P_1\right)({\bm{x}}^k) - \nabla P_2({\bm{x}}^k) + \nabla f({\bm{x}}^k).
\end{equation*}
Now, recall from condition \eqref{iBPDCA-inexcond} and Assumption \ref{assumE} that, for any $k \geq \max\{K_0,K_1\}+1$,
\begin{equation*}
\begin{aligned}
\Delta^{k-1}
&\in \partial P_1({\bm{x}}^{k}) - \nabla P_2({\bm{x}}^{k-1}) + \nabla f({\bm{x}}^{k-1}) + \gamma_{k-1}(\nabla \phi(\bm{x}^{k})-\nabla \phi(\bm{x}^{k-1})) \\
&= \partial\iota_{\mathcal{Q}}({\bm{x}}^{k}) + \partial P_1({\bm{x}}^{k})
- \nabla P_2({\bm{x}}^{k-1}) + \nabla f({\bm{x}}^{k-1})
+ \gamma_{k-1}(\nabla \phi(\bm{x}^{k})-\nabla \phi(\bm{x}^{k-1})) \\
&\subseteq \partial\left(\iota_{\mathcal{Q}}+P_1\right)({\bm{x}}^k)
- \nabla P_2({\bm{x}}^{k-1}) + \nabla f({\bm{x}}^{k-1})
+ \gamma_{k-1}(\nabla \phi(\bm{x}^{k})-\nabla \phi(\bm{x}^{k-1})),
\end{aligned}
\end{equation*}
where the equality follows from $\bm{x}^k\in\mathrm{int}\,\mathrm{dom}\,\phi=\mathrm{int}\,\mathcal{Q}$ and hence $\partial\iota_{\mathcal{Q}}({\bm{x}}^{k})=\{0\}$ for all $k\geq0$, and the last inclusion can be easily verified by the definition of the subdifferential of a proper closed convex function (see, for example, \cite[Theorem 23.8]{r1970convex}). Using the above two relations, one can obtain that, for any $k\geq\max\{K_0,K_1\}+1$,
\begin{equation}\label{elemSubdiffF}
\begin{aligned}
&\Delta^{k-1} - (\nabla P_2({\bm{x}}^{k}) - \nabla P_2({\bm{x}}^{k-1})) - \gamma_{k-1}(\nabla \phi(\bm{x}^{k})-\nabla \phi(\bm{x}^{k-1})) + (\nabla f({\bm{x}}^{k}) - \nabla f({\bm{x}}^{k-1})) \\
&~~\in \partial\left(\iota_{\mathcal{Q}}+P_1\right)({\bm{x}}^k) - \nabla P_2({\bm{x}}^{k})
+ \nabla f({\bm{x}}^{k}) = \partial F({\bm{x}}^k).
\end{aligned}
\end{equation}
This, together with $L<\gamma_{\min}\leq\gamma_k\leq\gamma_{\max}<\infty$ for all $k\geq0$, the {global} Lipschitz continuity of $\nabla f$, $\nabla P_2$ and $\nabla \phi$ on $\mathcal{N}_{\tilde{\varepsilon}}$ implies that there exists a constant $a>0$ such that
\begin{equation}\label{subdiffbdSC1}
\mathrm{dist}(0,\,\partial F({\bm{x}}^k)) \leq a\|{\bm{x}}^{k}-{\bm{x}}^{k-1}\| + \|{\Delta^{k-1}}\|, \quad \forall\,k \geq \max\{K_0,K_1\}+1.
\end{equation}

Moreover, we see from \eqref{suffdes-C1} that, for all $k \geq \max\{K_0,K_1\}+1$,
\begin{equation*}%\label{suffdesnewSC1}
\begin{aligned}
F(\bm{x}^{k+1})
&\leq F(\bm{x}^{k})
- \big((1-\sigma)\gamma_k-L\big)\mathcal{D}_{\phi}(\bm{x}^{k+1},\,\bm{x}^{k})
- \gamma_k\mathcal{D}_{\phi}(\bm{x}^{k},\,\bm{x}^{k+1})  \\
&\leq F(\bm{x}^{k})
- \big((1-\sigma)\gamma_{\min}-L\big)\mathcal{D}_{\phi}(\bm{x}^{k+1},\,\bm{x}^{k})
- \gamma_{\min}\mathcal{D}_{\phi}(\bm{x}^{k},\,\bm{x}^{k+1}) \\
&\leq F(\bm{x}^{k})
- \frac{\big((2-\sigma)\gamma_{\min}-L\big)\mu}{2}\,\|\bm{x}^{k+1} -\bm{x}^{k}\|^2,
\end{aligned}
\end{equation*}
where the last inequality follows from the strong convexity of $\phi$ on $\mathcal{N}_{\phi}$ (by Assumption \ref{assumC}1). For notational simplicity, let $b:=\frac{\left((2-\sigma)\gamma_{\min}-L\right)\mu}{2}$. Then, we have that
\begin{equation}\label{suffides-fSC1}
F(\bm{x}^{k}) - F(\bm{x}^{k+1}) \geq b\|\bm{x}^{k+1} -\bm{x}^{k}\|^2, \quad \forall\,k \geq \max\{K_0,K_1\}+1.
\end{equation}

We are now ready to show that the sequence is convergent. Note from Proposition \ref{pro-fval}(i) that $\zeta:=\lim\limits_{k\to\infty}F({\bm{x}}^k)$ exists. In the following, we will consider two cases.

\textit{Case 1}. Suppose first that $F(\bm{x}^{k_{\zeta}})=\zeta$ for some $k_{\zeta}\geq\max\{K_0,K_1\}+1$. Since $\{F(\bm{x}^k)\}_{k=1}^{\infty}$ is non-increasing (see Proposition \ref{pro-bdSC1}(i)), we must have $F(\bm{x}^{k})\equiv \zeta$ for all $k\geq k_{\zeta}$. This together with \eqref{suffides-fSC1} implies that $\bm{x}^{k_{\zeta}+t}=\bm{x}^{k_{\zeta}}$ for all $t\geq0$, namely, $\{\bm{x}^k\}_{k=0}^{\infty}$ converges finitely.

\textit{Case 2}. We consider the case where $F(\bm{x}^k)>\zeta$ for all $k\geq\max\{K_0,K_1\}+1$. Since $F$ is a KL function and $F\equiv\zeta$ on $\Omega$ (by Proposition \ref{pro-fval}(ii)), we then have from the uniformized KL property (Proposition \ref{uniKL}) that, there exist $\varepsilon>0$, $\eta>0$, and $\varphi\in\Xi_\eta$ such that
\begin{equation*}%\label{KLSC1-p}
\varphi'\left(F(\bm{x})-\zeta\right)\mathrm{dist}\left(0,\,\partial F(\bm{x})\right)\geq1,
\end{equation*}
for all $\bm{x}$ satisfying $\mathrm{dist}(\bm{x},\,\Omega)<\varepsilon$ and $\zeta<F(\bm{x})<\zeta+\eta$. Moreover, since $\lim\limits_{k\to\infty}\mathrm{dist}(\bm{x}^k,\Omega)=0$ and $\{F(\bm{x}^k)\}_{k=1}^{\infty}$ is non-increasing and converges to $\zeta$ (by Proposition \ref{pro-bdSC1}(i)), then for such $\varepsilon$ and $\eta$, there exists $K_2\geq1$ such that $\mathrm{dist}(\bm{x}^k,\,\Omega)<\varepsilon$ and $\zeta<F(\bm{x}^k)<\zeta+\eta$ for all $k\geq K_2$. Thus, we have
\begin{equation}\label{KLSC1}
\varphi'\big(F(\bm{x}^k)-\zeta\big)\mathrm{dist}\big(0,\,\partial F(\bm{x}^k)\big)\geq1, \quad \forall\,k \geq K_2.
\end{equation}
Now, let $K_3:=\max\{K_0,K_1,K_2\}+1$. Then, it holds that, for any $k\geq K_3$,
\begin{equation*}
\begin{aligned}
&\quad \left[\varphi\big(F(\bm{x}^k)-\zeta\big)-\varphi\big(F(\bm{x}^{k+1})-\zeta\big)\right] \mathrm{dist}\big(0,\,\partial F(\bm{x}^k)\big)  \\
&\geq \varphi'\big(F(\bm{x}^k)-\zeta\big)\mathrm{dist}\big(0,\,\partial F(\bm{x}^k)\big)\cdot\big(F(\bm{x}^k)-F(\bm{x}^{k+1})\big)\\
&\geq F(\bm{x}^k)-F(\bm{x}^{k+1})\\
&\geq b\|\bm{x}^{k+1} -\bm{x}^{k}\|^2,
\end{aligned}
\end{equation*}
where the first inequality follows from the concavity of $\varphi$, the second inequality holds from \eqref{KLSC1} and the non-increasing property of $\{F(\bm{x}^k)\}_{k=1}^{\infty}$, and the last inequality holds from \eqref{suffides-fSC1}. Combining the above inequality
and \eqref{subdiffbdSC1}, we further obtain that,
\begin{equation}\label{fklSC1}
\|\bm{x}^{k+1} -\bm{x}^{k}\|^2\leq \frac{a}{b}\left[\varphi\big(F(\bm{x}^k)-\zeta\big)-\varphi\big(F(\bm{x}^{k+1})-\zeta\big)\right]
\left(\|\bm{x}^{k}-\bm{x}^{k-1}\| + \frac{1}{a}\|\Delta^{k-1}\|\right).
\end{equation}
Taking the square root of \eqref{fklSC1} and using the inequality $\sqrt{uv}\leq\frac{u+v}{2}$ for $u,v\geq0$, we see that
\begin{equation*}
\begin{aligned}
\|\bm{x}^{k+1} -\bm{x}^{k}\|&\leq \sqrt{\frac{a}{b}\left[\varphi\big(F(\bm{x}^k)-\zeta\big)-\varphi\big(F(\bm{x}^{k+1})-\zeta\big)\right]
}
\cdot
\sqrt{\|\bm{x}^{k}-\bm{x}^{k-1}\| + \frac{1}{a}\|\Delta^{k-1}\|}  \\
&\leq \frac{a}{2b}\left[\varphi\big(F(\bm{x}^k)-\zeta\big)-\varphi\big(F(\bm{x}^{k+1})-\zeta\big)\right]
+ \frac{1}{2}\|{\bm{x}}^{k}-{\bm{x}}^{k-1}\|
+ \frac{1}{2a}\|\Delta^{k-1}\|,
\end{aligned}
\end{equation*}
which yields
\begin{equation*}
\begin{aligned}
\|\bm{x}^{k+1} -\bm{x}^{k}\|
&\leq \frac{a}{b}\left[\varphi\big(F(\bm{x}^k)-\zeta\big)-\varphi\big(F(\bm{x}^{k+1})-\zeta\big)\right]
+ \|{\bm{x}}^{k}-{\bm{x}}^{k-1}\| - \|\bm{x}^{k+1}-\bm{x}^{k}\|  \\
&\qquad + \frac{1}{a}\|{\Delta^{k-1}}\|.
\end{aligned}
\end{equation*}
Summing the above relation from $k=K_3$ to $\infty$, together with $\sum_{k=0}^{\infty}\|\Delta^k\|<\infty$, we deduce that
\begin{equation*}
\sum_{k=K_3}^{\infty}\|\bm{x}^{k+1} -\bm{x}^{k}\|
\leq \frac{a}{b}\,\varphi\big(F(\bm{x}^{K_3})-\zeta\big)
+ \|\bm{x}^{K_3}-\bm{x}^{K_3-1}\|
+ \frac{1}{a}\sum_{k=K_3}^{\infty}\|\Delta^{k-1}\|
< \infty,
\end{equation*}
which implies that $\sum_{k=0}^{\infty}\|\bm{x}^{k+1}-\bm{x}^{k}\|<\infty$. Therefore, the sequence $\{\bm{x}^{k}\}$ is convergent. This completes the proof.
\end{proof}

We next show that the sequence $\{\bm{x}^k\}$ generated by the iBPDCA with (SC2) is also convergent to an $\ell$-stationary point of problem \eqref{DCpro} under additional appropriate assumptions. The analysis follows a similar line of arguments as presented in Theorem \ref{thm-wholeseqSC1} for the iBPDCA with (SC1), but will utilize the following auxiliary function with $\tau>0$:
\begin{equation*}
H_{\tau}(\bm{u},\,\bm{v}) = F(\bm{u}) + \frac{\tau}{2}\|\bm{u}-\bm{v}\|^2, \quad \forall\,\bm{u},\,\bm{v}\in\mathbb{E}.
\end{equation*}
Hence, the analysis could be more intricate.

\begin{theorem}[\textbf{Global sequential convergence of the iBPDCA with (SC2)}]\label{thm-wholeseqSC2}
Suppose that Assumptions \ref{assumA}, \ref{assumB}, \ref{assumC}, \ref{assumD}, \ref{assumE} hold, $L<\gamma_{\min}\leq\gamma_k\leq\gamma_{\max}<\infty$, $0\leq\sigma<(\gamma_{\min}-L)/\gamma_{\max}$ for all $k\geq0$, and $\sum_{k=0}^{\infty}\|\Delta^k\|<\infty$. Let $\{\bm{x}^k\}$ be the sequence generated by the iBPDCA with (SC2) in Algorithm \ref{alg-iBPDCA}. If $H_{\tau}$ is a KL function for any $\tau>0$, then the sequence $\{\bm{x}^k\}$ converges to an $\ell$-stationary point of problem \eqref{DCpro} for sufficiently small $\sigma$.
\end{theorem}
\begin{proof}
In view of Theorem \ref{thm-subseq}, we have that $\lim\limits_{k\to\infty}\mathrm{dist}({\bm{x}}^k, \,\Omega)=0$ with $\Omega\subseteq\mathcal{S}$, where $\Omega$ is the set of all cluster points of $\{\bm{x}^{k}\}$ and $\mathcal{S}$ is the set of all stationary points of problem \eqref{DCpro}. Note that, under Assumption \ref{assumD}2, a stationary point as defined in \eqref{defsta} is inherently an $\ell$-stationary point as defined in \eqref{deflimsta}. Thus, we only need to show that the sequence is convergent.

To this end, we first note from $\Omega\subseteq\mathrm{dom}\,P_1\cap\,\mathcal{Q}$, $\Omega\subseteq\mathcal{S}$, and Assumptions \ref{assumB}, \ref{assumC}, \ref{assumD} that, for any $\tilde{\varepsilon}>0$, there exists $K_1>0$ such that $\mathrm{dist}({\bm{x}}^k, \,\Omega)<\tilde{\varepsilon}$ and $\bm{x}^k\in\mathcal{N}_{\phi}\cap\mathcal{N}_f\cap\mathcal{N}_{\mathcal{S}}$ for all $k \geq K_1$. Moreover, since $\Omega$ is compact (due to the boundedness of $\{{\bm{x}}^k\}$), by shrinking $\tilde{\varepsilon}$ if necessary, we may assume without loss of generality that $\nabla f$, $\nabla P_2$ and $\nabla\phi$ are Lipschitz continuous on the bounded set $\mathcal{N}_{\tilde{\varepsilon}}:=\big\{\bm{x}\in\mathrm{dom}\,P_1\cap\mathcal{Q}\cap\mathcal{N}_{\phi}\cap\mathcal{N}_f\cap\mathcal{N}_{\mathcal{S}}:\mathrm{dist}(\bm{x}, \,\Omega)<\tilde{\varepsilon}\big\}$. Let $\ell_{\phi}$ be the Lipschitz constant of $\nabla\phi$ on $\mathcal{N}_{\tilde{\varepsilon}}$.

Next, we see from \eqref{suffdes-C2p} that, for all $k \geq K_1$,
\begin{equation*}%\label{suffdesnewSC1}
\begin{aligned}
&\quad F(\bm{x}^{k+1}) - F(\bm{x}^{k}) \\
&\leq
- \big(\gamma_k-L\big)\mathcal{D}_{\phi}(\bm{x}^{k+1},\,\bm{x}^{k})
- \gamma_k\mathcal{D}_{\phi}(\bm{x}^{k},\,\bm{x}^{k+1})
+ \sigma\gamma_k\mathcal{D}_{\phi}(\bm{x}^{k},\,\bm{x}^{k-1})  \\
&\leq
- \big(\gamma_{\min}-L\big)\mathcal{D}_{\phi}(\bm{x}^{k+1},\,\bm{x}^{k})
- \gamma_{\min}\mathcal{D}_{\phi}(\bm{x}^{k},\,\bm{x}^{k+1})
+ \sigma\gamma_{\max}\mathcal{D}_{\phi}(\bm{x}^{k},\,\bm{x}^{k-1})  \\
&\leq
- \frac{\big(2\gamma_{\min}-L\big)\mu}{2}\,\|\bm{x}^{k+1}-\bm{x}^{k}\|^2
+ \frac{\sigma\gamma_{\max}\ell_{\phi}}{2}\,\|\bm{x}^{k}-\bm{x}^{k-1}\|^2 \\
&\leq \frac{\sigma\gamma_{\max}\ell_{\phi}}{2}\,\|\bm{x}^{k}-\bm{x}^{k-1}\|^2
- \frac{\sigma\gamma_{\max}\ell_{\phi}}{2}\,\|\bm{x}^{k+1}-\bm{x}^{k}\|^2
- \frac{\big(2\gamma_{\min}-L\big)\mu-\sigma\gamma_{\max}\ell_{\phi}}{2}
\,\|\bm{x}^{k+1}-\bm{x}^{k}\|^2,
\end{aligned}
\end{equation*}
where the second last inequality follows from the strong convexity of $\phi$ on $\mathcal{N}_{\phi}$ (by Assumption \ref{assumC}1) and the Lipschitz continuity of $\nabla\phi$ on $\mathcal{N}_{\tilde{\varepsilon}}$. Since $\gamma_{\min}>L$, we must have $\big(2\gamma_{\min}-L\big)\mu-\sigma\gamma_{\max}\ell_{\phi}>0$ for sufficiently small $\sigma$. Let $\tau:=\sigma\gamma_{\max}\ell_{\phi}$ and $b:=\frac{\left(2\gamma_{\min}-L\right)\mu-\sigma\gamma_{\max}\ell_{\phi}}{2}$ for notational simplicity. Then, for sufficiently small $\sigma$, we have that
\begin{equation}\label{suffides-fSC2}
H_{\tau}(\bm{x}^{k},\bm{x}^{k-1}) - H_{\tau}(\bm{x}^{k+1},\bm{x}^k)
\geq b\|\bm{x}^{k+1} -\bm{x}^{k}\|^2, \quad \forall\,k\geq K_1.
\end{equation}
This implies that, for sufficiently small $\sigma$, the squence $\left\{H_{\tau}({\bm{x}}^k,{\bm{x}}^{k-1})\right\}$ is non-increasing after a finite number of iterations. Moreover, since $F$ is level-bounded (by Assumption \ref{assumA}4), we have that $F^* := \mathrm{\inf}\big\{F(\bm{x})\,|\,\bm{x}\in\mathcal{Q}\big\}>-\infty$ and hence $\big\{F(\bm{x}^{k})\big\}_{k=0}^{\infty}$ is bounded from below. This together with the nonnegativity of $\frac{\tau}{2}\|\bm{x}^{k}-\bm{x}^{k-1}\|^2$ implies that  $\left\{H_{\tau}({\bm{x}}^k,{\bm{x}}^{k-1})\right\}$ is also bounded from below. And it follows from the strong convexity of $\phi$ on $\mathcal{N}_{\phi}$ (by Assumption \ref{assumC}1) and Proposition \ref{pro-bdSC2}(iii) that $\bm{x}^k-\bm{x}^{k-1}\to0$. Then, combining with Proposition \ref{pro-fval}(i), we can conclude that $\lim\limits_{k\to\infty}H_{\tau}({\bm{x}}^k,{\bm{x}}^{k-})=\lim\limits_{k\to\infty}F({\bm{x}}^k)+\lim\limits_{k\to\infty}\frac{\tau}{2}\|{\bm{x}}^k-\bm{x}^{k-1}\|^2=\lim\limits_{k\to\infty}F({\bm{x}}^k)=\zeta$. In the following, we consider two cases.

\textit{Case 1}. Suppose first that $H_{\tau}(\bm{x}^{k_{\zeta}},\,\bm{x}^{k_{\zeta}-1})=\zeta$ for some $k_{\zeta}\geq K_1+1$. Since $\{H_{\tau}(\bm{x}^k,\bm{x}^{k-1})\}$ is non-increasing for all $k\geq K_1$, we must have $H_{\tau}(\bm{x}^{k},\bm{x}^{k-1}) = \zeta$ for all $k\geq k_{\zeta}$. This together with \eqref{suffides-fSC2} implies that $\bm{x}^{k_{\zeta}+t}=\bm{x}^{k_{\zeta}}$ for all $t\geq0$, namely, $\{\bm{x}^k\}_{k=0}^{\infty}$ converges finitely.

\textit{Case 2}. We consider the case where $H_{\tau}(\bm{x}^k,\bm{x}^{k-1})>\zeta$ for all $k\geq\max\{K_0,K_1\}+1$. In this case, we will divide the proof into three steps: (1) we first prove that $H_{\tau}$ is constant on the set of cluster points of the sequence $\{(\bm{x}^k,\bm{x}^{k-1})\}$ and then apply the uniformized KL property (Proposition \ref{uniKL}); (2) we bound the distance from 0 to $\partial H_{\tau}(\bm{x}^k, \bm{x}^{k-1})$; (3) we show that $\{\bm{x}^k\}$ is a Cauchy sequence and hence is convergent. The complete proof is presented as follows.

\textit{Step 1}. In view of Proposition \ref{pro-bdSC2}(iii) and the strong convexity of $\phi$ on $\mathcal{N}_{\phi}$ (by Assumption \ref{assumC}1), we have that $\|\bm{x}^{k+1}-\bm{x}^k\|\to 0$. Then, it is easy to see that the set of cluster points of $\{(\bm{x}^k,\bm{x}^{k-1})\}$ is $\Lambda:=\{(\bm{x},\bm{x}): \bm{x}\in \Omega\}$. Moreover, for any $(\widehat{\bm{x}},\widehat{\bm{x}})\in\Lambda$ (and hence $\widehat{\bm{x}}\in\Omega$), we have $H(\widehat{\bm{x}},\widehat{\bm{x}}) = F(\widehat{\bm{x}})$. Thus, we can conclude from Proposition \ref{pro-fval}(ii) that $H_{\tau}\equiv\zeta$
on $\Lambda$. This fact together with our assumption that $H_{\tau}$ is a KL function for any $\tau>0$ and the uniformized KL property (Proposition \ref{uniKL}) implies that, there exist $\varepsilon>0$, $\eta>0$, and $\varphi\in\Xi_\eta$ such that
\begin{equation*}%\label{KLSC1-p}
\varphi'\left(H_{\tau}(\bm{u},\bm{v})-\zeta\right)\mathrm{dist}\left(0,\,\partial H_{\tau}(\bm{u},\bm{v})\right)\geq1,
\end{equation*}
for all $(\bm{u},\bm{v})$ satisfying $\mathrm{dist}((\bm{u},\bm{v}),\,\Lambda)<\varepsilon$ and $\zeta<H_{\tau}(\bm{u},\bm{v})<\zeta+\eta$. On the other hand, since $\lim\limits_{k\to\infty}\mathrm{dist}((\bm{x}^k,\bm{x}^{k-1}),\Lambda)=0$ and $\left\{H_{\tau}({\bm{x}}^k,{\bm{x}}^{k-1})\right\}$ is non-increasing after $K_1$ iterations and converges to $\zeta$, then for such $\varepsilon$ and $\eta$, there exists $K_2\geq1$ such that $\mathrm{dist}((\bm{x}^k,\bm{x}^{k-1}),\,\Lambda)<\varepsilon$ and $\zeta<H_{\tau}({\bm{x}}^k,{\bm{x}}^{k-1})<\zeta+\eta$ for all $k\geq K_2$. Thus, we have
\begin{equation}\label{KLSC2}
\varphi'\big(H_{\tau}({\bm{x}}^k,{\bm{x}}^{k-1})-\zeta\big)\mathrm{dist}\big(0,\, \partial H_{\tau}({\bm{x}}^k,{\bm{x}}^{k-1})\big)\geq1, \quad \forall\,k \geq K_2.
\end{equation}

\textit{Step 2}. We consider the subdifferential of $H_{\tau}$ at $(\bm{x}^k, \bm{x}^{k-1})$ for any $k\geq \max\{K_0,K_1,K_2\}+1$, where $K_0$ is given in Assumption \ref{assumE}. Note that, for all $k \geq K_1$, ${\bm{x}}^k\in\mathcal{N}_{\mathcal{S}}$ and $P_2$ is continuously differentiable on $\mathcal{N}_{\mathcal{S}}$ (by Assumption \ref{assumD}2). Then, for all $k \geq \max\{K_0,K_1,K_2\}+1$,
\begin{equation*}
\begin{aligned}
\partial H_{\tau}(\bm{x}^k,\bm{x}^{k-1})
&=\left[\partial F(\bm{x}^k) + \tau(\bm{x}^k-\bm{x}^{k-1}),
\,-\tau(\bm{x}^k-\bm{x}^{k-1})\right] \\
&\ni\left[\Delta^{k-1} - \big(\nabla P_2({\bm{x}}^{k}) - \nabla P_2({\bm{x}}^{k-1})\big)
- \gamma_{k-1}\big(\nabla \phi(\bm{x}^{k})-\nabla \phi(\bm{x}^{k-1})\big) \right. \\
&\qquad \left. + \big(\nabla f({\bm{x}}^{k}) - \nabla f({\bm{x}}^{k-1})\big)
+ \tau(\bm{x}^k-\bm{x}^{k-1}), \,-\tau(\bm{x}^k-\bm{x}^{k-1})\right],
\end{aligned}
\end{equation*}
where the inclusion follows from \eqref{elemSubdiffF}. This, together with $L<\gamma_{\min}\leq\gamma_k\leq\gamma_{\max}<\infty$ for all $k\geq0$, the {global} Lipschitz continuity of $\nabla f$, $\nabla P_2$ and $\nabla \phi$ on $\mathcal{N}_{\tilde{\varepsilon}}$ implies that there exists a constant $a>0$ such that
\begin{equation}\label{subdiffbdSC2}
\mathrm{dist}(0,\,\partial H_{\tau}(\bm{x}^k,\bm{x}^{k-1})) \leq a\|{\bm{x}}^{k}-{\bm{x}}^{k-1}\| + \|{\Delta^{k-1}}\|, \quad \forall\,k \geq \max\{K_0,K_1,K_2\}+1.
\end{equation}

\textit{Step 3}. We now prove the convergence of the sequence. Let $K_3:=\max\{K_0,K_1,K_2\}+1$. Then, we have that, for any $k\geq K_3$,
\begin{equation*}
\begin{aligned}
&\quad \left[\varphi\big(H_{\tau}(\bm{x}^k,\bm{x}^{k-1})-\zeta\big)
-\varphi\big(H_{\tau}(\bm{x}^{k+1},\bm{x}^{k})-\zeta\big)\right] \mathrm{dist}\big(0,\,\partial H_{\tau}(\bm{x}^k,\bm{x}^{k-1})\big)  \\
&\geq \varphi'\big(H_{\tau}(\bm{x}^k,\bm{x}^{k-1})-\zeta\big)\mathrm{dist}\big(0,\,\partial H_{\tau}(\bm{x}^k,\bm{x}^{k-1})\big)
\cdot\big(H_{\tau}(\bm{x}^k,\bm{x}^{k-1})
-H_{\tau}(\bm{x}^{k+1},\bm{x}^{k})\big)  \\
&\geq H_{\tau}(\bm{x}^k,\bm{x}^{k-1})
-H_{\tau}(\bm{x}^{k+1},\bm{x}^{k})  \\
&\geq b\|\bm{x}^{k+1} -\bm{x}^{k}\|^2,
\end{aligned}
\end{equation*}
where the first inequality follows from the concavity of $\varphi$, the second inequality holds from \eqref{KLSC2} and the fact that $\left\{H_{\tau}({\bm{x}}^k,{\bm{x}}^{k-1})\right\}$ is non-increasing after $K_1$ iterations, and the last inequality holds from \eqref{suffides-fSC2}. Combining the above inequality and \eqref{subdiffbdSC2}, we further obtain that,
\begin{equation}\label{fklSC2}
\begin{aligned}
&\quad \|\bm{x}^{k+1} -\bm{x}^{k}\|^2 \\
&\leq \frac{a}{b}\left[\varphi\big(H_{\tau}(\bm{x}^k,\bm{x}^{k-1})-\zeta\big)
-\varphi\big(H_{\tau}(\bm{x}^{k+1},\bm{x}^{k})-\zeta\big)\right]
\left(\|\bm{x}^{k}-\bm{x}^{k-1}\| + \frac{1}{a}\|\Delta^{k-1}\|\right).
\end{aligned}
\end{equation}
Taking the square root of \eqref{fklSC2} and using the inequality $\sqrt{uv}\leq\frac{u+v}{2}$ for $u,v\geq0$, we see that
\begin{equation*}
\begin{aligned}
\|\bm{x}^{k+1} -\bm{x}^{k}\|&\leq \sqrt{\frac{a}{b}\left[\varphi\big(H_{\tau}(\bm{x}^k,\bm{x}^{k-1})-\zeta\big)
-\varphi\big(H_{\tau}(\bm{x}^{k+1},\bm{x}^{k})-\zeta\big)\right]
}
\cdot
\sqrt{\|\bm{x}^{k}-\bm{x}^{k-1}\| + \frac{1}{a}\|\Delta^{k-1}\|}  \\
&\leq \frac{a}{2b}\left[\varphi\big(H_{\tau}(\bm{x}^k,\bm{x}^{k-1})-\zeta\big)
-\varphi\big(H_{\tau}(\bm{x}^{k+1},\bm{x}^{k})-\zeta\big)\right]
+ \frac{1}{2}\|{\bm{x}}^{k}-{\bm{x}}^{k-1}\|
+ \frac{1}{2a}\|\Delta^{k-1}\|,
\end{aligned}
\end{equation*}
which yields
\begin{equation*}
\begin{aligned}
\|\bm{x}^{k+1} -\bm{x}^{k}\|
&\leq \frac{a}{b}\left[\varphi\big(H_{\tau}(\bm{x}^k,\bm{x}^{k-1})-\zeta\big)
-\varphi\big(H_{\tau}(\bm{x}^{k+1},\bm{x}^{k})-\zeta\big)\right]  \\
&\qquad + \|{\bm{x}}^{k}-{\bm{x}}^{k-1}\|
- \|\bm{x}^{k+1}-\bm{x}^{k}\|
+ \frac{1}{a}\|{\Delta^{k-1}}\|,
\end{aligned}
\end{equation*}
Summing the above relation from $k=K_3$ to $\infty$, together with $\sum_{k=0}^{\infty}\|\Delta^k\|<\infty$, we have that
\begin{equation*}
\sum_{k=K_3}^{\infty}\|\bm{x}^{k+1} -\bm{x}^{k}\|
\leq \frac{a}{b}\,\varphi\big(H_{\tau}(\bm{x}^{K_3}, \bm{x}^{K_3-1})-\zeta\big)
+ \|\bm{x}^{K_3}-\bm{x}^{K_3-1}\|
+ \frac{1}{a}\sum_{k=K_3}^{\infty}\|\Delta^{k-1}\|
< \infty,
\end{equation*}
which implies that $\sum_{k=0}^{\infty}\|\bm{x}^{k+1}-\bm{x}^{k}\|<\infty$. Therefore, the sequence $\{\bm{x}^{k}\}$ is convergent. This completes the proof.
\end{proof}

Before closing this section, we would like to make some remarks regarding the global sequential convergence results. First, one may have noticed that our current analysis relies on the summable error condition $\sum_{k=0}^{\infty}\|\Delta^k\|<\infty$ to establish strong convergence results in Theorems \ref{thm-wholeseqSC1} and \ref{thm-wholeseqSC2}.
%This summable error condition can be trivially satisfied if, for instance, no errors occur on the left-hand side of condition \eqref{iBPDCA-inexcond} (namely, $\Delta^k\equiv0$) when applying a certain subsolver for solving the subproblem \eqref{iBCDA-subpro}.
Fortunately, it has been shown that $\sum^{\infty}_{k=0}\mathcal{D}_{\phi}(\bm{x}^{k+1}, \,\bm{x}^{k})<\infty$ in \eqref{summaC1-p} and \eqref{summaC2-p} for (SC1) and (SC2), respectively. Thus, one could also verify that $\|\Delta^k\|\leq\sigma\gamma_k\mathcal{D}_{\phi}(\bm{x}^{k+1}, \,\bm{x}^{k})$ along with (SC1) or verify that $\|\Delta^k\|\leq\sigma\gamma_k\mathcal{D}_{\phi}(\bm{x}^{k}, \,\bm{x}^{k-1})$ along with (SC2) to ensure $\sum_{k=0}^{\infty}\|\Delta^k\|<\infty$. Interestingly, our numerical experiments show that it is usually sufficient to \textit{solely} employ our stopping criterion (SC1) or (SC2) for obtaining satisfactory empirical performance. Therefore, in our experiments, we choose not to verify the additional inequality to ensure $\sum_{k=0}^{\infty}\|\Delta^k\|<\infty$ for the sake of simplicity. Of course, eliminating the need for such a summable error condition in theory remains an interesting topic for future research. Second, a similar auxiliary function in the form of $H_{\tau}$ has also been used in, for example, \cite{wcp2018proximal,y2024proximal}, for developing the global sequential convergence under the KL property. When $F$ is semialgebraic, $H_{\tau}$ with any $\tau>0$ is also semialgebraic and hence is a KL function; see, for example,
\cite[Section 4.3]{abrs2010proximal} and
\cite[Section 2]{bdl2007the}. More discussions can be found in \cite[Remark 4.2]{wcp2018proximal}.
% Finally, it may also be possible to investigate the convergence rate of the sequence $\{\bm{x}^k\}$ under an additional condition on the KL exponent of $F$ or $H_{\tau}$ with other appropriate conditions, using similar arguments as in \cite{tft2022new,wcp2018proximal}.

%%%%%%%%%%%%%%%%%%%%%%%%%%%%%%%%%%%%%%%%%%%%%
\section{Numerical experiments}\label{sec-num}

In this section, we conduct some numerical experiments to evaluate the performance of our iBPDCA in Algorithm \ref{alg-iBPDCA} for solving the $\ell_{1-2}$ regularized least squares problem \eqref{probleml12} and the constrained $\ell_{1-2}$ sparse optimization problem \eqref{probleml12con}. All experiments are run in {\sc Matlab} R2023a on a PC with Intel processor i7-12700K@3.60GHz (with 12 cores and 20 threads) and 64GB of RAM, equipped with a Windows OS.

%%%%%%%%%%%%%%%%%%%%%%%%%%%%%%%%%%%%%%%%%%%%%%%%%%%%%%%%%%%%
\subsection{The $\ell_{1-2}$ regularized least squares problem}\label{sec-num-l12reg}

In this subsection, we consider the $\ell_{1-2}$ regularized least squares problem (see, e.g., \cite{lyhx2015computing,wcp2018proximal,ylhx2015minimization}):
\begin{equation}\label{probleml12}
\min\limits_{\bm{x}\in\mathbb{R}^n}~ F_{\text{reg}}(\bm{x}):=\frac{1}{2}\|A\bm{x}-\bm{b}\|^2 + \lambda\big( \|\bm{x}\|_1 - \|\bm{x}\| \big),
\end{equation}
where $A\in\mathbb{R}^{m\times n}$, $\bm{b}\in\mathbb{R}^m$, and $\lambda>0$ is the regularization parameter. To apply our iBPDCA in Algorithm \ref{alg-iBPDCA} for solving problem \eqref{probleml12}, we consider the following equivalent reformulation:
\begin{equation*}
\min\limits_{\bm{x}\in\mathcal{Q}:=\mathbb{R}^n}~
\underbrace{\frac{1}{2}\|A\bm{x}-\bm{b}\|^2 + \lambda\|\bm{x}\|_1}_{P_1(\bm{x})} - \underbrace{\lambda\|\bm{x}\|}_{P_2(\bm{x})} + \underbrace{0}_{f(\bm{x})},
\end{equation*}
and choose the quadratic kernel function $\phi(\bm{x}):=\frac{1}{2}\|\bm{x}\|^2$. It is easy to verify that, with such choices, $(f,\,\phi)$ is 0-smooth adaptable restricted on $\mathbb{R}^n$. Moreover, we assume that the $A$ in \eqref{probleml12} does not have zero columns so that $F_{\text{reg}}$ is level-bounded; see \cite[Example 4.1(b)]{lp2017further} and \cite[Lemma 3.1]{ylhx2015minimization}. Thus, our iBPDCA is applicable and the associated subproblem at the $k$-th iteration ($k\geq0$) takes the following form:
\begin{equation}\label{subprob_l12reg}
\min\limits_{\bm{x}\in\mathbb{R}^n}~ \lambda\|\bm{x}\|_1 - \langle\bm{\xi}^k,\,\bm{x}-\bm{x}^k\rangle + \frac{1}{2}\|A\bm{x}-\bm{b}\|^2+\frac{\gamma_k}{2}\|\bm{x}-\bm{x}^k\|^2,
\end{equation}
where $\bm{\xi}^k\in\partial P_2(\bm{x}^k)$.

%%%%%%%%%%%%%%%%%%%%%%%%%%%%%%%%%%%%%%%%%%%%%%%%%%%%%%%%%%%%%%
\subsubsection{A dual semi-smooth Newton method for solving the subproblem \eqref{subprob_l12reg}}\label{sec_l12reg_ssn}

We next discuss how to efficiently solve \eqref{subprob_l12reg} via a dual semi-smooth Newton ({\sc Ssn}) method to find a point $\bm{x}^{k+1}$ associated with an error pair $(\Delta^k,\,\delta_k)$ satisfying the relative stopping criterion (SC1) or (SC2). Specifically, we attempt to solve the following equivalent problem:
\begin{equation}\label{subprobref_l12reg}
\begin{aligned}
&\min\limits_{\bm{x}\in\mathbb{R}^n,\,\bm{y}\in\mathbb{R}^m}~\lambda\|\bm{x}\|_1
- \langle\bm{\xi}^k,\,\bm{x}\rangle+\frac{1}{2}\|\bm{y}-\bm{b}\|^2
+\frac{\gamma_k}{2}\|\bm{x}-\bm{x}^k\|^2 \\
&\quad~~\,\mathrm{s.t.}\qquad A\bm{x} = \bm{y}.
\end{aligned}
\end{equation}
By some manipulations, one can show that the dual problem of \eqref{subprobref_l12reg} can be equivalently given by (in a minimization form)
\begin{equation}\label{subprobdual_l12reg}
\hspace{-2mm}
\min\limits_{\bm{z}\in\mathbb{R}^m}
\left\{\begin{aligned}
&\Psi_k^{\text{reg}}(\bm{z}):=
\frac{1}{2}\|\bm{z}\|^2 + \langle\bm{z},\,\bm{b}\rangle
-\lambda\left\|\texttt{prox}_{\lambda\gamma_k^{-1}\|\cdot\|_1}\!
\left(\bm{v}_k(\bm{z})\right)\right\|_1 \\[3pt]
&~~-\frac{\gamma_k}{2} \left\|\texttt{prox}_{\lambda\gamma_k^{-1}\|\cdot\|_1}\!
\left(\bm{v}_k(\bm{z})\right)
- \bm{v}_k(\bm{z})\right\|^2
+ \frac{\gamma_k}{2}\left\|\bm{v}_k(\bm{z})\right\|^2
- \frac{\gamma_k}{2} \|\bm{x}^k\|^2
\end{aligned}\right\},
\end{equation}
where $\bm{z}\in\mathbb{R}^m$ is the dual variable and
$\bm{v}_k(\bm{z}) := \gamma_k^{-1}\bm{\xi}^k+\bm{x}^k
-\gamma_k^{-1}A^{\top}\bm{z}$. The detailed derivation of the dual problem \eqref{subprobdual_l12reg} is relegated to Appendix \ref{apd-subprobdual_l12reg}. From the property of the Moreau envelope of $\lambda\gamma_k^{-1}\|\cdot\|_1$ (see, e.g.,
\cite[Proposition 12.29]{bc2011convex}), we see that $\Psi_k^{\text{reg}}$ is strongly convex and continuously differentiable with the gradient
\begin{equation*}%\label{subdualgrad_l12reg}
\nabla\Psi_k^{\text{reg}}(\bm{z})
= - A\texttt{prox}_{\lambda\gamma_k^{-1}\|\cdot\|_1}
\!\left(\bm{v}_k(\bm{z})\right)
+ \bm{z} + \bm{b}.
\end{equation*}
Thus, the optimal solution of problem \eqref{subprobdual_l12reg} can be readily obtained by solving the following nonlinear equation:
\begin{equation}\label{subdualequa_l12reg}
\nabla \Psi_k^{\text{reg}}(\bm{z})=0.
\end{equation}
In view of the nice properties of $\texttt{prox}_{\lambda\gamma_k^{-1}\|\cdot\|_1}$, we then follow \cite{lst2018highly,twst2020sparse,zts2020learning} to apply a globally convergent and locally superlinearly convergent {\sc Ssn} method to solve \eqref{subdualequa_l12reg}. To this end, we define a multifunction $\widehat{\partial}^2 \Psi_k^{\text{reg}}:\mathbb{R}^m \rightrightarrows \mathbb{R}^{m \times m}$ as follows:
\begin{equation*}
\widehat{\partial}^2\Psi_k^{\text{reg}}(\bm{z}) := I + \gamma_k^{-1}A\partial\texttt{prox}_{\lambda\gamma_k^{-1}\|\cdot\|_1}\!
\left(\bm{v}_k(\bm{z})\right)A^{\top},
\end{equation*}
where $\partial\texttt{prox}_{\lambda\gamma_k^{-1}\|\cdot\|_1}\!
\left(\bm{v}_k(\bm{z})\right)$ is the Clarke subdifferential of the Lipschitz continuous mapping $\texttt{prox}_{\lambda\gamma_k^{-1}\|\cdot\|_1}(\cdot)$ at $\bm{v}_k(\bm{z})$, defined as follows
\begin{equation*}
\partial\texttt{prox}_{\alpha^{-1}\|\cdot\|_1}(\bm{u})
:= \left\{\mathrm{Diag}(\bm{d}) \,:\, \bm{d}\in\mathbb{R}^n, ~d_i\in
\left\{\begin{aligned}
&\{1\}, && \mathrm{if}~~|u_i|>\alpha^{-1}, \\
&[0,\,1], && \mathrm{if}~~|u_i|=\alpha^{-1}, \\
&\{0\}, && \mathrm{if}~~|u_i|<\alpha^{-1},
\end{aligned}\right.~~\right\}.
\end{equation*}
It is clear that all elements in $\widehat{\partial}^2\Psi_k^{\text{reg}}(\bm{z})$ are positive definite. We are now ready to present the {\sc Ssn} method for solving equation \eqref{subdualequa_l12reg} in Algorithm \ref{algo:SSN} and refer readers to \cite[Theorem 3.6]{lst2018highly} for its convergence results.

\begin{algorithm}[htb!]
\caption{A semi-smooth Newton ({\sc Ssn}) method for solving equation \eqref{subdualequa_l12reg}}\label{algo:SSN}
 	
\textbf{Initialization:} Choose $\bar{\eta}\in(0,1)$, $\gamma\in(0,1]$, $\mu\in(0,1/2)$, $\delta\in(0,1)$, and an initial point $\bm{z}^{k,0}\in\mathbb{R}^m$. Set $t=0$. Repeat until a termination criterion is met. \vspace{-1mm}
\begin{itemize}[leftmargin=1.6cm]
\item[\textbf{Step 1.}] Compute $\nabla\Psi_k^{\text{reg}}(\bm{z}^{k,t})$ and select an element $H^{k,t}\in\widehat{\partial}^2\Psi_k^{\text{reg}}(\bm{z}^{k,t})$. Solve the linear system $H^{k,t}\bm{d} = -\nabla\Psi_k^{\text{reg}}(\bm{z}^{k,t})$ nearly exactly by the (sparse) Cholesky factorization with forward and backward substitutions, \textit{or} approximately by the preconditioned conjugate gradient method to find $\bm{d}^{k,t}$ such that $\big\|H^{k,t}\bm{d}^{k,t} + \nabla\Psi_k^{\text{reg}}(\bm{z}^{k,t})\big\|
    \leq \min\big(\bar{\eta}, \,\|\nabla\Psi_k^{\text{reg}}(\bm{z}^{k,t})\|^{1+\gamma}\big)$.
	
\item[\textbf{Step 2.}] (\textbf{Inexact line search}) Find a step size $\alpha_t:=\delta^{i_t}$, where $i_t$ is the smallest nonnegative integer $i$ for which $\Psi_k^{\text{reg}}(\bm{z}^{k,t} + \delta^i\bm{d}^{k,t})
    \leq \Psi_k^{\text{reg}}(\bm{z}^{k,t}) + \mu \delta^{i}\langle\nabla\Psi_k^{\text{reg}}(\bm{z}^{k,t}), \,\bm{d}^{k,t}\rangle$.

\item[\textbf{Step 3.}] Set $\bm{z}^{k,t+1} = \bm{z}^{k,t} + \alpha_t\bm{d}^{k,t}$, $t=t+1$, and go to \textbf{Step 1}.
\end{itemize}
\end{algorithm}

We next show that our inexact stopping criteria (SC1) and (SC2) can be achieved through some appropriate manipulations based on the dual sequence generated by the {\sc Ssn} method. Specifically, at the $k$-th iteration, we apply the {\sc Ssn} method for solving equation \eqref{subdualequa_l12reg}, which generates a dual sequence $\{\bm{z}^{k,t}\}$. Let
\begin{equation*}
\bm{w}^{k,t}:=\texttt{prox}_{\lambda\gamma_k^{-1}\|\cdot\|_1}
\big(\gamma_k^{-1}\bm{\xi}^k+\bm{x}^k-\gamma_k^{-1}A^{\top}\bm{z}^{k,t}\big)
\quad \mbox{and} \quad
\bm{e}^{k,t} := \nabla\Psi_k^{\text{reg}}\big(\bm{z}^{k,t}\big).
\end{equation*}
We then have the following proposition, whose proof is relegated to Appendix \ref{apd-pro-l12reg}.

\begin{proposition}\label{pro-scnew-l12reg}
If $(\bm{w}^{k,t},\bm{e}^{k,t})$ satisfies
\begin{equation}\label{SC1new_l12reg}
\|A^{\top}\bm{e}^{k,t}\|^2 +
|\langle A^{\top} \bm{e}^{k,t}, \,\bm{w}^{k,t}-\bm{x}^{k}\rangle|
\leq \frac{\sigma\gamma_k}{2}\|\bm{w}^{k,t}-\bm{x}^{k}\|^2,
\end{equation}
then the inexact stopping criterion (SC1) holds for $\bm{x}^{k+1}:=\bm{w}^{k,t}$, $\Delta^k:=-A^{\top}\bm{e}^{k,t}$ and $\delta_k:=0$. Similarly, If $(\bm{w}^{k,t},\bm{e}^{k,t})$ satisfies
\begin{equation}\label{SC2new_l12reg}
\|A^{\top}\bm{e}^{k,t}\|^2 +
|\langle A^{\top} \bm{e}^{k,t}, \,\bm{w}^{k,t}-\bm{x}^{k}\rangle|
\leq \frac{\sigma\gamma_k}{2}\|\bm{x}^{k}-\bm{x}^{k-1}\|^2,
\end{equation}
then the inexact stopping criterion (SC2) holds for $\bm{x}^{k+1}:=\bm{w}^{k,t}$, $\Delta^k:=-A^{\top}\bm{e}^{k,t}$ and $\delta_k:=0$.
\end{proposition}

From Proposition \ref{pro-scnew-l12reg}, we see that the inexact stopping criterion (SC1) or (SC2) is indeed verifiable and can be satisfied as long as inequality \eqref{SC1new_l12reg} or \eqref{SC2new_l12reg} holds. Moreover, when the dual sequence $\{\bm{z}^{k,t}\}$ generated by the {\sc Ssn} method is convergent, we have that $\{\bm{w}^{k,t}\}$ is bounded and $\bm{e}^{k,t}\to0$, and hence $\|A^{\top}\bm{e}^{k,t}\|^2 +
|\langle A^{\top} \bm{e}^{k,t}, \,\bm{w}^{k,t}-\bm{x}^{k}\rangle|\to0$. On the other hand, when $\bm{x}^{k}$ is not the optimal solution of the $k$-th subproblem \eqref{subprob_l12reg} or $\bm{x}^{k}\neq\bm{x}^{k-1}$, the right-hand-side term in \eqref{SC1new_l12reg} or \eqref{SC2new_l12reg} cannot approach zero. Therefore, inequality \eqref{SC1new_l12reg} or \eqref{SC2new_l12reg} must hold after finitely many iterations.

%%%%%%%%%%%%%%%%%%%%%%%%%%%%%%%%%%%%%%%%%%
\subsubsection{Comparison results}\label{sec-comp-l12reg}

We will evaluate the performance of the iBPDCA with (SC1) (denoted by iBPDCA-SC1) and the iBPDCA with (SC2) (denoted by iBPDCA-SC2). For both iBPDCA-SC1 and iBPDCA-SC2, we set $\gamma_k=\max\left\{\frac{1}{\sqrt{k+1}}, \,10^{-1}\right\}$. With such choices of $\{\gamma_k\}$, we then follow Theorem \ref{thm-subseq} to set $\sigma=0.9$ for iBPDCA-SC1 and set $\sigma=0.09$ for iBPDCA-SC2 to guarantee the convergence.
%Here, we would like to point out that the theoretical choices of the tolerance parameter $\sigma$ could be conservative, and a larger value of $\sigma$ is possible and may result in better numerical performances. In this paper, we will forgo the investigation of tuning such a parameter for the sake of simplicity.
In addition, for the {\sc Ssn} in Algorithm \ref{algo:SSN}, we set $\mu=10^{-4}$, $\delta=0.5$, $\bar{\eta}=10^{-3}$ and $\gamma=0.2$. Moreover, we will initialize {\sc Ssn} with the origin at the first outer iteration and then employ a \textit{warm-start} strategy thereafter. Specifically, at each outer iteration, we initialize {\sc Ssn} with the approximate solution obtained by the {\sc Ssn} method in the previous outer iteration.

We also include the general iterative shrinkage and thresholding (GIST) algorithm \cite{gzlhy2013general}, the non-monotone accelerated proximal gradient method with line search (denoted by nmAPG)\footnote{The implementations of nmAPG in our experiments are based on the original {\sc Matlab} codes, which are
available at \url{https://zhouchenlin.github.io/.}} \cite{ll2015accelerated}, and the proximal difference-of-convex algorithm with extrapolation (pDCAe)\footnote{The {\sc Matlab} codes of the pDCAe for solving \eqref{probleml12} are available at \url{https://www.polyu.edu.hk/ama/profile/pong/pDCAe_final_codes/}.} \cite{wcp2018proximal} in our comparisons. Note that pDCAe can be viewed as a special case of Bregman proximal DC algorithm with extrapolation (BPDCAe) in \cite{tft2022new}, where the corresponding kernel function is chosen as $\phi(\bm{x}) = \frac{1}{2}\|\bm{x}\|^2$. For each of these algorithms, we follow the parameter settings recommended in their respective references.

%The pDCAe for solving \eqref{probleml12} is given as follows: choose proper extrapolation parameters $\{\beta_k\}$, let $\bm{x}^{-1} = \bm{x}^0$, and then at the $k$-th iteration,
%\begin{equation*}
%\left\{\begin{aligned}
%&\mathrm{Take}~\mathrm{any}~\bm{\xi}^k
%%\in \lambda\partial \|\bm{x}^k\|
%\in  \partial P_2 (\bm{x}^k)~\mathrm{and}~\mathrm{compute}  \\
%&\bm{y}^k = \bm{x}^k + \beta_k (\bm{x}^k - \bm{x}^{k-1}),  \\
%&\bm{x}^{k+1} = \mathop{\mathrm{argmin}}\limits_{\bm{x}} \left\{\langle A^{\top}(A\bm{y}^k-\bm{b}) - \bm{\xi}^k, \,\bm{x}\rangle + \frac{L_A}{2}\|\bm{x} - \bm{y}^k\|^2 + \lambda\|\bm{x}\|_1\right\},
%\end{aligned}\right.
%\end{equation*}
%where $L_A:=\lambda_{\max}(A^{\top}A)$ is the largest eigenvalue of $A^{\top}A$. We follow \cite[Section 5]{wcp2018proximal} to choose the extrapolation parameters $\{\beta_k\}$ and perform both the fixed restart and the adaptive restart strategies.

We initialize all methods with a point $\bm{x}^0$ obtained by applying the fast iterative shrinkage-thresholding algorithm (FISTA) with backtracking \cite{bt2009a} for solving the classical $\ell_1$ regularized least squares problem:$\min\big\{\lambda\|\bm{x}\|_1+\frac{1}{2}\|A\bm{x}-\bm{b}\|^2\big\}$, using 200 iterations. Moreover, for iBPDCA-SC2, in order to prevent an improper choice of $\bm{x}^{-1}$, we will use (SC1) as a warm start at the first outer iteration and then transition to (SC2) thereafter. Finally, we terminate all methods when the number of iterations reaches 30000 or the following holds for 3 consecutive iterations:
\begin{equation*}
\max\left\{\frac{\|\bm{x}^k-\bm{x}^{k-1}\|}{1+\|\bm{x}^k\|},\, \frac{|F_{\text{reg}}(\bm{x}^k)-F_{\text{reg}}(\bm{x}^{k-1})|}{1+|F_{\text{reg}}(\bm{x}^k)|}\right\} < 10^{-7}
\quad \mbox{or} \quad
\frac{|F_{\text{reg}}(\bm{x}^k)-F_{\text{reg}}(\bm{x}^{k-1})|}{1+|F_{\text{reg}}(\bm{x}^k)|} < 10^{-10}.
\end{equation*}

In the following experiments, we choose $\lambda \in \{0.01, 0.1, 1, 10\}$ and consider $(m,n,s)=(100i,1000i,20i)$ for $i = 2,5,10,15,20$. For each triple $(m, n, s)$, we follow \cite[Section 5]{wcp2018proximal} to randomly generate a trial as follows. First, we generate a matrix $A \in \mathbb{R}^{m\times n}$ with i.i.d. standard Gaussian entries. We then choose a subset $\mathcal{S}\subset\{1, \cdots, n\}$ of size $s$ uniformly at random and generate an $s$-sparse vector $\bm{x}_{\text{orig}}\in\mathbb{R}^{n}$, which has i.i.d. standard Gaussian entries on $\mathcal{S}$ and zeros on the complement set $\mathcal{S}^c$. Finally, we generate the vector $\bm{b} \in \mathbb{R}^{m}$ by setting $\bm{b}=A\bm{x}_{\text{orig}} + 0.01\cdot\widehat{\bm{n}}$, where $\widehat{\bm{n}}\in\mathbb{R}^m$ is a random vector with i.i.d. standard Gaussian entries.

%In order to evaluate the performance of each method, we record the objective function value (denoted by ``\texttt{obj}"), the number of iterations (denoted by ``\texttt{iter}", where the total number of the {\sc Ssn} iterations in iBPDCA is also given in the bracket), the computational time (denoted by ``\texttt{time}"), and the computational time used to obtain an initial point by FISTA (denoted by ``\texttt{t0}").
The average computational results for each triple $(m,n,s)$ from 20 instances are presented in Tables \ref{Table1-l12reg} and \ref{Table2-l12reg}. From the results, one can observe that both iBPDCA-SC1 and iBPDCA-SC2 exhibit superior numerical performance compared to GIST, nmAPG, and pDCAe, especially when the regularization parameter $\lambda$ is small or the problem size is large. For example, for the cases where $\lambda\leq1$ and $n\geq10000$, both iBPDCA-SC1 and iBPDCA-SC2 are consistently about 8 times faster than GIST, 5 times faster than pDCAe, and 2 times faster than nmAPG, while achieving comparable or even better objective function values. Moreover, one can also see that iBPDCA-SC1 and iBPDCA-SC2 exhibit the similar performance. This is indeed expected since they essentially employ the same framework, but differ in their stopping criteria for solving the subproblems. With the kernel function chosen as $\phi(\bm{x})=\frac{1}{2}\|\bm{x}\|^2$ in this part of experiments, the cost of computing the associated Bregman distance becomes negligible, rendering the effort required to verify inequalities \eqref{SC1new_l12reg} and \eqref{SC2new_l12reg} almost identical. Therefore, the CPU time of iBPDCA-SC1 and iBPDCA-SC2 can be similar when they undergo a similar number of total {\sc Ssn} iterations.

We also test all the algorithms with two instances $(A,\bm{b})$ obtained from the data sets \texttt{mpg7} and \texttt{gisette} in the UCI data repository \cite{UCIdata}. Here, the \texttt{mpg7} data set is an expanded version of its original version, with the last digit signifying that an order 7 polynomial is used to generate the basis functions; see \cite[Section 4.1]{lst2018highly} for more details. For \texttt{mpg7}, the matrix $A$ is of size $392\times3432$ with $\lambda_{\max}(A^{\top}A)\approx1.28\times10^4$. For \texttt{gisette}, the matrix $A$ is of size $1000\times4971$ with $\lambda_{\max}(A^{\top}A)\approx3.36\times10^6$. For each data set, we choose $\lambda=\lambda_c\|A^{\top}\bm{b}\|_{\infty}$ with $\lambda_c\in\{10^{-3}, 10^{-4},10^{-5}\}$. Figure \ref{Figl1l2regreal} shows the numerical results of GIST, nmAPG, pDCAe, iBPDCA-SC1 and iBPDCA-SC2, where we plot the objective value $F_{\text{reg}}(\bm{x}^k)$ against the computational time.
%The normalized objective value is computed by $\frac{F_{\text{reg}}(\bm{x}^k)-F_{\text{reg}}^{\min}}{F_{\text{reg}}(\bm{x}^0)-F_{\text{reg}}^{\min}}$, where $F_{\text{reg}}^{\min}$ denotes the minimum of the terminating objective values obtained among all methods.
From the results, we see that GIST, nmAPG, and pDCAe generally exhibit a slower rate of reduction in the objective value, potentially attributed to the large Lipschitz constant associated with the real data. In contrast, our iBPDCA-SC1 and iBPDCA-SC2 continue to demonstrate superior performance on these two real data sets.

\begin{remark}[\textbf{Comments on difference between pDCAe and iBPDCA}]\label{rek-diff}
As discussed in Section \ref{sec-iBPDCA}, by allowing the inexact minimization of the subproblem, our iBPDCA has more flexibility in practical implementations when it is applied to a specific problem. Indeed, for problem \eqref{probleml12}, the pDCAe developed in \cite{wcp2018proximal} is applied to the DC decomposition scheme shown on the left below so that the subproblem admits an easy-to-compute solution, while our iBPDCA can be applied to the DC decomposition scheme shown on the right below, as implemented in our experiments.
\begin{equation*}%\label{different-schemes}
\begin{aligned}
\quad
&\hspace{2.8cm}{\text{pDCAe}} \\
&\hspace{3.2cm}\Downarrow \\
&\min\limits_{\bm{x} \in \mathbb{R}^n}~
{\overbrace{{\lambda\|\bm{x}\|_1}}^{P_1(\bm{x})}
\,-{\overbrace{{\lambda\|\bm{x}\|}}^{P_2(\bm{x})}}+\, {\overbrace{{\frac{1}{2}\|A\bm{x}-\bm{b}\|_2^2}}^{f(\bm{x})}}}
\end{aligned}
\quad~~ vs ~~\quad
\begin{aligned}
&\hspace{2.8cm}{\text{iBPDCA}} \\
&\hspace{3.2cm}\Downarrow \\
&\min\limits_{\bm{x} \in \mathbb{R}^n}~
{\overbrace{{\lambda\|\bm{x}\|_1
+ \frac{1}{2}\|A\bm{x}-\bm{b}\|_2^2}}^{P_1(\bm{x})}
\,-\overbrace{{\lambda\|\bm{x}\|}}^{P_2(\bm{x})}\,
+\,\overbrace{{0}}^{f(\bm{x})}}.
\end{aligned}	
\end{equation*}
These two decomposition schemes result in different algorithms with different numerical performances. Since pDCAe also approximates the smooth convex part $\frac{1}{2}\|A\bm{x}-\bm{b}\|^2$ by a quadratic majorant, its performance would be significantly influenced by the Lipschitz constant $L_A:=\lambda_{\max}(A^{\top}A)$. This sensitivity to the Lipschitz constant is indeed a well-recognized limitation of first-order methods and may lead to the inferior performance of GIST, nmAPG and pDCAe for solving large-scale problems, as observed in Tables \ref{Table1-l12reg}$\&$\ref{Table2-l12reg} and Figure \ref{Figl1l2regreal}. This is because, for a randomly generated matrix $A$, the largest eigenvalue of $A^{\top}A$ tends to increase with the size, and for the data sets \texttt{mpg7} and \texttt{gisette}, the largest eigenvalue of $A^{\top}A$ is around $1.28\times10^4$ and $3.36\times10^6$, respectively. In contrast, our iBPDCA can sidestep the majorization of the least squares term and choose instead to solve the subproblem via a highly efficient second-order {\sc Ssn} method. The results in Tables \ref{Table1-l12reg}$\&$\ref{Table2-l12reg} and Figure \ref{Figl1l2regreal} illustrate the promising performance of this algorithmic framework. But we should emphasize that these achievements are made possible by allowing the inexact minimization of the subproblem. This also highlights the motivation for developing an efficient inexact algorithmic framework in this work.
%We would also like to mention that similar inexact DCAs have been developed by Tang et al. \cite{twst2020sparse} and Zhang et al. \cite{zts2020learning} for solving the nonconvex square-root-loss regression problem and the MCP penalized graphical problem, respectively, with encouraging numerical performances. In comparison, our work addresses a more general problem \eqref{DCpro} and introduces a more comprehensive inexact algorithmic framework. This broadens the potential applications of our approach.
%It opens up new possibilities for addressing a variety of complex problems across different domains, showcasing the versatility and innovative nature of our inexact algorithmic framework.
\end{remark}

\begin{table}[ht]
\caption{The average computational results for each triple $(m,n,s)$ from 20 instances on the $\ell_{1-2}$ regularized least squares problem with $\lambda\in\{0.01,0.1\}$, where ``\texttt{obj}" denotes the objective function value, ``\texttt{iter}" denotes the number of iterations (the total number of the {\sc Ssn} iterations in iBPDCA is also given in the bracket), ``\texttt{time}" denotes the computational time, and ``\texttt{t0}" denotes the computational time used to obtain an initial point by FISTA using 200 iterations.}\label{Table1-l12reg}
\centering \tabcolsep 3.6pt
\scalebox{0.9}{
\renewcommand\arraystretch{1}
\begin{tabular}{|l|l|lllc|lllc|}
\hline
$(m,n,s)$ & \texttt{method} & \texttt{obj} & \texttt{iter} & \texttt{time} & \texttt{t0} & \texttt{obj} & \texttt{iter} & \texttt{time} & \texttt{t0}  \\
\hline
&&\multicolumn{4}{c|}{$\lambda=0.01$} & \multicolumn{4}{c|}{$\lambda=0.1$}  \\
\hline
\multirow{5}{*}{(200,2000,40)}
&GIST       & 3.22e-1 & 30000 & 2.74 & 0.03 & 2.54e+0 & 16498 & 1.45 & 0.03 \\
&nmAPG      & 2.54e-1 & 4507 & 0.40 & 0.03 & 2.54e+0 & 1535 & 0.14 & 0.03 \\
&pDCAe      & 2.73e-1 & 30000 & 2.06 & 0.03 & 2.54e+0 & 8353 & 0.56 & 0.03 \\
&iBPDCA-SC1 & 2.54e-1 & 120 (405) & 0.64 & 0.03 & 2.54e+0 & 30 (170) & 0.29 & 0.03 \\
&iBPDCA-SC2 & 2.54e-1 & 120 (414) & 0.64 & 0.03 & 2.54e+0 & 30 (173) & 0.30 & 0.03 \\
\hline
\multirow{5}{*}{(500,5000,100)}
&GIST       & 9.92e-1 & 30000 & 14.47 & 0.16 & 7.02e+0 & 29927 & 14.28 & 0.16 \\
&nmAPG      & 6.92e-1 & 7634 & 3.49 & 0.16 & 6.92e+0 & 2660 & 1.29 & 0.16 \\
&pDCAe      & 8.12e-1 & 30000 & 10.60 & 0.16 & 6.92e+0 & 24732 & 8.62 & 0.16 \\
&iBPDCA-SC1 & 6.92e-1 & 161 (599) & 5.10 & 0.16 & 6.92e+0 & 35 (259) & 2.66 & 0.16 \\
&iBPDCA-SC2 & 6.92e-1 & 161 (612) & 5.17 & 0.16 & 6.92e+0 & 35 (263) & 2.68 & 0.16 \\
\hline
\multirow{5}{*}{(1000,10000,200)}
&GIST       & 2.39e+0 & 30000 & 212.54 & 1.92 & 1.59e+1 & 30000 & 211.05 & 1.91 \\
&nmAPG      & 1.48e+0 & 11233 & 69.61 & 1.92 & 1.48e+1 & 4356 & 34.83 & 1.91 \\
&pDCAe      & 1.88e+0 & 30000 & 137.02 & 1.92 & 1.50e+1 & 30000 & 137.30 & 1.91 \\
&iBPDCA-SC1 & 1.48e+0 & 242 (898) & 31.63 & 1.92 & 1.48e+1 & 50 (379) & 16.71 & 1.91 \\
&iBPDCA-SC2 & 1.48e+0 & 242 (914) & 31.86 & 1.92 & 1.48e+1 & 50 (382) & 16.52 & 1.91 \\
\hline
\multirow{5}{*}{(1500,15000,300)}
&GIST       & 3.86e+0 & 30000 & 552.27 & 5.09 & 2.53e+1 & 30000 & 547.84 & 5.09 \\
&nmAPG      & 2.24e+0 & 13908 & 217.94 & 5.09 & 2.24e+1 & 5412 & 114.91 & 5.09 \\
&pDCAe      & 2.99e+0 & 30000 & 362.98 & 5.09 & 2.31e+1 & 30000 & 363.23 & 5.09 \\
&iBPDCA-SC1 & 2.24e+0 & 299 (1072) & 86.72 & 5.09 & 2.24e+1 & 59 (439) & 44.33 & 5.09 \\
&iBPDCA-SC2 & 2.24e+0 & 299 (1088) & 87.13 & 5.09 & 2.24e+1 & 59 (444) & 44.19 & 5.09 \\
\hline
\multirow{5}{*}{(2000,20000,400)}
&GIST       & 5.31e+0 & 30000 & 1018.43 & 9.35 & 3.49e+1 & 30000 & 1005.39 & 9.34 \\
&nmAPG      & 3.00e+0 & 16571 & 459.36 & 9.35 & 3.00e+1 & 6394 & 242.11 & 9.34 \\
&pDCAe      & 4.12e+0 & 30000 & 653.66 & 9.35 & 3.13e+1 & 30000 & 651.64 & 9.34 \\
&iBPDCA-SC1 & 3.00e+0 & 367 (1295) & 201.98 & 9.35 & 3.00e+1 & 73 (518) & 98.71 & 9.34 \\
&iBPDCA-SC2 & 3.00e+0 & 367 (1312) & 202.73 & 9.35 & 3.00e+1 & 73 (524) & 98.39 & 9.34 \\
\hline
\end{tabular}
}
\end{table}

\begin{table}[ht]
\caption{
Same as Table 1 but for the $\ell_{1-2}$ regularized least squares problem with $\lambda\in\{1,10\}$.
%The average computational results for each triple $(m,n,s)$ from 20 instances on the $\ell_{1-2}$ regularized least squares problem with $\lambda\in\{1,10\}$, where ``\texttt{obj}" denotes the objective function value, ``\texttt{iter}" denotes the number of iterations (the total number of the {\sc Ssn} iterations in iBPDCA is also given in the bracket), ``\texttt{time}" denotes the computational time, and ``\texttt{t0}" denotes the computational time used to obtain an initial point by FISTA.
}\label{Table2-l12reg}
\centering \tabcolsep 4.5pt
\scalebox{0.9}{
\renewcommand\arraystretch{1}
\begin{tabular}{|l|l|lllc|lllc|}
\hline
$(m,n,s)$ & \texttt{method} & \texttt{obj} & \texttt{iter} & \texttt{time} & \texttt{t0} & \texttt{obj} & \texttt{iter} & \texttt{time} & \texttt{t0}  \\
\hline
&&\multicolumn{4}{c|}{$\lambda=1$} & \multicolumn{4}{c|}{$\lambda=10$}  \\
\hline
\multirow{5}{*}{(200,2000,40)}
&GIST       & 2.53e+1 & 704 & 0.07 & 0.03    & 2.44e+2 & 94 & 0.01 & 0.03 \\
&nmAPG      & 2.53e+1 & 301 & 0.03 & 0.03    & 2.44e+2 & 161 & 0.02 & 0.03 \\
&pDCAe      & 2.53e+1 & 515 & 0.04 & 0.03    & 2.44e+2 & 153 & 0.01 & 0.03 \\
&iBPDCA-SC1 & 2.53e+1 & 9 (50) & 0.10 & 0.03 & 2.44e+2 & 8 (32) & 0.08 & 0.03 \\
&iBPDCA-SC2 & 2.53e+1 & 9 (50) & 0.10 & 0.03 & 2.44e+2 & 8 (32) & 0.08 & 0.03 \\
\hline
\multirow{5}{*}{(500,5000,100)}
&GIST       & 6.90e+1 & 4434 & 2.22 & 0.16 & 6.80e+2 & 240 & 0.15 & 0.16 \\
&nmAPG      & 6.90e+1 & 683 & 0.41 & 0.16 & 6.80e+2 & 198 & 0.12 & 0.16 \\
&pDCAe      & 6.90e+1 & 2633 & 0.94 & 0.16 & 6.80e+2 & 237 & 0.10 & 0.16 \\
&iBPDCA-SC1 & 6.90e+1 & 12 (116) & 1.29 & 0.16 & 6.80e+2 & 7 (45) & 0.57 & 0.16 \\
&iBPDCA-SC2 & 6.90e+1 & 12 (117) & 1.30 & 0.16 & 6.80e+2 & 7 (45) & 0.58 & 0.16 \\
\hline
\multirow{5}{*}{(1000,10000,200)}
&GIST       & 1.48e+2 & 14834 & 105.07 & 1.91 & 1.47e+3 & 585 & 4.33 & 1.90 \\
&nmAPG      & 1.48e+2 & 1421 & 14.14 & 1.91 & 1.47e+3 & 268 & 1.86 & 1.90 \\
&pDCAe      & 1.48e+2 & 8547 & 39.06 & 1.91 & 1.47e+3 & 401 & 1.87 & 1.90 \\
&iBPDCA-SC1 & 1.48e+2 & 15 (185) & 9.74 & 1.91 & 1.47e+3 & 7 (69) & 4.11 & 1.90 \\
&iBPDCA-SC2 & 1.48e+2 & 15 (184) & 9.64 & 1.91 & 1.47e+3 & 7 (68) & 4.10 & 1.90 \\
\hline
\multirow{5}{*}{(1500,15000,300)}
&GIST       & 2.24e+2 & 23697 & 434.23 & 5.09 & 2.23e+3 & 1213 & 22.69 & 5.09 \\
&nmAPG      & 2.24e+2 & 2014 & 56.34 & 5.09 & 2.23e+3 & 409 & 9.10 & 5.09 \\
&pDCAe      & 2.24e+2 & 13670 & 165.42 & 5.09 & 2.23e+3 & 777 & 9.52 & 5.09 \\
&iBPDCA-SC1 & 2.24e+2 & 17 (225) & 27.07 & 5.09 & 2.23e+3 & 8 (93) & 13.07 & 5.09 \\
&iBPDCA-SC2 & 2.24e+2 & 17 (226) & 27.00 & 5.09 & 2.23e+3 & 8 (93) & 12.98 & 5.09 \\
\hline
\multirow{5}{*}{(2000,20000,400)}
&GIST       & 2.99e+2 & 28289 & 952.63 & 9.32 & 2.98e+3 & 2661 & 91.11 & 9.29 \\
&nmAPG      & 2.99e+2 & 2506 & 128.11 & 9.32 & 2.98e+3 & 623 & 29.32 & 9.29 \\
&pDCAe      & 2.99e+2 & 20631 & 448.91 & 9.32 & 2.98e+3 & 1685 & 37.01 & 9.29 \\
&iBPDCA-SC1 & 2.99e+2 & 20 (264) & 59.98 & 9.32 & 2.98e+3 & 9 (129) & 33.45 & 9.29 \\
&iBPDCA-SC2 & 2.99e+2 & 20 (265) & 59.49 & 9.32 & 2.98e+3 & 9 (129) & 33.23 & 9.29 \\
\hline
\end{tabular}
}
\end{table}

\begin{figure}[ht]
\centering
\subfigure[\texttt{mpg7}, where $A$ is of size $392\times3432$ with $\lambda_{\max}(A^{\top}A)\approx1.28\times10^4$]{
\includegraphics[width=5cm]{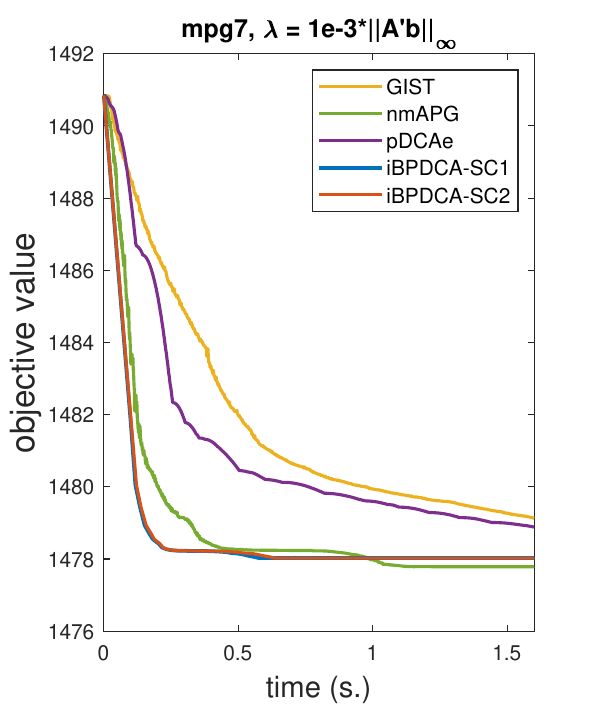}
\includegraphics[width=5cm]{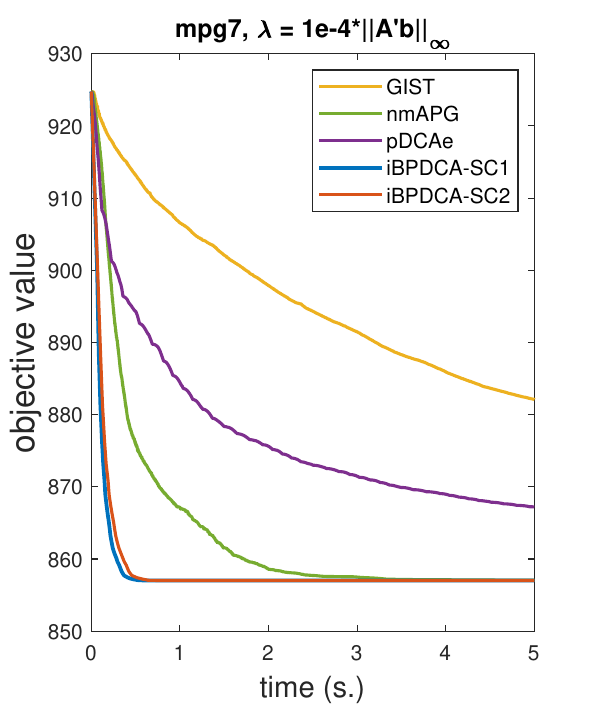}
\includegraphics[width=5cm]{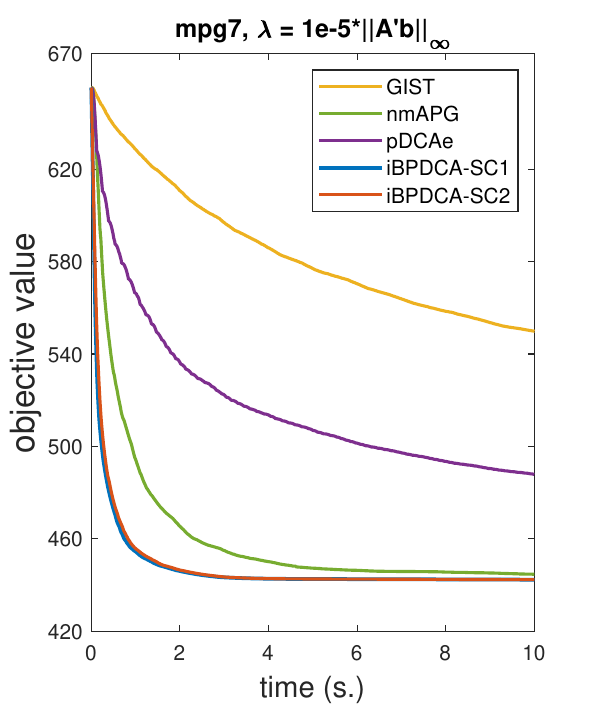}
}
\subfigure[\texttt{gisette}, where $A$ is of size $1000\times4971$ with $\lambda_{\max}(A^{\top}A)\approx3.36\times10^6$]{
\includegraphics[width=5cm]{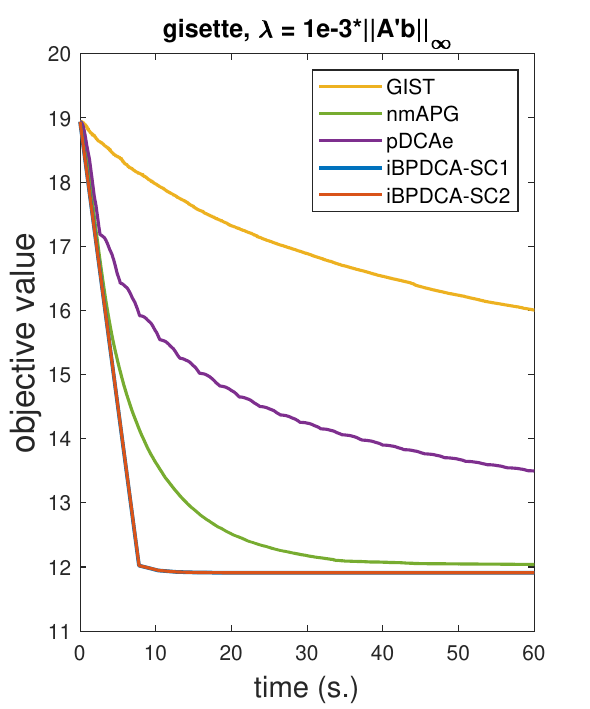}
\includegraphics[width=5cm]{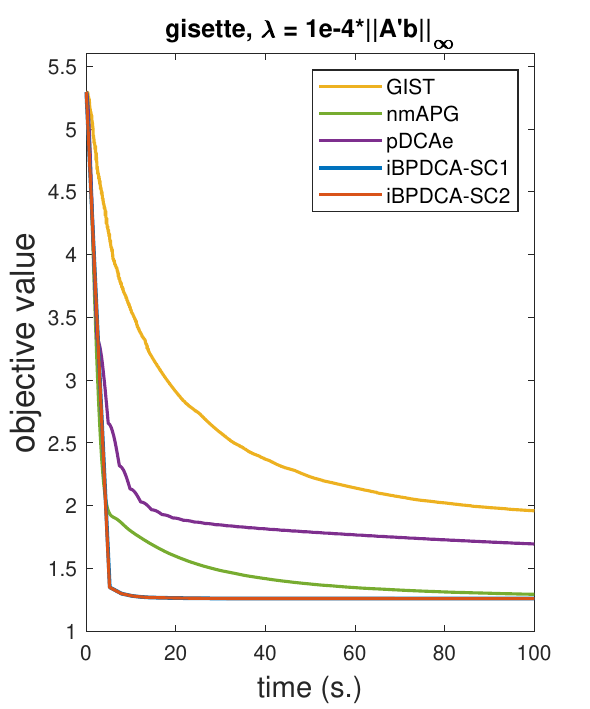}
\includegraphics[width=5cm]{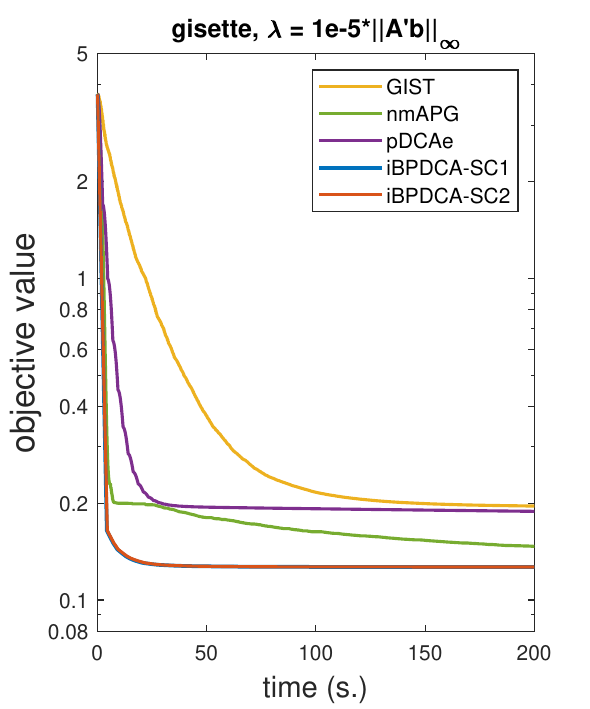}
}
\caption{Numerical results of GIST, nmAPG, pDCAe, iBPDCA-SC1 and iBPDCA-SC2 for solving the $\ell_{1-2}$ regularized least squares problem on \texttt{mpg7} and \texttt{gisette} from the UCI data repository.}\label{Figl1l2regreal}
\end{figure}

\subsection{The constrained $\ell_{1-2}$ sparse optimization problem}\label{sec-num-l12con}

In this section, we consider the constrained $\ell_{1-2}$ sparse optimization problem (see, e.g., \cite{lyhx2015computing,ylhx2015minimization,ypl2021convergence,zlpx2023retraction,zpx2023extended}):
\begin{equation}\label{probleml12con}
\begin{aligned}
\min _{\bm{x} \in \mathbb{R}^n} & \quad \|\bm{x}\|_1 - \mu\|\bm{x}\|\\
\text { s.t. } &\quad \|A \bm{x}-\bm{b}\| \leq \kappa,\\
 &\quad \|\bm{x}\|_{\infty} \leq M,
\end{aligned}
\end{equation}
where $\mu\in [0,1)$, $A\in\mathbb{R}^{m \times n}$ has full row rank, and $\bm{b} \in \mathbb{R}^m$. Moreover, we choose $\kappa\in(0,\,\|\bm{b}\|)$ so that the origin is not a feasible point and follow \cite[Section 6]{zpx2023extended} to choose $M:=(1-\mu)^{-1}\left(\|A^{\dag}\bm{b}\|_1-\mu\|A^\dag\bm{b}\|\right)$ so that the feasible region is bounded, where $A^{\dag}$ denotes the pseudo-inverse of $A$. To apply our iBPDCA in Algorithm \ref{alg-iBPDCA} for solving problem \eqref{probleml12con}, we consider the following equivalent reformulation:
\begin{equation*}
\min\limits_{\bm{x}\in\mathcal{Q}:=\mathbb{R}^n}~
F_{\text{con}}(\bm{x}):=
\underbrace{\iota_{\kappa}(A\bm{x}-\bm{b})
+ \iota_{M}(\bm{x})+\|\bm{x}\|_1}_{P_1(\bm{x})}
\,-\, \underbrace{\mu\|\bm{x}\|}_{P_2(\bm{x})}
\,+\, \underbrace{0}_{f(\bm{x})},
\end{equation*}
where $\iota_{\kappa}$ and $\iota_{M}$ denote respectively the indicator functions on the sets $\left\{\bm{x} \in \mathbb{R}^n:\|\bm{x}\|\leq \kappa\right\}$ and $\left\{\bm{x} \in \mathbb{R}^n:\|\bm{x}\|_{\infty} \leq M\right\}$, and choose the kernel function $\phi(\bm{x}):=\frac{1}{2}\|\bm{x}\|^2+\frac{1}{2}\|A\bm{x}\|^2$. It is easy to verify that $(f,\,\phi)$ is 0-smooth adaptable restricted on $\mathbb{R}^n$ and that $F_{\text{con}}$ is level-bounded. Thus, our iBPDCA is applicable and the associated subproblem at the $k$-th iteration ($k\geq0$) takes the following form:
\begin{equation}\label{subprobleml12con}
\min _{\bm{x} \in \mathbb{R}^n}
~~\iota_{\kappa}(A\bm{x}-\bm{b})
+ \iota_{M}(\bm{x})+\|\bm{x}\|_1
- \langle\bm{\xi}^k, \,\bm{x}-\bm{x}^k\rangle
+ \frac{\gamma_k}{2}\big\|\bm{x}-\bm{x}^k\big\|^2
+ \frac{\gamma_k}{2}\big\|A\bm{x}-A\bm{x}^k\big\|^2,
\end{equation}
where $\bm{\xi}^k \in \partial P_2(\bm{x}^k)$.
%$\bm{\xi}^k\in\mu\partial\|\bm{x}^k\|$.

%%%%%%%%%%%%%%%%%%%%%%%%%%%%%%%%%%%%%%%%%%%%%%%%%%%%%%%%%%%%%
\subsubsection{A dual semismooth Newton method for the subproblem \eqref{subprobleml12con}}

Similar to the implementations as described in subsection \ref{sec_l12reg_ssn}, we discuss how to efficiently solve \eqref{subprobleml12con} via a dual {\sc Ssn} method to find a point $\bm{x}^{k+1}$ associated with an error pair $(\Delta^k,\,\delta_k)$ satisfying the relative stopping criterion (SC1) or (SC2). Specifically, we are dedicated to solving the following equivalent reformulation:
\begin{equation}\label{subprobref_l12con}
\begin{aligned}
&\min\limits_{\bm{x}\in\mathbb{R}^n,\,\bm{y}\in\mathbb{R}^m}~
\|\bm{x}\|_1 + \iota_{M}(\bm{x})
+ \frac{\gamma_k}{2}\left\|\bm{x}-\big(\bm{x}^k+\gamma_k^{-1}\bm{\xi}^k\big)\right\|^2
+ \iota_{\kappa}(\bm{y})
+ \frac{\gamma_k}{2}\left\|\bm{y} - \big(A\bm{x}^k-\bm{b}\big)\right\|^2 \\
&\quad~~\,\mathrm{s.t.}\qquad A\bm{x} - \bm{y} = \bm{b}.
\end{aligned}
\end{equation}
For notational simplicity, let $g(\cdot):=\|\cdot\|_1+\iota_{M}(\cdot)$, $\bm{s}^k:=\bm{x}^k+\gamma_k^{-1}\bm{\xi}^k$ and $\bm{b}^k:=A\bm{x}^k-\bm{b}$. By some manipulations, it can be shown that the dual problem of \eqref{subprobref_l12con} admits the following equivalent minimization form:
\begin{equation}\label{dualprobleml12con}
\min _{\bm{z} \in \mathbb{R}^m}\left\{\begin{aligned}
\Psi_k^{\text{con}}(\bm{z})
&:= \langle\bm{z}, \bm{b}\rangle
+\frac{\gamma_k}{2}\left\|\bm{s}^k-\gamma_k^{-1} A^{\top} \bm{z}\right\|^2
-\left\|\texttt{prox}_{\gamma_k^{-1}g}\left(\bm{s}^k-\gamma_k^{-1} A^{\top} \bm{z}\right)\right\|_1\\[3pt]
&\qquad -\frac{\gamma_k}{2}\left\|\texttt{prox}_{\gamma_k^{-1}g}\left(\bm{s}^k-\gamma_k^{-1} A^{\top} \bm{z}\right)
-\left(\bm{s}^k-\gamma_k^{-1}A^{\top}\bm{z}\right)\right\|^2\\[3pt]
&\qquad +\frac{\gamma_k}{2}\left\|\bm{b}^k+\gamma_k^{-1}\bm{z}\right\|^2
-\frac{\gamma_k}{2}\left\|\Pi_{\kappa}\left(\bm{b}^k+\gamma_k^{-1}\bm{z}\right)
-\left(\bm{b}^k+\gamma_k^{-1}\bm{z}\right)\right\|^2 \\[3pt]
&\qquad -\frac{\gamma_k}{2}\|\bm{s}^k\|^2
-\frac{\gamma_k}{2}\|\bm{b}^k\|^2
\end{aligned}\right\},
\end{equation}
where $\bm{z}\in\mathbb{R}^m$ is the dual variable and $\Pi_{\kappa}$ is the projection operator over $\left\{\bm{x} \in \mathbb{R}^n:\|\bm{x}\|\leq \kappa\right\}$. The detailed derivation of the dual problem \eqref{dualprobleml12con} is relegated to Appendix \ref{apd-dualprobleml12con}. From the property of the Moreau envelope of $\gamma_k^{-1}g$ and $\gamma_k^{-1}\iota_{\kappa}$ (see, e.g.,
\cite[Proposition 12.29]{bc2011convex}), we see that $\Psi_k^{\text{con}}$ is convex and continuously differentiable with the gradient:
\begin{equation*}
\nabla \Psi_k^{\text{con}}(\bm{z})=-A \texttt{prox}_{\gamma_k^{-1}g}\left(\bm{s}^k-\gamma_k^{-1} A^{\top} \bm{z}\right)+\Pi_{\kappa}\left(\bm{b}^k+\gamma_k^{-1} \bm{z}\right)+\bm{b}.
\end{equation*}
Thus, an optimal solution of problem \eqref{dualprobleml12con} can be obtained by solving the following nonlinear equation:
\begin{equation}\label{subdualequa_l12con}
\nabla \Psi_k^{\text{con}}(\bm{z})=0.
\end{equation}
Similar to the implementations in subsection \ref{sec_l12reg_ssn}, to apply the SSN method, we define a multifunction $\widehat{\partial}^2\Psi_k^{\text{con}}:\mathbb{R}^m \rightrightarrows \mathbb{R}^{m \times m}$ as follows:
\begin{equation*}
\widehat{\partial}^2\Psi_k^{\text{con}}(\bm{z}) := \gamma_k^{-1}A\partial\texttt{prox}_{\gamma_k^{-1}g}(\bm{s}^k-\gamma_k^{-1}A^{\top}\bm{z})A^{\top}
+ \gamma_k^{-1}\partial\Pi_{\kappa}\left(\bm{b}^k+\gamma_k^{-1} \bm{z}\right),
\end{equation*}
where $\partial\texttt{prox}_{\gamma_k^{-1}g}\!
\left(\bm{s}^k-\gamma_k^{-1}A^{\top}\bm{z}\right)$ is the Clarke subdifferential of the Lipschitz continuous mapping $\texttt{prox}_{\gamma_k^{-1}g}(\cdot)$ at $\bm{s}^k-\gamma_k^{-1}A^{\top}\bm{z}$, defined as follows:
\begin{equation*}
\partial\texttt{prox}_{\alpha^{-1}g}(\bm{u})
:= \left\{\mathrm{Diag}(\bm{d}) \,:\, \bm{d}\in\mathbb{R}^n, ~d_i\in
\left\{\begin{aligned}
&\{1\}, && \mathrm{if}~~|u_i|>\alpha^{-1}~\text{and}~\big|\texttt{prox}_{\alpha^{-1}|\cdot|}(u_i)\big|<M, \\[3pt]
&[0,\,1], && \mathrm{if}~~|u_i|=\alpha^{-1}~\text{or}~\big|\texttt{prox}_{\alpha^{-1}|\cdot|}(u_i)\big|=M, \\[3pt]
&\{0\}, && \mathrm{if}~~|u_i|<\alpha^{-1}~\text{or}~\big|\texttt{prox}_{\alpha^{-1}|\cdot|}(u_i)\big|>M,
\end{aligned}\right.~~\right\},
\end{equation*}
and $\partial\Pi_{\kappa}\left(\bm{b}^k+\gamma_k^{-1}\bm{z}\right)$ is the Clarke subdifferential of the Lipschitz continuous mapping $\Pi_{\kappa}(\cdot)$ at $\bm{b}^k+\gamma_k^{-1}\bm{z}$, defined as follows (see, e.g.,
\cite[Remark 3.1]{zzst2020efficient}):
\begin{equation*}
\partial\Pi_{\kappa}(\bm{u})
=\left\{\begin{aligned}
&\{I\}, &&~~\mathrm{if}~~\|\bm{u}\| < \kappa, \\
&\Big\{I - \frac{t}{\kappa^2}\bm{u}\bm{u}^{\top} \mid t\in[0,1]\Big\}, &&~~ \mathrm{if}~~\|\bm{u}\| = \kappa, \\
&\Big\{ \frac{\kappa}{\|\bm{u}\|}\big(I - \frac{1}{\|\bm{u}\|^2}\bm{u}\bm{u}^{\top}\big) \Big\}, &&~~ \mathrm{if}~~\|\bm{u}\| > \kappa.
\end{aligned}\right.
\end{equation*}
Note that the elements in $\widehat{\partial}^2\Psi_k^{\text{con}}(\bm{z})$ may only be positive semidefinite. We then need to incorporate a small adaptive regularization term when solving the linear system in the {\sc Ssn} method. We present the whole iterative framework in Algorithm \ref{algo:SSNcon} and refer readers to \cite[Theorems 3.4 and 3.5]{zst2010newton} for its convergence results.

\begin{algorithm}[htb!]
\caption{A semi-smooth Newton ({\sc Ssn}) method for solving equation \eqref{subdualequa_l12con}}\label{algo:SSNcon}
 	
\textbf{Initialization:} Choose $\bar{\eta}\in(0,1)$, $\gamma\in(0,1]$, $\mu\in(0,1/2)$, $\delta\in(0,1)$, $\tau_1,\tau_2\in(0,1)$, and an initial point $\bm{z}^{k,0}\in\mathbb{R}^m$. Set $t=0$. Repeat until a termination criterion is met. \vspace{-1mm}
\begin{itemize}[leftmargin=1.6cm]
\item[\textbf{Step 1.}] Compute $\nabla\Psi_k^{\text{con}}(\bm{z}^{k,t})$, select an element $H^{k,t}\in\widehat{\partial}^2\Psi_k^{\text{con}}(\bm{z}^{k,t})$, and let $\varepsilon_t:=\tau_1\min\big\{\tau_2,\,\|\nabla\Psi_k^{\text{con}}(\bm{z}^{k,t})\|\big\}$. Solve the linear system $\left(H^{k,t}+\varepsilon_t I_m\right)\bm{d} = -\nabla\Psi_k^{\text{con}}(\bm{z}^{k,t})$ nearly exactly by the (sparse) Cholesky factorization with forward and backward substitutions, \textit{or} approximately by the preconditioned conjugate gradient method to find $\bm{d}^{k,t}$ such that $\big\|\left(H^{k,t}+\varepsilon_t I_m\right)\bm{d}^{k,t} + \nabla\Psi_k^{\text{con}}(\bm{z}^{k,t})\big\|
    \leq \min\big(\bar{\eta}, \,\|\nabla\Psi_k^{\text{con}}(\bm{z}^{k,t})\|^{1+\gamma}\big)$.
	
\item[\textbf{Step 2.}] (\textbf{Inexact line search}) Find a step size $\alpha_t:=\delta^{i_t}$, where $i_t$ is the smallest nonnegative integer $i$ for which $\Psi_k^{\text{con}}(\bm{z}^{k,t} + \delta^i\bm{d}^{k,t})
    \leq \Psi_k^{\text{con}}(\bm{z}^{k,t}) + \mu \delta^{i}\langle\nabla\Psi_k^{\text{con}}(\bm{z}^{k,t}), \,\bm{d}^{k,t}\rangle$.

\item[\textbf{Step 3.}] Set $\bm{z}^{k,t+1} = \bm{z}^{k,t} + \alpha_t\bm{d}^{k,t}$, $t=t+1$, and go to \textbf{Step 1}.
\end{itemize}
\end{algorithm}

We next show that our inexact stopping criteria (SC1) and (SC2) can also be achieved through the appropriate manipulations based on the dual sequence generated by the {\sc Ssn} method in Algorithm \ref{algo:SSNcon}. To this end, we first assume that a partially strictly feasible point $\bm{x}^{\text{feas}}$ satisfying $\|A\bm{x}^{\text{feas}}-\bm{b}\|<\kappa$ and $\|\bm{x}^{\text{feas}}\|_{\infty} \leq M$ is available on hand. Indeed, from \cite[Section 6.1]{zlpx2023retraction}, with the choice of $M=(1-\mu)^{-1}\left(\|A^{\dag}\bm{b}\|_1-\mu\|A^\dag\bm{b}\|\right)$, such a point can be simply obtained by setting $\bm{x}^{\text{feas}}:=A^\dag\bm{b}$. Then, at the $k$-th iteration, we apply the {\sc Ssn} method for solving equation \eqref{subdualequa_l12con}, which generates a dual sequence $\{\bm{z}^{k,t}\}$. Let
\begin{equation*}
\bm{w}^{k,t}:=\texttt{prox}_{\gamma_k^{-1}g}
\big(\bm{x}^k+\gamma_k^{-1}\bm{\xi}^k-\gamma_k^{-1}A^{\top}\bm{z}^{k,t}\big)
\quad \mbox{and} \quad
\bm{e}^{k,t} := \nabla\Psi_k^{\text{con}}\big(\bm{z}^{k,t}\big).
\end{equation*}
Note that the terminating approximate primal solution $\bm{w}^{k,t}$ may not be exactly feasible to the subproblem \eqref{subprobleml12con} and thus the inexact stopping criteria (SC1) and (SC2) cannot be verified at $\bm{w}^{k,t}$ in this scenario. Therefore, we further adapt a retraction strategy with the aid of the partially strictly feasible point $\bm{x}^{\text{feas}}$. Specifically, we define
\begin{equation*}
\widetilde{\bm{w}}^{k,t}:=\rho_{k,t}\bm{w}^{k,t} + (1-\rho_{k,t})\bm{x}^{\text{feas}}
\end{equation*}
with
\begin{equation*}%\label{retrapara}
\rho_{k,t} := \left\{
\begin{aligned}
&1, &&\|A\bm{w}^{k,t}-\bm{b}\|\leq\kappa,\\
&\frac{\kappa-\|A\bm{x}^{\text{feas}}-\bm{b}\|}{\|A\bm{w}^{k,t}-\bm{b}\|-\|A\bm{x}^{\text{feas}}-\bm{b}\|}, &&\|A\bm{w}^{k,t}-\bm{b}\|>\kappa.\\
\end{aligned}\right.
\end{equation*}
Then, one can easily verify that $\widetilde{\bm{w}}^{k,t}$ is a feasible point of the subproblem \eqref{subprobleml12con}, namely, it satisfies that $\|A\widetilde{\bm{w}}^{k,t}-\bm{b}\|\leq\kappa$ and $\|\widetilde{\bm{w}}^{k,t}\|_{\infty} \leq M$. Thus, $\widetilde{\bm{w}}^{k,t}$ is a feasible candidate that can be used in the verifications of stopping criteria (SC1) and (SC2). Indeed, we have the following proposition, whose proof is relegated to Appendix \ref{apd-pro-l12con}.

\begin{proposition}\label{pro-scnew-l12con}
Let
\begin{equation*}
\begin{aligned}
\Delta^{k,t}
&:=-\gamma_kA^{\top}\bm{e}^{k,t} + \gamma_k(\widetilde{\bm{w}}^{k,t}-\bm{w}^{k,t})
+ \gamma_k A^{\top}A(\widetilde{\bm{w}}^{k,t}-\bm{w}^{k,t}), \\
\bm{d}_1^{k,t}
&:=\gamma_k\left[\big({\bm{x}}^{k}+\gamma_k^{-1} \bm{\xi}^k-\gamma_k^{-1} A^{\top} {\bm{z}}^{k,t}\big)-\bm{w}^{k,t}\right], \\
\delta^1_{k,t}
&:= g(\widetilde{\bm{w}}^{k,t}) - g(\bm{w}^{k,t})
- \langle{\bm{d}}_1^{k,t},\,\widetilde{\bm{w}}^{k,t}-\bm{w}^{k,t}\rangle, \\
\bm{d}_2^{k,t}
&:=\gamma_k\left[\big(A{\bm{x}}^k-\bm{b}+\gamma_k^{-1}\bm{z}^{k,t}\big)
- \Pi_{\kappa}\big(A\bm{x}^k-\bm{b}+\gamma_k^{-1}\bm{z}^{k,t}\big)\right], \\
\delta_{k,t}^2
&:=\langle \bm{e}^{k,t} - A(\widetilde{\bm{w}}^{k,t}-\bm{w}^{k,t}), \,\bm{d}_2^{k,t}\rangle.
\end{aligned}
\end{equation*}
If $(\widetilde{\bm{w}}^{k,t},\Delta^{k,t},\delta^1_{k,t},\delta^2_{k,t})$ satisfies the termination criterion
\begin{equation}\label{SC1new_l12con}
\|\Delta^{k,t}\|^2
+ |\langle\Delta^{k,t}, \,\widetilde{\bm{w}}^{k,t}-\bm{x}^{k}\rangle|
+ \delta^1_{k,t} + \delta^2_{k,t}
\leq \frac{\sigma\gamma_k}{2}\left(\big\|\widetilde{\bm{w}}^{k,t}-\bm{x}^k\big\|^2
+ \big\|A\widetilde{\bm{w}}^{k,t}-A\bm{x}^k\big\|^2\right),
\end{equation}
then the stopping criterion (SC1) holds for $\bm{x}^{k+1}:=\widetilde{\bm{w}}^{k,t}$, $\Delta^k:=\Delta^{k,t}$, and $\delta_k:=\delta^1_{k,t} + \delta^2_{k,t}$.
Similarly, if $(\widetilde{\bm{w}}^{k,t},\Delta^{k,t},\delta^1_{k,t},\delta^2_{k,t})$ satisfies the termination criterion
\begin{equation} \label{SC2new_l12con}
\|\Delta^{k,t}\|^2
+ |\langle\Delta^{k,t}, \,\widetilde{\bm{w}}^{k,t}-\bm{x}^{k}\rangle|
+ \delta^1_{k,t} + \delta^2_{k,t}
\leq \frac{\sigma\gamma_k}{2}\left(\big\|\bm{x}^k-\bm{x}^{k-1}\big\|^2
+ \big\|A\bm{x}^k-A\bm{x}^{k-1}\big\|^2\right),
\end{equation}
then the stopping criterion (SC2) holds for $\bm{x}^{k+1}:=\widetilde{\bm{w}}^{k,t}$, $\Delta^k :=\Delta^{k,t}$, and $\delta_k:=\delta^1_{k,t} + \delta^2_{k,t}$.
\end{proposition}

Similar to discussions at the end of subsection \ref{sec_l12reg_ssn}, one can also see from Proposition \ref{pro-scnew-l12con} that our inexact stopping criterion (SC1) or (SC2) holds as long as the checkable inequality \eqref{SC1new_l12con} or \eqref{SC2new_l12con} holds. Moreover, when the dual sequence $\{\bm{z}^{k,t}\}$ generated by the {\sc Ssn} method is convergent under proper conditions (see \cite[Theorems 3.4 and 3.5]{zst2010newton}), we have that $\{\bm{w}^{k,t}\}$ is bounded, $\bm{e}^{k,t}\to0$ and $\widetilde{\bm{w}}^{k,t}-\bm{w}^{k,t}\to0$. Thus, we see that $\|\Delta^{k,t}\|^2
+ |\langle\Delta^{k,t}, \,\widetilde{\bm{w}}^{k,t}-\bm{x}^{k}\rangle|
+ \delta^1_{k,t} + \delta^2_{k,t}\to0$. On the other hand, when $\bm{x}^k$ is not the optimal solution of the $k$-th subproblem \eqref{subprobleml12con} or $\bm{x}^{k}\neq\bm{x}^{k-1}$, the right-hand-side term in \eqref{SC1new_l12con} or \eqref{SC2new_l12con} cannot approach zero. Therefore, the inequality \eqref{SC1new_l12con} or \eqref{SC2new_l12con} must hold after finitely many iterations.

\begin{remark}[\textbf{Comments on practical verifications of \eqref{SC1new_l12con} and \eqref{SC2new_l12con}}]\label{rek-verif}
One may have noticed from the above that, we choose the kernel function as $\phi(\bm{x}) = \frac{1}{2}\|\bm{x}\|^2 + \frac{1}{2}\|A\bm{x}\|^2$ so that an efficient dual {\sc Ssn} method can be readily applied for solving the subproblem.
However, compared with the verification of inequality \eqref{SC1new_l12reg} or \eqref{SC2new_l12reg}, the verification of inequality \eqref{SC1new_l12con} or \eqref{SC2new_l12con} is more involved and demands more computational resources. Specifically, it necessitates the retraction of $\bm{w}^{k,t}$ to obtain $\widetilde{\bm{w}}^{k,t}$ and the computation of a more expensive Bregman distance on the right-hand-side of inequity \eqref{SC1new_l12con} or \eqref{SC2new_l12con}; the cost of the latter becomes substantial as the problem size increases. In this case, the stopping criterion (SC2), corresponding to inequality \eqref{SC2new_l12con}, may offer more benefits in practical implementations. First, at the $k$-th outer iteration, (SC1) should compute the Bregman distance between $\widetilde{\bm{w}}^{k,t}$ and $\bm{x}^k$ to calculate the tolerance $\epsilon_{k,t}:=\frac{\sigma\gamma_k}{2}\left(\big\|\widetilde{\bm{w}}^{k,t}-\bm{x}^k\big\|^2
+ \big\|A\widetilde{\bm{w}}^{k,t}-A\bm{x}^k\big\|^2\right)$ on the right-hand-side of inequality \eqref{SC1new_l12con} at every inner iteration. In contrast, (SC2) merely requires to compute the Bregman distance between $\bm{x}^k$ and $\bm{x}^{k-1}$ once at the beginning of the inner loop to calculate the tolerance $\epsilon_k:=\frac{\sigma\gamma_k}{2}\big(\big\|\bm{x}^k-\bm{x}^{k-1}\big\|^2
+ \big\|A\bm{x}^k-A\bm{x}^{k-1}\big\|^2\big)$ on the right-hand-side of inequality \eqref{SC2new_l12con}. Second, once the tolerance $\epsilon_k$ is calculated, we can further employ an economical way to check inequality \eqref{SC2new_l12con} (and hence (SC2)) during the inner iterations. Specifically, we first compute $\|\nabla\Psi_k^{\text{con}}\big(\bm{z}^{k,t}\big)\|$ which is very cheap, and only start to retract $\bm{w}^{k,t}$ and check inequality \eqref{SC2new_l12con} when $\|\nabla\Psi_k^{\text{con}}\big(\bm{z}^{k,t}\big)\|\leq\epsilon_k$. This strategy would help us to postpone the explicit construction of $\widetilde{\bm{w}}^{k,t}$ and the calculation of error quantities on the left-hand-side of \eqref{SC2new_l12con} as long as possible to save cost, while still enforcing inequality \eqref{SC2new_l12con} (and hence (SC2)) to guarantee the convergence. In contrast, this economical checking strategy might not be available for the verification of inequality \eqref{SC1new_l12reg} (and hence (SC1)) since calculating the tolerance $\epsilon_{k,t}$ necessitates the explicit construction of $\widetilde{\bm{w}}^{k,t}$ at every inner iteration. In view of the above, the stopping criterion (SC2) can be a more advantageous option especially when solving large-sacle problems, as evidenced in Tables \ref{Table1-l12con} and \ref{Table2-l12con}.
\end{remark}

%%%%%%%%%%%%%%%%%%%%%%%%%%%%%%%%%%%%%%%%%%
\subsubsection{Comparison results}

We will evaluate the performances of the iBPDCA with (SC1) (denoted by iBPDCA-SC1) and
the iBPDCA with (SC2) (denoted by iBPDCA-SC2), using the same parameter settings as described in the first paragraph of subsection \ref{sec-comp-l12reg} with additional parameters $\tau_1$, $\tau_2$ in Algorithm \ref{algo:SSNcon} being set to $\tau_1=0.99$ and $\tau_2=10^{-6}$. We will also compare our methods with a recently developed extended sequential quadratic method with extrapolation (ESQMe)\footnote{The {\sc Matlab} codes of the ESQMe for solving \eqref{probleml12con} are available at \url{https://www.polyu.edu.hk/ama/profile/pong/ESQMe_codes/}.} \cite{zpx2023extended}. The ESQMe for solving problem \eqref{probleml12con} is roughly given as follows: let $c(\bm{x}):=\frac{1}{2}\big(\|A\bm{x}-\bm{b}\|^2-\kappa^2\big)$, $L_A:=\lambda_{\max}(A^{\top}A)$, $\bm{x}^{-1} = \bm{x}^0$, choose proper extrapolation parameters $\{\beta_k\}$, and then at the $k$-th iteration,
\begin{equation*}
\left\{\begin{aligned}
&\textbf{Step 1}. ~~\mathrm{Take}~\mathrm{any}~\bm{\xi}^k \in
\partial P_2(\bm{x}^k)~\mathrm{and}~\mathrm{compute}~\bm{y}^k = \bm{x}^k + \beta_k (\bm{x}^k - \bm{x}^{k-1}),  \\
&\textbf{Step 2}. ~~(\bm{x}^{k+1},\,s_{k+1}) = \mathop{\mathrm{argmin}}\limits_{(\bm{x},\,s)\in\mathbb{R}^{n+1}}
~\|\bm{x}\|_1 - \langle\bm{\xi}^k, \,\bm{x}\rangle + \theta_k s
+ \frac{\theta_kL_A}{2}\|\bm{x}-\bm{y}^k\|^2 \\
&\hspace{4.8cm} \mathrm{s.t.} \hspace{0.8cm} c(\bm{y}^k) + \langle\nabla c(\bm{y}^k), \,\bm{x}-\bm{y}^k\rangle \leq s, ~\|\bm{x}\|_{\infty}\leq M, ~s\geq0, \\
&\textbf{Step 3}. ~~\mathrm{update}~\theta_k~\mbox{under a proper criterion.}
\end{aligned}\right.
\end{equation*}
We follow
\cite[Section 6.1]{zpx2023extended} to choose the extrapolation parameters $\{\beta_k\}$ and other involved parameters.

We initialize all methods with a point $\bm{x}^0$ generated as follows: We first apply the popular solver SPGL1 (version 2.1) \cite{vf2009probing} for solving the classical constrained $\ell_1$ sparse optimization problem: $\min\big\{\|\bm{x}\|_1:\|A \bm{x}-\bm{b}\|\leq\kappa\big\}$, using 200 iterations to obtain an approximate solution $\bm{x}_{\text{spgl1}}$. Since $\bm{x}_{\text{spgl1}}$ may violate the constraint slightly, we then adapt the same retraction strategy as described before Proposition \ref{pro-scnew-l12con} to retract $\bm{x}_{\text{spgl1}}$ and then set the resulting point as $\bm{x}^0$. Moreover, for iBPDCA-SC2, in order to prevent an improper choice of $\bm{x}^{-1}$, we will use (SC1) as a warm start at the first iteration and then transition to (SC2) thereafter. Finally, we terminate all methods when the number of iterations reaches 20000 or the following holds for 3 consecutive iterations:
\begin{equation*}
\max\left\{\frac{\|\bm{x}^k-\bm{x}^{k-1}\|}{1+\|\bm{x}^k\|},\, \frac{|F_{\text{con}}(\bm{x}^k)-F_{\text{con}}(\bm{x}^{k-1})|}{1+|F_{\text{con}}(\bm{x}^k)|}\right\} < 10^{-7}
\quad \mbox{or} \quad
\frac{|F_{\text{con}}(\bm{x}^k)-F_{\text{con}}(\bm{x}^{k-1})|}{1+|F_{\text{con}}(\bm{x}^k)|} < 10^{-10}.
\end{equation*}

In the following experiments, we let $\mu=0.95$ in \eqref{probleml12con} and $(m,n,s)=(100i,1000i,20i)$ for $i=5,10,15,20,25,30$. For each triple $(m, n, s)$, a random instance is generated using the same way as in subsection \ref{sec-comp-l12reg}. Moreover, we choose $\kappa=\texttt{nf}\cdot\|0.01\cdot\widehat{\bm{n}}\|$ with $\texttt{nf}\in\{1.1,\,2\}$ (the first choice is also used in \cite[Section 6.1]{zpx2023extended}), and present the average computational results for each triple $(m, n, s)$ from 20 instances in Tables \ref{Table1-l12con} and \ref{Table2-l12con}, respectively.
%In the tables, we record the objective function value (denoted by ``\texttt{obj}"), the violation of the constraint $\|A\bm{x}^k-\bm{b}\|-\kappa$ (denoted by ``\texttt{feas}"), the recovery error $\frac{\|\bm{x}^k-\bm{x}_{\text{orig}}\|}{1+\|\bm{x}_{\text{orig}}\|}$ (denoted by ``\texttt{rec}"), the number of iterations (denoted by ``\texttt{iter}", where the total number of the {\sc Ssn} iterations in iBPDCA is also given in the bracket), the computational time (denoted by ``\texttt{time}"), and the computational time used to obtain an initial point by SPGL1 (denoted by ``\texttt{t0}").

From Tables \ref{Table1-l12con} and \ref{Table2-l12con}, one can see that our iBPDCA-SC1 and iBPDCA-SC2 significantly outperform ESQMe in both recovery error and computational efficiency. Moreover, similar to the results on the $\ell_{1-2}$ regularized least squares problem in Section \ref{sec-comp-l12reg}, iBPDCA-SC1 and iBPDCA-SC2 exhibit
comparable performance in terms of the total number of {\sc Ssn} iterations. However, they now show a divergence in the CPU time especially when the problem size becomes larger. As highlighted in Remark \ref{rek-verif}, in this part of the experiments, we choose the kernel function as $\phi(\bm{x}) = \frac{1}{2}\|\bm{x}\|^2 + \frac{1}{2}\|A\bm{x}\|^2$ to facilitate the adaptability of the dual {\sc Ssn} method. But this choice also increases the cost of computing the associated Bregman distance, which becomes non-negligible, especially for large-scale problems. In this scenario, iBPDCA-SC2 appears to be more favorable as it can save the verification cost while still ensuring the convergence, as demonstrated in Tables \ref{Table1-l12con} and \ref{Table2-l12con}. This also validates the necessity of developing a more practically efficient inexact stopping criterion (SC2) to fit large-scale problems that often arise in real-world applications.

Finally, we test all the methods with two instances $(A,\bm{b})$ obtained from the data sets \texttt{mpg7} and \texttt{gisette} in the UCI data repository. For each data set, we choose $\kappa=\kappa_c\|\bm{b}\|$ with $\kappa_c\in\{0.05, 0.1, 0.2\}$. The numerical results are reported in Figure \ref{Figl1l2conreal}, where we plot the objective value $F_{\text{con}}(\bm{x}^k)$ against the computational time.
%where the normalized objective value is computed by $\frac{F_{\text{con}}(\bm{x}^k)-F_{\text{con}}^{\min}}{F_{\text{con}}(\bm{x}^0)-F_{\text{con}}^{\min}}$ with $F_{\text{con}}^{\min}$ being the minimum of the terminating objective values obtained among all methods.
The results further demonstrate the encouraging performance of our iBPDCA-SC1 and iBPDCA-SC2 on real data sets.

\begin{table}[ht]
\caption{The average computational results for each triple $(m, n, s)$ from 20 instances on the constrained $\ell_{1-2}$ sparse optimization problem with $\kappa=1.1\cdot\|0.01\cdot\widehat{\bm{n}}\|$. In the table, ``\texttt{obj}" denotes the objective function value, ``\texttt{feas}" denotes the violation of the constraint $\|A\bm{x}^k-\bm{b}\|-\kappa$, ``\texttt{rec}" denotes the recovery error $\frac{\|\bm{x}^k-\bm{x}_{\text{orig}}\|}{1+\|\bm{x}_{\text{orig}}\|}$, ``\texttt{iter}" denotes the number of iterations (the total number of the {\sc Ssn} iterations in iBPDCA is also given in the bracket), ``\texttt{time}" denotes the computational time, and ``\texttt{t0}" denotes the computational time used to obtain an initial point by SPGL1 using 200 iterations.}\label{Table1-l12con}
\centering \tabcolsep 5pt
\scalebox{1}
{\renewcommand\arraystretch{1}
\begin{tabular}[!]{lllllllll}
\hline
$(m,n,s)$ & {\tt method} & \texttt{obj} & \texttt{feas} & \texttt{rec}
& \texttt{iter} & \texttt{time} & \texttt{t0} \\
\hline \vspace{-4mm}\\
%(200,2000,40)
%&ESQMe     & 2.57e+01 & 1.07e-11 & 8.58e-03 & 3615 & 0.31 & 0.03        \\
%&iBPDCA-SC1 & 2.57e+01 & -1.82e-15 & 8.48e-03 & 11 (108) & 0.06 & 0.03    \\
%&iBPDCA-SC2 & 2.57e+01 & 4.27e-15 & 8.48e-03 & 11 (107) & 0.06 & 0.03    \\[5pt]
(500,5000,100)
&ESQMe     & 6.95e+1 & 5.11e-11 & 1.31e-2 & 9704 & 4.45 & 0.10     \\
&iBPDCA-SC1 & 6.95e+1 & 6.72e-14 & 9.33e-3 & 14 (182) & 0.94 & 0.10     \\
&iBPDCA-SC2 & 6.95e+1 & 7.00e-14 & 9.33e-3 & 14 (181) & 0.91 & 0.10     \\[5pt]
(1000,10000,200)
&ESQMe     & 1.49e+2 & 5.26e-09 & 3.51e-2 & 16278 & 75.86 & 0.84     \\
&iBPDCA-SC1 & 1.49e+2 & 4.73e-12 & 1.25e-2 & 17 (234) & 6.90 & 0.84     \\
&iBPDCA-SC2 & 1.49e+2 & 7.95e-12 & 1.25e-2 & 17 (236) & 6.34 & 0.84    \\[5pt]
(1500,15000,300)
&ESQMe     & 2.27e+2 & 2.60e-07 & 1.36e-1 & 18371 & 226.52 & 2.23    \\
&iBPDCA-SC1 & 2.25e+2 & 1.42e-11 & 1.66e-2 & 19 (279) & 20.10 & 2.23    \\
&iBPDCA-SC2 & 2.25e+2 & 8.39e-12 & 1.66e-2 & 19 (283) & 18.62 & 2.23   \\[5pt]
(2000,20000,400)
&ESQMe     & 3.24e+2 & 2.89e-07 & 2.70e-1 & 19621 & 439.05 & 4.04     \\
&iBPDCA-SC1 & 3.00e+2 & 1.67e-11 & 2.48e-2 & 24 (324) & 44.70 & 4.04     \\
&iBPDCA-SC2 & 3.00e+2 & 8.15e-12 & 2.48e-2 & 24 (327) & 41.15 & 4.04    \\[5pt]
(2500,25000,500)
&ESQMe     & 4.91e+2 & 1.01e-06 & 4.82e-1 & 19940 & 696.39 & 6.25    \\
&iBPDCA-SC1 & 3.80e+2 & 3.45e-12 & 3.78e-2 & 25 (323) & 72.30 & 6.25       \\
&iBPDCA-SC2 & 3.80e+2 & 3.47e-12 & 3.78e-2 & 25 (326) & 67.11 & 6.25    \\[5pt]
(3000,30000,600)
&ESQMe     & 7.83e+2 & 1.07e-06 & 7.54e-1 & 19690 & 994.75 & 8.97    \\
&iBPDCA-SC1 & 4.52e+2 & 6.54e-12 & 2.77e-2 & 23 (348) & 121.16 & 8.97       \\
&iBPDCA-SC2 & 4.52e+2 & 6.52e-12 & 2.77e-2 & 23 (347) & 112.05 & 8.97    \\[3pt]
\hline
\end{tabular}
}
\end{table}

\begin{table}[ht]
\caption{
Same as Table 3 but for the constrained $\ell_{1-2}$ sparse optimization problem with
$\kappa=2\cdot\|0.01\cdot\widehat{\bm{n}}\|$.
%The average computational results for each triple $(m, n, s)$ from 20 instances on the constrained $\ell_{1-2}$ sparse optimization problem with $\kappa=2\cdot\|0.01\cdot\widehat{\bm{n}}\|$. In the table, ``\texttt{obj}" denotes the objective function value, ``\texttt{feas}" denotes the violation of the constraint $\|A\bm{x}^k-\bm{b}\|-\kappa$, ``\texttt{rec}" denotes the recovery error $\frac{\|\bm{x}^k-\bm{x}_{\text{orig}}\|}{1+\|\bm{x}_{\text{orig}}\|}$, ``\texttt{iter}" denotes the number of iterations (the total number of the {\sc Ssn} iterations in iBPDCA is also given in the bracket), ``\texttt{time}" denotes the computational time, and ``\texttt{t0}" denotes the computational time used to obtain an initial point by SPGL1.
}\label{Table2-l12con}
\centering \tabcolsep 5pt
\scalebox{1}
{\renewcommand\arraystretch{1}
\begin{tabular}[!]{lllllllll}
\hline
$(m,n,s)$ & {\tt method} & \texttt{obj} & \texttt{feas} & \texttt{rec}
& \texttt{iter} & \texttt{time} & \texttt{t0} \\
\hline \vspace{-4mm}\\
%(200,2000,40)
%&ESQMe     & 2.56e+01 & 2.93e-11 & 9.82e-03 & 1551 & 0.13 & 0.03    \\
%&iBPDCA-SC1 & 2.56e+01 & 4.76e-15 & 9.76e-03 & 11 (78) & 0.05 & 0.03     \\
%&iBPDCA-SC2 & 2.56e+01 & 8.07e-15 & 9.76e-03 & 11 (78) & 0.05 & 0.03  \\[5pt]
(500,5000,100)
&ESQMe     & 6.94e+1 &-5.22e-14 & 1.12e-2 & 4909 & 2.30 & 0.11     \\
&iBPDCA-SC1 & 6.94e+1 & 4.57e-13 & 1.07e-2 & 15 (149) & 0.78 & 0.11     \\
&iBPDCA-SC2 & 6.94e+1 & 4.42e-13 & 1.07e-2 & 15 (151) & 0.77 & 0.11     \\[5pt]
(1000,10000,200)
&ESQMe     & 1.49e+2 & 3.73e-11 & 1.95e-2 & 10289 & 48.02 & 0.84     \\
&iBPDCA-SC1 & 1.49e+2 & 1.45e-11 & 1.37e-2 & 17 (189) & 5.48 & 0.84     \\
&iBPDCA-SC2 & 1.49e+2 & 5.50e-12 & 1.37e-2 & 17 (190) & 5.05 & 0.84   \\[5pt]
(1500,15000,300)
&ESQMe     & 2.25e+2 & 1.24e-10 & 3.34e-2 & 12798 & 157.66 & 2.18     \\
&iBPDCA-SC1 & 2.25e+2 & 1.42e-11 & 1.76e-2 & 18 (224) & 16.02 & 2.18     \\
&iBPDCA-SC2 & 2.25e+2 & 1.65e-11 & 1.76e-2 & 18 (226) & 14.70 & 2.18    \\[5pt]
(2000,20000,400)
&ESQMe     & 3.03e+2 & 3.50e-08 & 1.24e-1 & 17412 & 388.69 & 4.07     \\
&iBPDCA-SC1 & 3.00e+2 & 7.08e-12 & 2.59e-2 & 23 (266) & 35.92 & 4.07     \\
&iBPDCA-SC2 & 3.00e+2 & 5.16e-12 & 2.59e-2 & 23 (269) & 33.34 & 4.07    \\[5pt]
(2500,25000,500)
&ESQMe     & 4.34e+2 & 2.32e-07 & 3.22e-1 & 18441 & 646.75 & 6.27     \\
&iBPDCA-SC1 & 3.80e+2 & 3.11e-12 & 3.84e-2 & 25 (275) & 61.33 & 6.27        \\
&iBPDCA-SC2 & 3.80e+2 & 4.97e-12 & 3.84e-2 & 25 (278) & 56.88 & 6.27     \\[5pt]
(3000,30000,600)
&ESQMe     & 8.11e+2 & 3.92e-07 & 7.15e-1 & 18461 & 936.02 & 9.09       \\
&iBPDCA-SC1 & 4.51e+2 & 1.07e-11 & 2.83e-2 & 23 (297) & 102.13 & 9.09        \\
&iBPDCA-SC2 & 4.51e+2 & 7.72e-12 & 2.83e-2 & 23 (301) & 95.87 & 9.09    \\[3pt]
\hline
\end{tabular}
}
\end{table}

\begin{figure}[ht]
\centering
\subfigure[\texttt{mpg7}, where $A$ is of size $392\times3432$ with $\lambda_{\max}(A^{\top}A)\approx1.28\times10^4$]{
\includegraphics[width=5cm]{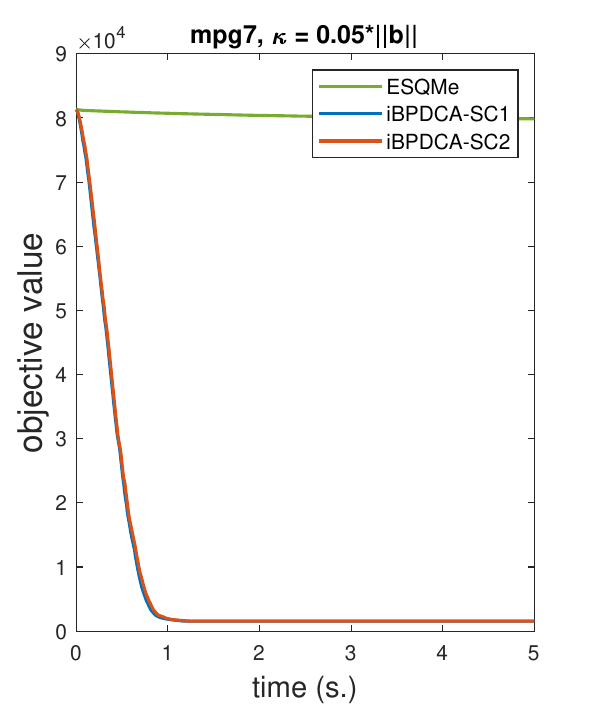}
\includegraphics[width=5cm]{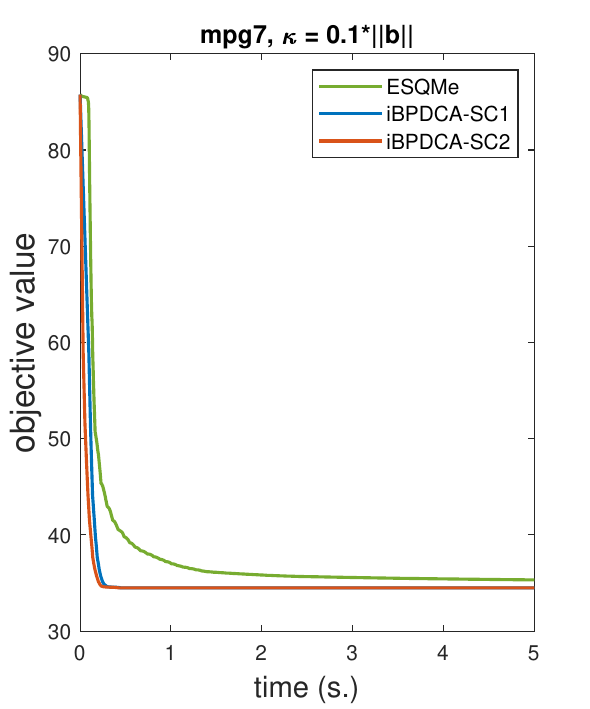}
\includegraphics[width=5cm]{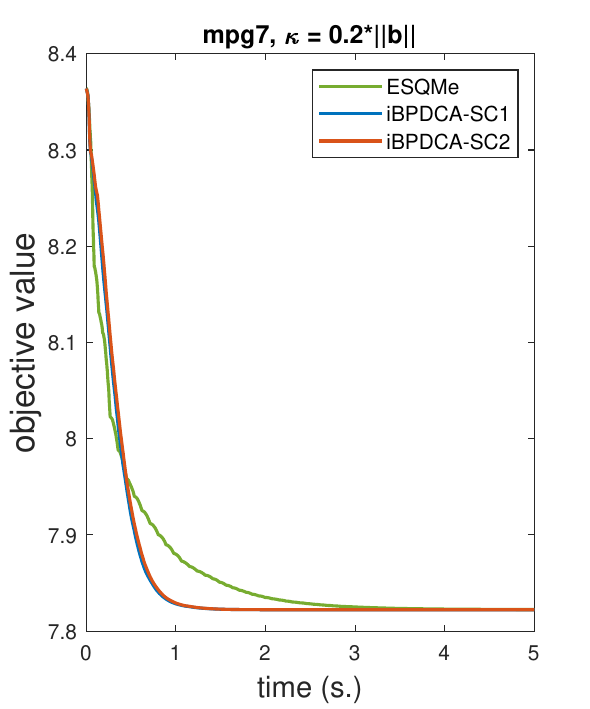}
}
\subfigure[\texttt{gisette}, where $A$ is of size $1000\times4971$ with $\lambda_{\max}(A^{\top}A)\approx3.36\times10^6$]{
\includegraphics[width=5cm]{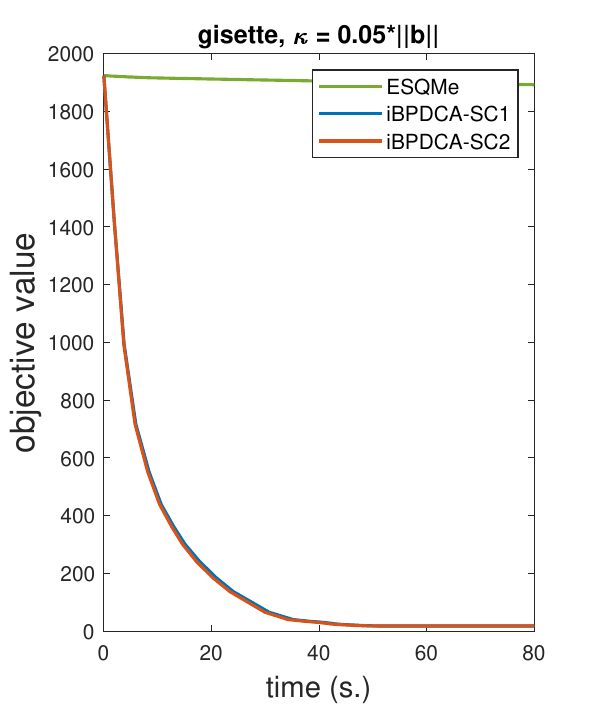}
\includegraphics[width=5cm]{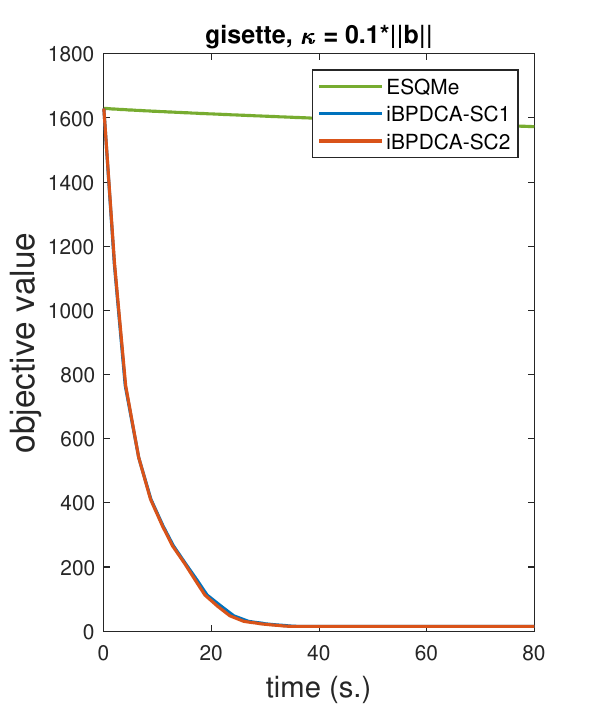}
\includegraphics[width=5cm]{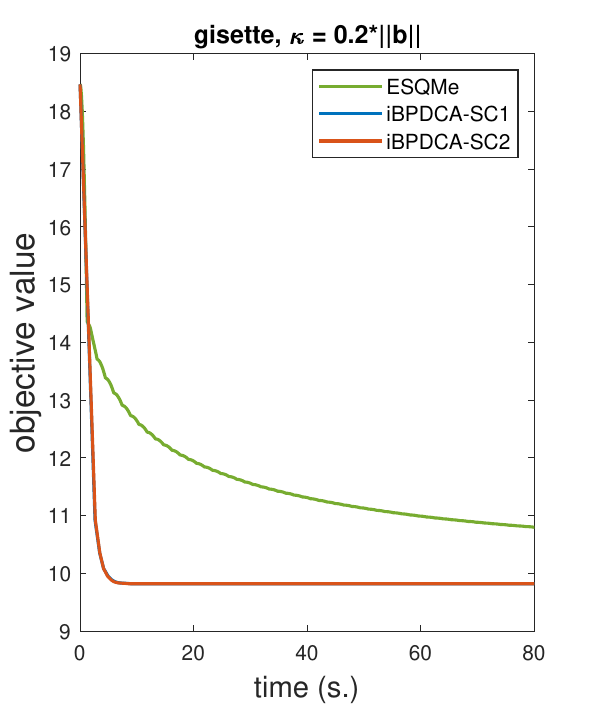}
}
\caption{Numerical results of ESQMe, iBPDCA-SC1 and iBPDCA-SC2 for solving the constrained $\ell_{1-2}$ sparse optimization problem on \texttt{mpg7} and \texttt{gisette} from the UCI data repository.}\label{Figl1l2conreal}
\end{figure}

%%%%%%%%%%%%%%%%%%%%%%%%%%%%%%%%%%%%%%%%%%%%%%%%%%%%%%%%
\section{Conclusions}\label{sec-conc}

In this paper, we developed an inexact Bregman proximal DC algorithm (iBPDCA) based on two types of relative stopping criteria for solving a class of DC optimization problems in the form of \eqref{DCpro}, and established both global subsequential and global sequential convergence results under suitable conditions. In comparison to existing BPDCA and many other proximal DC-type algorithms, our iBPDCA offers a comprehensive inexact algorithmic framework that can encompass many of these existing methods and accommodate different types of errors that may occur when solving the subproblem. The inherent versatility and flexibility of our iBPDCA readily enable its potential application across different DC decompositions, thereby facilitating the design of a more efficient DCA scheme. We conducted some numerical experiments on the $\ell_{1-2}$ regularized least squares problem and the constrained $\ell_{1-2}$ sparse optimization problem to demonstrate the efficiency of our iBPDCA with different types of stopping criteria.

Note that, in \cite{tft2022new}, the authors also proposed a Bregman proximal DC algorithm with extrapolation (BPDCAe), which demonstrated better numerical performance compared to BPDCA. Exploring the potential extension of the BPDCAe algorithm to an inexact version is both interesting and valuable. However, such an extension would require substantial additional effort and is beyond the scope of the current paper. We will consider this as a future research topic.

%%%%%%%%%%%%%%%%%%%%%%%%%%%%%%%%%%%%%%%%%%
% Acknowledgments here
\section*{Acknowledgments}

We thank the editor and referees for their valuable suggestions and comments, which have helped to improve the quality of this paper.

%%%%%%%%%%%%%%%%%%%%%%%%%%%%%%%%%%%%%%%%%%%%%%%%%%%%%%%%
\appendix

\section{Missing proofs and derivations in Section \ref{sec-num}}

%%%%%%%%%%%%%%%%%%%%%%%%%%%%%%%%%%%%%%%%%%
\subsection{Derivation of the dual problem \eqref{subprobdual_l12reg}}\label{apd-subprobdual_l12reg}

First, the associated Lagrangian function of \eqref{subprobref_l12reg} is given by
\begin{equation*}
\begin{aligned}
\mathcal{L}(\bm{x},\bm{y},\bm{z})
&=\lambda\|\bm{x}\|_1-\langle\bm{\xi}^k,\bm{x}\rangle+\frac{1}{2}\|\bm{y}-\bm{b}\|^2
+\frac{\gamma_k}{2}\|\bm{x}-\bm{x}^k\|^2+\langle\bm{z},\bm{\bm{A}}\bm{x}-\bm{y}\rangle \\
&= \lambda\|\bm{x}\|_1
+ \frac{\gamma_k}{2}\|\bm{x}-\bm{v}_k(\bm{z})\|^2
  -\frac{\gamma_k}{2}\left\|\bm{v}_k(\bm{z})\right\|^2
+ \frac{\gamma_k}{2} \|\bm{x}^k\|^2
+ \frac{1}{2}\|\bm{y}-\bm{b}-\bm{z}\|^2 \\
&\qquad -\frac{1}{2}\|\bm{z}\|^2 - \langle\bm{z},\,\bm{b}\rangle,
\end{aligned}
\end{equation*}
where $\bm{z}\in\mathbb{R}^m$ is the dual variable and $\bm{v}_k(\bm{z}) := \gamma_k^{-1}\bm{\xi}^k+\bm{x}^k-\gamma_k^{-1}A^{\top}\bm{z}$.
%It can be further equivalently written as
%\begin{equation*}
%\mathcal{L}(\bm{x},\bm{y},\bm{z})=\lambda\|\bm{x}\|_1
%+\frac{\gamma_k}{2}\|\bm{x}-\bm{v}_k(\bm{z})\|^2- \frac{\gamma_k}{2}\left\|\bm{v}_k(\bm{z})\right\|^2
%+ \frac{\gamma_k}{2} \|\bm{x}^k\|^2+\frac{1}{2}\|\bm{y}-\bm{b}-\bm{z}\|^2-\frac{1}{2}\|\bm{z}\|^2 - \langle\bm{z},\,\bm{b}\rangle,
%\end{equation*}
%where $\bm{v}_k(\bm{z}) := \gamma_k^{-1}\bm{\xi}^k+\bm{x}^k
%-\gamma_k^{-1}A^{\top}\bm{z}$.
Then, the dual problem of \eqref{subprobref_l12reg} is give by
\begin{equation*}
\max\limits_{\bm{z}\in\mathbb{R}^m}\left\{\,\min_{\bm{x}\in\mathbb{R}^n,\,\bm{y}\in\mathbb{R}^m}
~\mathcal{L}(\bm{x},\bm{y},\bm{z})\,\right\}.
\end{equation*}
Observe that
\begin{equation*}
\begin{aligned}
&~\min_{\bm{x}\in\mathbb{R}^n,\,\bm{y}\in\mathbb{R}^m}~\mathcal{L}(\bm{x},\bm{y},\bm{z})\\
%&=\min_{\bm{x}\in\mathbb{R}^n,\,\bm{y}\in\mathbb{R}^m}\{\lambda\|\bm{x}\|_1
%+\frac{\gamma_k}{2}\|\bm{x}-\bm{v}_k(\bm{z})\|^2
%- \frac{\gamma_k}{2}\left\|\bm{v}_k(\bm{z})\right\|^2
%+ \frac{\gamma_k}{2} \|\bm{x}^k\|^2 \\
%&\qquad\qquad\qquad+\frac{1}{2}\|\bm{y}-\bm{b}-\bm{z}\|^2-\frac{1}{2}\|\bm{z}\|^2 - \langle\bm{z},\,\bm{b}\rangle\}\\
&=\min_{\bm{x}\in\mathbb{R}^n}\left\{\lambda\|\bm{x}\|_1
+ \frac{\gamma_k}{2}\|\bm{x}-\bm{v}_k(\bm{z})\|^2\right\}
- \frac{\gamma_k}{2}\left\|\bm{v}_k(\bm{z})\right\|^2
+ \frac{\gamma_k}{2} \|\bm{x}^k\|^2 \\
&\qquad + \min_{\bm{y}\in\mathbb{R}^m}\left\{\frac{1}{2}\|\bm{y}-\bm{b}-\bm{z}\|^2
\right\}
-\frac{1}{2}\|\bm{z}\|^2 - \langle\bm{z},\,\bm{b}\rangle \\
&=\lambda\left\|\texttt{prox}_{\lambda\gamma_k^{-1}\|\cdot\|_1}\!
\left(\bm{v}_k(\bm{z})\right)\right\|_1+\frac{\gamma_k}{2} \left\|\texttt{prox}_{\lambda\gamma_k^{-1}\|\cdot\|_1}\!
\left(\bm{v}_k(\bm{z})\right) - \bm{v}_k(\bm{z})\right\|^2  \\
&\qquad - \frac{\gamma_k}{2}\left\|\bm{v}_k(\bm{z})\right\|^2
+ \frac{\gamma_k}{2}\|\bm{x}^k\|^2
- \frac{1}{2}\|\bm{z}\|^2 - \langle\bm{z},\,\bm{b}\rangle.
\end{aligned}
\end{equation*}
%where the second equality follows from the definition of the proximal mapping of $\lambda\gamma_k^{-1}\|\cdot\|_1$ \blue{(here, $\texttt{prox}_{\lambda\gamma_k^{-1}\|\cdot\|_1}(\bm{u})=\operatorname{sgn}(\bm{u}) \circ \max \left\{|\bm{u}|-\lambda\gamma_k^{-1}, 0\right\}$)} and the properties of the quadratic function.
Thus, the dual problem of \eqref{subprobref_l12reg} admits the following form
\begin{equation*}
\hspace{-2mm}
\max\limits_{\bm{z}\in\mathbb{R}^m}\!
\left\{\begin{aligned}
&
-\frac{1}{2}\|\bm{z}\|^2 - \langle\bm{z},\,\bm{b}\rangle
+\lambda\left\|\texttt{prox}_{\lambda\gamma_k^{-1}\|\cdot\|_1}\!
\left(\bm{v}_k(\bm{z})\right)\right\|_1 \\[3pt]
&~~+\frac{\gamma_k}{2} \left\|\texttt{prox}_{\lambda\gamma_k^{-1}\|\cdot\|_1}\!
\left(\bm{v}_k(\bm{z})\right)
- \bm{v}_k(\bm{z})\right\|^2
- \frac{\gamma_k}{2}\left\|\bm{v}_k(\bm{z})\right\|^2
+ \frac{\gamma_k}{2} \|\bm{x}^k\|^2
\end{aligned}\right\},
\end{equation*}
which is equivalent to \eqref{subprobdual_l12reg}, when expressed in a minimization form.
%\begin{equation*}
%\hspace{-2mm}
%\min\limits_{\bm{z}\in\mathbb{R}^m}\!
%\left\{\begin{aligned}
%&\Psi_k^{\text{reg}}(\bm{z}):=
%\frac{1}{2}\|\bm{z}\|^2 + \langle\bm{z},\,\bm{b}\rangle
%-\lambda\left\|\texttt{prox}_{\lambda\gamma_k^{-1}\|\cdot\|_1}\!
%\left(\bm{v}_k(\bm{z})\right)\right\|_1 \\[3pt]
%&~~-\frac{\gamma_k}{2} \left\|\texttt{prox}_{\lambda\gamma_k^{-1}\|\cdot\|_1}\!
%\left(\bm{v}_k(\bm{z})\right)
%- \bm{v}_k(\bm{z})\right\|^2
%+ \frac{\gamma_k}{2}\left\|\bm{v}_k(\bm{z})\right\|^2
%- \frac{\gamma_k}{2} \|\bm{x}^k\|^2
%\end{aligned}\right\}.
%\end{equation*}

%%%%%%%%%%%%%%%%%%%%%%%%%%%%%%%%%%%%%%%%%%%%%%%%%%%%%%%
\subsection{Derivation of the dual problem \eqref{dualprobleml12con}}\label{apd-dualprobleml12con}

For notational simplicity, let $g(\cdot):=\|\cdot\|_1+\iota_{M}(\cdot)$, $\bm{s}^k:=\bm{x}^k+\gamma_k^{-1}\bm{\xi}^k$ and $\bm{b}^k:=A\bm{x}^k-\bm{b}$. Then, the associated Lagrangian function of \eqref{subprobref_l12con} is given by
\begin{equation*}
\begin{aligned}
\mathcal{L}(\bm{x},\bm{y},\bm{z})
&= g(\bm{x})
+ \frac{\gamma_k}{2}\left\|\bm{x}-\bm{s}^k\right\|^2
+ \iota_{\kappa}(\bm{y})
+ \frac{\gamma_k}{2}\left\|\bm{y} -\bm{b}^k\right\|^2
+ \langle\bm{z},\bm{\bm{A}}\bm{x}-\bm{y}-\bm{b}\rangle \\
&= g(\bm{x})
+\frac{\gamma_k}{2}\left\|\bm{x}-\big(\bm{s}^k-\gamma_k^{-1}A^{\top}\bm{z}\big)\right\|^2 - \frac{\gamma_k}{2}\left\|\bm{s}^k-\gamma_k^{-1} A^{\top} \bm{z}\right\|^2
+ \frac{\gamma_k}{2}\|\bm{s}^k\|^2 \\
&\qquad + \iota_{\kappa}(\bm{y}) +\frac{\gamma_k}{2}\left\|\bm{y}-\big(\bm{b}^k+\gamma_k^{-1}\bm{z}\big)\right\|^2  -\frac{\gamma_k}{2}\left\|\bm{b}^k+\gamma_k^{-1}\bm{z}\right\|^2 +\frac{\gamma_k}{2}\|\bm{b}^k\|^2 - \langle\bm{z},\,\bm{b}\rangle,
\end{aligned}
\end{equation*}
where $\bm{z}\in\mathbb{R}^m$ is the dual variable.
%It can be further equivalently written as
%\begin{equation*}
%\begin{aligned}
%\mathcal{L}(\bm{x},\bm{y},\bm{z})&=g(\bm{x})
%    +\frac{\gamma_k}{2}\left\|\bm{x}-\big(\bm{s}^k-\gamma_k^{-1}A^{\top}\bm{z}\big)\right\|^2 -\frac{\gamma_k}{2}\left\|\bm{s}^k-\gamma_k^{-1} A^{\top} \bm{z}\right\|^2 + \frac{\gamma_k}{2}\|\bm{s}^k\|^2 \\
%    &\quad + \iota_{\kappa}(\bm{y}) +\frac{\gamma_k}{2}\left\|\bm{y}-\big(\bm{b}^k+\gamma_k^{-1}\bm{z}\big)\right\|^2  -\frac{\gamma_k}{2}\left\|\bm{b}^k+\gamma_k^{-1}\bm{z}\right\|^2 +\frac{\gamma_k}{2}\|\bm{b}^k\|^2 - \langle\bm{z},\,\bm{b}\rangle,
%\end{aligned}
%\end{equation*}
Then, the dual problem of \eqref{subprobref_l12con} is give by
\begin{equation*}
\max\limits_{\bm{z}\in\mathbb{R}^m}\left\{\,\min_{\bm{x}\in\mathbb{R}^n,\,\bm{y}\in\mathbb{R}^m}
~\mathcal{L}(\bm{x},\bm{y},\bm{z})\,\right\}.
\end{equation*}
Observe that
\begin{equation*}
\begin{aligned}
&~\min_{\bm{x}\in\mathbb{R}^n,\,\bm{y}\in\mathbb{R}^m}~\mathcal{L}(\bm{x},\bm{y},\bm{z})\\
%&=\min_{\bm{x}\in\mathbb{R}^n,\,\bm{y}\in\mathbb{R}^m}~\{g(\bm{x})
%    +\frac{\gamma_k}{2}\left\|\bm{x}-\big(\bm{s}^k-\gamma_k^{-1}A^{\top}\bm{z}\big)\right\|^2 -\frac{\gamma_k}{2}\left\|\bm{s}^k-\gamma_k^{-1} A^{\top} \bm{z}\right\|^2 + \frac{\gamma_k}{2}\|\bm{s}^k\|^2 \\
%    &\quad\quad+ \iota_{\kappa}(\bm{y}) +\frac{\gamma_k}{2}\left\|\bm{y}-\big(\bm{b}^k+\gamma_k^{-1}\bm{z}\big)\right\|^2  -\frac{\gamma_k}{2}\left\|\bm{b}^k+\gamma_k^{-1}\bm{z}\right\|^2 +\frac{\gamma_k}{2}\|\bm{b}^k\|^2 - \langle\bm{z},\,\bm{b}\rangle\}\\
&=\min_{\bm{x}\in\mathbb{R}^n}
\left\{g(\bm{x})
+\frac{\gamma_k}{2}\left\|\bm{x}-\big(\bm{s}^k-\gamma_k^{-1}A^{\top}\bm{z}\big)\right\|^2 \right\}
- \frac{\gamma_k}{2}\left\|\bm{s}^k-\gamma_k^{-1} A^{\top} \bm{z}\right\|^2
+ \frac{\gamma_k}{2}\|\bm{s}^k\|^2 \\
&\qquad + \min_{\bm{y}\in\mathbb{R}^m}
\left\{\iota_{\kappa}(\bm{y}) +\frac{\gamma_k}{2}\left\|\bm{y}-\big(\bm{b}^k+\gamma_k^{-1}\bm{z}\big)\right\|^2
\right\}
- \frac{\gamma_k}{2}\left\|\bm{b}^k+\gamma_k^{-1}\bm{z}\right\|^2
+ \frac{\gamma_k}{2}\|\bm{b}^k\|^2
- \langle\bm{z},\,\bm{b}\rangle \\
&= \left\|\texttt{prox}_{\gamma_k^{-1}g}\left(\bm{s}^k-\gamma_k^{-1} A^{\top} \bm{z}\right)\right\|_1+\frac{\gamma_k}{2}\left\|\texttt{prox}_{\gamma_k^{-1}g}\left(\bm{s}^k-\gamma_k^{-1} A^{\top} \bm{z}\right)
-\left(\bm{s}^k-\gamma_k^{-1}A^{\top}\bm{z}\right)\right\|^2 \\
&\quad-\frac{\gamma_k}{2}\left\|\bm{s}^k-\gamma_k^{-1} A^{\top} \bm{z}\right\|^2 + \frac{\gamma_k}{2}\|\bm{s}^k\|^2\\
&\quad+\frac{\gamma_k}{2}\left\|\Pi_{\kappa}\left(\bm{b}^k+\gamma_k^{-1}\bm{z}\right)
-\left(\bm{b}^k+\gamma_k^{-1}\bm{z}\right)\right\|^2 -\frac{\gamma_k}{2}\left\|\bm{b}^k+\gamma_k^{-1}\bm{z}\right\|^2 +\frac{\gamma_k}{2}\|\bm{b}^k\|^2 - \langle\bm{z},\,\bm{b}\rangle,
\end{aligned}
\end{equation*}
where $\texttt{prox}_{\gamma_k^{-1}g}$ is the proximal mapping of $\gamma_k^{-1}g$ and $\Pi_{\kappa}$ is the projection operator over $\{\bm{x} \in \mathbb{R}^n:\|\bm{x}\|\leq \kappa\}$.
%Here, one can verify that
%\begin{equation*}
%\texttt{prox}_{\gamma_k^{-1}g}(\bm{u})
%=\min\left\{\max\left\{\operatorname{sgn}(\bm{u})\circ\max \left\{|\bm{u}|-\gamma_k^{-1}, 0\right\},\,-M\right\},\,M\right\}
%\end{equation*}
%and
%\begin{equation*}
%\Pi_{\kappa}(\bm{u})= \begin{cases}\bm{u}, & \text { if }\|\bm{u}\| \leq \kappa, \\ \frac{\kappa}{\|\bm{u}\|} \bm{u}, & \text { if }\|\bm{u}\|>\kappa.\end{cases}
%\end{equation*}
Thus, the dual problem of \eqref{subprobref_l12con} admits the following form
\begin{equation*}
\max _{\bm{z} \in \mathbb{R}^m}\left\{\begin{aligned}
&-\langle\bm{z}, \bm{b}\rangle
-\frac{\gamma_k}{2}\left\|\bm{s}^k-\gamma_k^{-1} A^{\top} \bm{z}\right\|^2
+\left\|\texttt{prox}_{\gamma_k^{-1}g}\left(\bm{s}^k-\gamma_k^{-1} A^{\top} \bm{z}\right)\right\|_1\\[3pt]
&\qquad +\frac{\gamma_k}{2}\left\|\texttt{prox}_{\gamma_k^{-1}g}\left(\bm{s}^k-\gamma_k^{-1} A^{\top} \bm{z}\right)
-\left(\bm{s}^k-\gamma_k^{-1}A^{\top}\bm{z}\right)\right\|^2\\[3pt]
&\qquad -\frac{\gamma_k}{2}\left\|\bm{b}^k+\gamma_k^{-1}\bm{z}\right\|^2
+\frac{\gamma_k}{2}\left\|\Pi_{\kappa}\left(\bm{b}^k+\gamma_k^{-1}\bm{z}\right)
-\left(\bm{b}^k+\gamma_k^{-1}\bm{z}\right)\right\|^2 \\[3pt]
&\qquad +\frac{\gamma_k}{2}\|\bm{s}^k\|^2
+\frac{\gamma_k}{2}\|\bm{b}^k\|^2
\end{aligned}\right\},
\end{equation*}
which is equivalent to \eqref{dualprobleml12con}, when expressed in a minimization form.
%\begin{equation*}
%\min _{\bm{z} \in \mathbb{R}^m}\left\{\begin{aligned}
%\Psi_k^{\text{con}}(\bm{z})
%&:= \langle\bm{z}, \bm{b}\rangle
%+\frac{\gamma_k}{2}\left\|\bm{s}^k-\gamma_k^{-1} A^{\top} \bm{z}\right\|^2
%-\left\|\texttt{prox}_{\gamma_k^{-1}g}\left(\bm{s}^k-\gamma_k^{-1} A^{\top} \bm{z}\right)\right\|_1\\[3pt]
%&\qquad -\frac{\gamma_k}{2}\left\|\texttt{prox}_{\gamma_k^{-1}g}\left(\bm{s}^k-\gamma_k^{-1} A^{\top} \bm{z}\right)
%-\left(\bm{s}^k-\gamma_k^{-1}A^{\top}\bm{z}\right)\right\|^2\\[3pt]
%&\qquad +\frac{\gamma_k}{2}\left\|\bm{b}^k+\gamma_k^{-1}\bm{z}\right\|^2
%-\frac{\gamma_k}{2}\left\|\Pi_{\kappa}\left(\bm{b}^k+\gamma_k^{-1}\bm{z}\right)
%-\left(\bm{b}^k+\gamma_k^{-1}\bm{z}\right)\right\|^2 \\[3pt]
%&\qquad -\frac{\gamma_k}{2}\|\bm{s}^k\|^2
%-\frac{\gamma_k}{2}\|\bm{b}^k\|^2
%\end{aligned}\right\}.
%\end{equation*}

%%%%%%%%%%%%%%%%%%%%%%%%%%%%%%%%%%%%%%%%%%%
\subsection{Proof of Proposition \ref{pro-scnew-l12reg}}\label{apd-pro-l12reg}
\begin{proof}
From the definition of $\bm{w}^{k, t}$ and the property of the proximal mapping $\texttt{prox}_{\lambda\gamma_k^{-1}\|\cdot\|_1}$, we have
\begin{equation*}
0 \in  \lambda\gamma_k^{-1}\partial\|\bm{w}^{k,t}\|_1
+ \bm{w}^{k,t} - \big(\gamma_k^{-1}\bm{\xi}^k+\bm{x}^k-\gamma_k^{-1}A^{\top}\bm{z}^{k,t}\big),
\end{equation*}
which, together with $\bm{e}^{k,t} := \nabla\Psi_k(\bm{z}^{k,t})
= -A\bm{w}^{k,t}+\bm{z}^{k,t}+\bm{b}$, deduces that
\begin{equation*}
0 \in \lambda\partial\|\bm{w}^{k,t}\|_1
+ \gamma_k(\bm{w}^{k,t}-\bm{x}^k) - \bm{\xi}^k
+ A^{\top}\big(\bm{e}^{k,t} + A\bm{w}^{k,t}-\bm{b}\big),
\end{equation*}
and hence
\begin{equation*}
-A^{\top}\bm{e}^{k,t} \in \lambda\partial\|\bm{w}^{k,t}\|_1
- \bm{\xi}^k + A^{\top}(A\bm{w}^{k,t}-\bm{b})
+ \gamma_k(\bm{w}^{k,t}-\bm{x}^k).
\end{equation*}
It is then easy to see from the above relation that the point $\bm{w}^{k,t}$ associated with the error pair $(-A^{\top}\bm{e}^{k,t}, \,0)$ satisfies condition \eqref{iBPDCA-inexcond}. Using this relation, we can readily obtain the desired results.
%Moreover, when
%\begin{equation*}
%\|A^{\top}\bm{e}^{k,t}\|^2
%+ |\langle A^{\top} \bm{e}^{k,t}, \,\bm{w}^{k,t}-\bm{x}^{k}\rangle|
%\leq \frac{\sigma\gamma_k}{2}\|\bm{w}^{k,t}-\bm{x}^{k}\|^2,
%\end{equation*}
%we can obtain that the inexact stopping criterion (SC1) holds for $\bm{x}^{k+1}:=\bm{w}^{k,t}$, $\Delta^k:=-A^{\top}\bm{e}^{k,t}$ and $\delta_k:=0$. Similarly, when
%\begin{equation*}%\label{SC2new_l12reg}
%\|A^{\top}\bm{e}^{k,t}\|^2 +
%|\langle A^{\top} \bm{e}^{k,t}, \,\bm{w}^{k,t}-\bm{x}^{k}\rangle|
%\leq \frac{\sigma\gamma_k}{2}\|\bm{x}^{k}-\bm{x}^{k-1}\|^2,
%\end{equation*}
%we can obtain that the inexact stopping criterion (SC2) holds for $\bm{x}^{k+1}:=\bm{w}^{k,t}$, $\Delta^k:=-A^{\top}\bm{e}^{k,t}$ and $\delta_k:=0$. This completes the proof.
\end{proof}

%%%%%%%%%%%%%%%%%%%%%%%%%%%%%%%%%%%%%%%%%%%%%%%%%%%%%%%%%
\subsection{Proof of Proposition \ref{pro-scnew-l12con}}\label{apd-pro-l12con}

\begin{proof}
From the definition of $\bm{w}^{k,t}$ and the property of the proximal mapping $\texttt{prox}_{\gamma_k^{-1}g}$, we have that
\begin{equation*}
0 \in \partial g(\bm{w}^{k,t})+\gamma_k\left[\bm{w}^{k,t}-\left({\bm{x}}^{k}+\gamma_k^{-1} \bm{\xi}^k-\gamma_k^{-1} A^{\top} {\bm{z}}^{k,t}\right)\right].
\end{equation*}
Let $\bm{d}_1^{k,t}:=\gamma_k\left[\big({\bm{x}}^{k}+\gamma_k^{-1} \bm{\xi}^k-\gamma_k^{-1} A^{\top}\bm{z}^{k,t}\big)-\bm{w}^{k,t}\right]$. Clearly, $\bm{d}_1^{k,t} \in \partial g(\bm{w}^{k,t})$. Then, for any $\bm{u} \in \mathrm{dom}\,g$, we have that
\begin{equation*}
\begin{aligned}
\quad g(\bm{u})
&\geq g(\bm{w}^{k,t}) + \langle{\bm{d}}_1^{k,t}, \,\bm{u}-\bm{w}^{k,t}\rangle \\
&\geq g(\widetilde{\bm{w}}^{k,t}) + \langle{\bm{d}}_1^{k,t}, \,\bm{u}-\widetilde{\bm{w}}^{k,t}\rangle
- \big( g(\widetilde{\bm{w}}^{k,t}) - g(\bm{w}^{k,t})
- \langle{\bm{d}}_1^{k,t},\,\widetilde{\bm{w}}^{k,t}-\bm{w}^{k,t}\rangle\big),
\end{aligned}
\end{equation*}
which implies that $\bm{d}_1^{k,t} \in \partial_{\delta_{k,t}^1} g(\widetilde{\bm{w}}^{k,t})$ with $\delta^1_{k,t} := g(\widetilde{\bm{w}}^{k,t}) - g(\bm{w}^{k,t})
- \langle{\bm{d}}_1^{k,t},\,\widetilde{\bm{w}}^{k,t}-\bm{w}^{k,t}\rangle\geq0$ (due to the convexity of $g$ and $\bm{d}_1^{k,t} \in \partial g(\bm{w}^{k,t})$\brown{)}. Thus, we have
\begin{equation}\label{d1partialg}
0 \in \partial_{\delta_{k,t}^1} g(\widetilde{\bm{w}}^{k,t})
- \bm{\xi}^k + A^{\top}\bm{z}^{k,t} + \gamma_k(\bm{w}^{k,t}-\bm{x}^{k}).
\end{equation}
On the other hand, we see from the properties of projection onto the convex set $\left\{\bm{x} \in \mathbb{R}^n:\|\bm{x}\|\leq \kappa\right\}$ that, for any $\bm{u}$ satisfying $\|\bm{u}\|\leq\kappa$,
\begin{equation*}%\label{projection-ballsigma}
\big\langle \bm{u}-\Pi_{\kappa}\big(A \bm{x}^k-\bm{b}+\gamma_k^{-1}\bm{z}^{k,t}\big), \,\big(A{\bm{x}}^k-\bm{b}+\gamma_k^{-1}\bm{z}^{k,t}\big)
- \Pi_{\kappa}\big(A\bm{x}^k-\bm{b}+\gamma_k^{-1}\bm{z}^{k,t}\big)\big\rangle
\leq 0.
\end{equation*}
Let $\bm{d}_2^{k,t}:=\gamma_k\left[\big(A{\bm{x}}^k-\bm{b}+\gamma_k^{-1}\bm{z}^{k,t}\big)
- \Pi_{\kappa}\big(A\bm{x}^k-\bm{b}+\gamma_k^{-1}\bm{z}^{k,t}\big)\right]$ and recall the definition of $\bm{e}^{k, t}$ that $\Pi_{\kappa}\big(A\bm{x}^k-\bm{b}+\gamma_k^{-1}\bm{z}^{k,t}\big)=A \bm{w}^{k,t}-\bm{b}+\bm{e}^{k,t}$. Then, we can see from the above relation that
\begin{equation*}
\begin{aligned}
0
&\geq \big\langle \bm{u} - \big(A\bm{w}^{k,t}-\bm{b}+\bm{e}^{k,t}\big),
\,\bm{d}_2^{k,t}\big\rangle  \\
&= \langle \bm{u}-(A\widetilde{\bm{w}}^{k,t}-\bm{b}), \,\bm{d}_2^{k,t}\rangle
- \langle \bm{e}^{k,t} - A(\widetilde{\bm{w}}^{k,t}-\bm{w}^{k,t}), \,\bm{d}_2^{k,t}\rangle,
\end{aligned}
\end{equation*}
and hence
\begin{equation*}
\langle \bm{u}-(A\widetilde{\bm{w}}^{k,t}-\bm{b}), \,\bm{d}_2^{k,t}\rangle
\leq \langle \bm{e}^{k,t} - A(\widetilde{\bm{w}}^{k,t}-\bm{w}^{k,t}), \,\bm{d}_2^{k,t}\rangle
\leq \big|\langle \bm{e}^{k,t} - A(\widetilde{\bm{w}}^{k,t}-\bm{w}^{k,t}), \,\bm{d}_2^{k,t}\rangle\big|.
\end{equation*}
This, together with $\|A\widetilde{\bm{w}}^{k,t}-\bm{b}\|\leq\kappa$, implies that $\bm{d}_2^{k,t} \in \partial_{\delta_{k,t}^2}\iota_{\kappa}(A \widetilde{\bm{w}}^{k,t}-\bm{b})$ with $\delta_{k,t}^2:=|\langle \bm{e}^{k,t} - A(\widetilde{\bm{w}}^{k,t}-\bm{w}^{k,t}), \,\bm{d}_2^{k,t}\rangle|$. Thus, we have
\begin{equation}\label{d2partial-ballsigma}
-\gamma_k\bm{e}^{k,t}
\in \partial_{\delta_{k,t}^2}\iota_{\kappa}(A\widetilde{\bm{w}}^{k,t}-\bm{b})
- \bm{z}^{k,t} + \gamma_kA(\bm{w}^{k,t}-\bm{x}^{k}).
\end{equation}
Now, combining \eqref{d1partialg} and \eqref{d2partial-ballsigma}, we can obtain that
\begin{equation*}
-\gamma_kA^{\top}\bm{e}^{k,t}
\in \partial_{\delta_{k,t}^1}g(\widetilde{\bm{w}}^{k,t})
+ A^{\top}\partial_{\delta_{k,t}^2}\iota_{\kappa}(A\widetilde{\bm{w}}^{k,t}-\bm{b})
- \bm{\xi}^k + \gamma_k\big(\bm{w}^{k,t}-\bm{x}^{k}\big)
+ \gamma_kA^{\top}A\big(\bm{w}^{k,t}-\bm{x}^{k}\big),
\end{equation*}
which further implies that
\begin{equation*}
\begin{aligned}
\Delta^{k,t}
&:=-\gamma_kA^{\top}\bm{e}^{k,t} + \gamma_k(\widetilde{\bm{w}}^{k,t}-\bm{w}^{k,t})
+ \gamma_k A^{\top}A(\widetilde{\bm{w}}^{k,t}-\bm{w}^{k,t}) \\
&\in \partial_{\delta_{k,t}^1}g(\widetilde{\bm{w}}^{k,t})
+ A^{\top}\partial_{\delta_{k,t}^2}\iota_{\kappa}(A\widetilde{\bm{w}}^{k,t}-\bm{b})
- \bm{\xi}^k + \gamma_k\big(\widetilde{\bm{w}}^{k,t}-\bm{x}^{k}\big)
+ \gamma_kA^{\top}A\big(\widetilde{\bm{w}}^{k,t}-\bm{x}^{k}\big)  \\
&\subset\partial_{\delta_{k,t}}\left[g(\cdot)+\iota_{\kappa}(A\cdot-\bm{b})\right](\widetilde{\bm{w}}^{k,t})
- \bm{\xi}^k + \gamma_k\big(\widetilde{\bm{w}}^{k,t}-\bm{x}^{k}\big)
+ \gamma_kA^{\top}A\big(\widetilde{\bm{w}}^{k,t}-\bm{x}^{k}\big),
\end{aligned}
\end{equation*}
where $\delta_{k,t} := \delta^1_{k,t} + \delta^2_{k,t}$ and the last inclusion can be verified by the definition of the $\varepsilon$-subdifferential of $g(\cdot)+\iota_{\kappa}(A\cdot-\bm{b})$ at $\widetilde{\bm{w}}^{k,t}$. Then, one can see from the above relation that the point $\widetilde{\bm{w}}^{k,t}$ associated with the error pair $(\Delta^{k,t}, \,\delta_{k,t})$ satisfies condition \eqref{iBPDCA-inexcond}. Using this relation, we can readily obtain the desired results.
\end{proof}

%%%%%%%%%%%%%%%%%%%%%%%%%%%%%%%%%%%%%%%%%%%%%%%%%%%%%%%%%%%%%%%%%
%% References
\bibliographystyle{plain}
\bibliography{references/Ref_iBPDCA}

\end{document}